%% file: main.tex
\documentclass{article}

\usepackage[margin=1in]{geometry}
\usepackage{fullpage}
\usepackage[american]{babel}

\usepackage{amsmath, amsthm, amssymb, amsfonts, mathtools}
\usepackage{bm}

\input{macros.tex}

\usepackage{graphicx}
\usepackage{xcolor}
\definecolor{CeruleanRef}{RGB}{12,127,172}
\graphicspath{{./Figures/}}

\usepackage[ruled,vlined,algo2e,linesnumbered]{algorithm2e}
\usepackage[section]{algorithm}
\usepackage{algpseudocode}

\usepackage[colorlinks,linkcolor=blue, citecolor=blue]{hyperref}

\usepackage[capitalise]{cleveref}

\newtheorem{theorem}{Theorem}[section]
\newtheorem{lemma}[theorem]{Lemma}

\newtheorem{assumption}{Assumption}[section]

\newtheorem{remark}{Remark}[section]
\newtheorem{definition}{Definition}[section]

\crefname{equation}{expression}{expressions}
\crefname{algocfline}{algorithm}{algorithms}
\crefname{assumption}{assumption}{assumptions}
\crefname{condition}{condition}{conditions}
\crefname{proposition}{Proposition}{Propositions} %

\crefname{lemma}{Lemma}{Lemmas}
\Crefname{lemma}{Lemma}{Lemmas}
\crefname{corollary}{Corollary}{Corollaries}
\Crefname{corollary}{Corollary}{Corollaries}

\usepackage{multirow}
\usepackage{booktabs}
\usepackage{enumitem}
\usepackage{subcaption}
\usepackage[T1]{fontenc}
\usepackage{makecell}
\usepackage{comment} %
\usepackage{mathrsfs}
\setcellgapes{3pt}  %
\makegapedcells
\numberwithin{equation}{section}

\newcommand{\papertitle}{Stochastic Augmented Lagrangian Framework with Second-Order Convergence Guarantees for Nonconvex Expectation-Constrained Optimization}
\newcommand{\authorA}{Raghu Bollapragada}
\newcommand{\authorB}{Yash Kumar}
\newcommand{\affiliation}{Operations Research and Industrial Engineering Graduate Program, The University of Texas at Austin}
\newcommand{\emails}{\{raghu.bollapragada, yashkumar1803\}@utexas.edu}

\begin{document}

\title{\textbf{\papertitle}}

\author{
    \authorA\thanks{\affiliation. Emails: \texttt{\emails}} 
    \and \authorB\footnotemark[1] 
}

\date{\today}
\maketitle

\input{0-abstract}

\input{1-intro-lit}

\input{2-outer-algorithm}

\input{3-inner-algorithm}
\input{4-existing-solvers}

\input{4-SGD-NCS-RVR}

\input{4-Natasha2}
\input{4-SPIDER}
\input{5-numerical-experiments}

\input{6-conclusion}

\bibliographystyle{plain}
\bibliography{references}

\appendix
\input{7-appendix}

\end{document}

%% file: macros.tex
\providecommand{\cF}{{\mathcal{F}}}
\providecommand{\cG}{{\mathcal{G}}}

\providecommand{\cL}{{\mathcal{L}}}
\providecommand{\cM}{{\mathcal{M}}}

\providecommand{\cO}{{\mathcal{O}}}
\providecommand{\cP}{{\mathcal{P}}}

\providecommand{\cS}{{\mathcal{S}}}
\providecommand{\cT}{{\mathcal{T}}}

\providecommand{\cW}{{\mathcal{W}}}
\providecommand{\cX}{{\mathcal{X}}}

\providecommand{\EE}{\mathbb{E}} %

\providecommand{\NN}{\mathbb{N}} %
\providecommand{\PP}{{\mathbb{P}}}
\providecommand{\RR}{\mathbb{R}} %

\providecommand{\lt}{\left}
\providecommand{\rt}{\right}

\providecommand{\lambdaj}{\lambda^{(j)}} %
\providecommand{\zetaone}{\zeta^{(1)}}
\providecommand{\zetatwo}{\zeta^{(2)}}
\providecommand{\epsilong}{\epsilon^{(g)}}
\providecommand{\epsilonH}{\epsilon^{(H)}}
\providecommand{\ssum}{\textstyle\sum}
\providecommand{\pprod}{\textstyle\prod}

\providecommand{\simiid}{\stackrel{\mathrm{i.i.d.}}{\sim}}

\newcommand{\ignore}[1]{}

%% file: 0-abstract.tex
\begin{abstract}

In this paper, we propose and analyze an augmented Lagrangian framework for solving stochastic nonconvex optimization problems with expectation-based equality constraints over a closed and convex constraint set. The framework generates a sequence of nonconvex primal subproblems, which are solved inexactly using stochastic second-order methods. We establish iteration complexity results for obtaining approximate second-order stationary points, both in expectation and with prescribed probability, under corresponding conditions on the accuracy of the primal subproblem solutions. We further establish sample complexity results for the framework with general stochastic second-order subproblem solvers under reasonable assumptions on their theoretical guarantees.  Moreover, we incorporate three existing stochastic second-order solvers and derive the corresponding deterministic and probabilistic sample complexity results for obtaining approximate second-order stationary points.  To the best of our knowledge, such sample complexity guarantees for obtaining approximate second-order stationary points in nonconvex expectation-constrained optimization have not been established previously. Finally, we demonstrate the empirical performance of the proposed framework on two nonconvex machine learning problems.

\end{abstract}

%% file: 1-intro-lit.tex
\section{Introduction}

In this paper, we consider %
expectation-constrained stochastic optimization problems of the form
\begin{equation}\label{eq:gen_formula}
\begin{aligned}
    \min_{x \in \cX} \quad  f(x) & := \EE_\xi[F(x, \xi)]  \\
    \text{s.t.} \quad c(x) & := \EE_\zeta[C(x, \zeta)] = 0,
\end{aligned}
\end{equation}
where $f: \RR^d \rightarrow \RR$ and $c: \RR^d \rightarrow \RR^m$ are smooth, twice continuously differentiable and possibly nonconvex; %
$\xi$ and $\zeta$ are random variables with associated probability spaces $(\cS_\xi, \cG_\xi, \cP_\xi)$ and $(\cS_\zeta, \cG_\zeta, \cP_\zeta)$ respectively;
The stochastic functions 
$F: \RR^d \times \cS_\xi \rightarrow \RR$ and $C: \RR^d \times \cS_\zeta \rightarrow \RR^m$ are the stochastic objective and constraint functions; and $\EE_\xi[\cdot]$ and $\EE_\zeta[\cdot]$ denote expectations with respect to $\cP_\xi$ and $\cP_\zeta$, respectively. We assume that the constraint set $\cX \subseteq \RR^d$ is closed and convex %
and that only stochastic estimates of the objective function, constraint function, and their derivatives are available. %
Although \eqref{eq:gen_formula} contains only equality constraints, the formulation can accommodate inequality constraints. %
In particular, inequality constraints can be reformulated as equality constraints by introducing slack variables and incorporating their nonnegativity constraints into the set $\cX$. Problems of the form \eqref{eq:gen_formula} arise
in a range of applications, including fairness-constrained machine learning \cite{donini_empirical_2018,oneto_fairness_2020}, physics-informed neural networks \cite{basir_physics_2022,krishnapriyan_characterizing_2021}, robotics \cite{lavalle_planning_2006,lengagne_generation_2013}, and energy systems \cite{dincer_optimization_2017,wood_power_1996}.

Deterministic methods that use exact evaluations of functions and their derivatives to solve \eqref{eq:gen_formula} have been extensively studied;
see, e.g., \cite{nocedal2006numerical}. Nevertheless, nonconvex constrained optimization remains challenging because the problem landscape may contain multiple local minimizers/maximizers, and saddle points. Unlike in convex optimization, an approximate first-order stationary point of a nonconvex problem need not correspond to an approximate local minimizer. Therefore, algorithms require second-order (curvature) information to {distinguish between different stationary points.} 

These challenges are further exacerbated in the stochastic setting, where only stochastic estimates of the objective, constraint functions, and their derivatives are available, and errors in these estimators can adversely affect the optimization process. A few stochastic frameworks have been developed for expectation-constrained problems, including proximal-point methods \cite{boob_stochastic_2023,ma_quadratically_2020}, subgradient-switching methods \cite{lan_algorithms_2020,huang_oracle_2023}, sequential quadratic programming methods \cite{shen_sequential_2025,facchinei_stochastic_2025}, and augmented Lagrangian or penalty-based methods \cite{alacaoglu_complexity_2024,li_stochastic_2024,liu_spidertype_2026,xiao_developing_2026,zhang_solving_2022}. However, in nonconvex settings, these methods primarily establish convergence to only approximate first-order stationary points.

We propose and analyze stochastic algorithms for obtaining approximate second-order stationary points of \eqref{eq:gen_formula}. We develop an inexact stochastic augmented Lagrangian framework that solves a sequence of nonconvex subproblems over the convex set $\cX$, with penalty parameters and Lagrange multipliers updated between subproblem solves. Each subproblem is solved inexactly to satisfy either expectation-based or probability-based second-order inexactness conditions. We establish theoretical guarantees and evaluate the empirical performance of the framework using three stochastic second-order subproblem solvers. The proposed framework requires only stochastic estimates of the objective and constraint functions, their gradients, and Hessian vector products, thereby enabling the use of curvature information without explicitly forming or storing Hessian matrices.

\subsection{Notation and Preliminary Definitions}

We first introduce notation and preliminary definitions used throughout the paper. We denote the set of positive integers by $\NN := \{1, 2, 3, \dots, \}$, the set of non-negative integers by $\NN_0 := \{0, 1, 2, \dots, \}$, the set of real numbers by $\RR$, and the set of non-negative real numbers by $\RR_+$. The sets of $d$-dimensional real vectors and $d\times m$ real matrices are denoted by $\RR^d$ and $\RR^{d\times m}$, respectively. For a vector $a$, $a^{(j)}$ denotes its $j$th component.
For each realization of $\zeta$,  $C(\cdot,\zeta):\RR^d\to\RR^m$, with $C^{(j)}(\cdot,\zeta)$ denoting its $j$th component, we use {$\nabla C^{(j)}(x,\zeta)\in\RR^d$ to denote the gradient of
$C^{(j)}(\cdot,\zeta)$ at $x$, and} $\nabla^2 C^{(j)}(x,\zeta)\in\RR^{d\times d}$ to denote the Hessian of
$C^{(j)}(\cdot,\zeta)$ at $x$. Likewise, $c^{(j)}$ denotes the $j$th component of $c$, {$\nabla c^{(j)}(x) \in \RR^d$ denotes the gradient of $c^{(j)}(\cdot)$ at $x$,} and $\nabla^2 c^{(j)}(x)\in\RR^{d\times d}$ denotes its Hessian, for
each $j\in\{1,\ldots,m\}$. For a matrix $A$, $\mathrm{Null}(A):=\{x\mid Ax=0\}$ denotes its null space. For a real symmetric matrix $A$, $\sigma_{\min}(A)$ denotes its smallest eigenvalue. Unless indicated otherwise in the subscripts, $\|\cdot\|$ denotes the $\ell_2$  vector norm for vectors and the vector-induced $\ell_2$ norm for matrices. We use $\mathrm{proj}_\cX (x)$ and $\mathrm{dist}(x, \cX)$ to denote the projection
of $x$ onto $\cX$ and the distance from $x$ to $\cX$, respectively. That is, 
\begin{align}
  \mathrm{proj}_\cX (x) := \arg \min_{x' \in \cX} \|x - x'\|, \quad \mathrm{dist}(x, \cX) := \min_{x' \in \cX} \|x - x'\|.    
\end{align}
Furthermore, $T_{\cX}(x)$ and $N_{\cX}(x)$ denote the tangent and normal cones to $\cX$ at $x$, respectively.  The critical tangent cone is defined by 
\begin{align}
    T_\cX^{(c)}(x) := T_{\cX}(x) \cap {\mathrm{Null}(\nabla c(x))}.
\end{align}

For an event $E$, $\PP(E)$ denotes its probability and $\mathbf{1}_A$ denotes its indicator variable. In particular, $\mathbf{1}_A=1$ when $E$ occurs and $\mathbf{1}_A=0$ otherwise. The notation $Var_\zeta(C(x, \zeta))$ represents 
{the trace of the covariance matrix of stochastic vector function $C(x, \zeta)$ with respect to random variable $\zeta$. }
Finally, we use the standard asymptotic notation $\cO(\cdot)$, $\Omega(\cdot)$, and $\Theta(\cdot)$. In particular, $h(n)\in\cO(g(n))$ means that $h(n)$ is bounded above by a positive constant multiple of $g(n)$ for all sufficiently large $n$, while $h(n)\in\Omega(g(n))$ gives the corresponding lower bound. The notation $h(n)\in\Theta(g(n))$ means that both bounds hold, implying that both functions share same asymptotic growth rates. Additionaly, we write  $h(n) \in \tilde{\cO}(g(n))$ if $h(n) \in \cO(g(n)\log^r g(n))$ for some constant $r \ge 0$. That is, $\tilde{\cO}$ suppresses polylogarithmic factors.

With this notation, we say that a point $x \in \cX$ is a second-order stationary point of \eqref{eq:gen_formula} if there exists $\lambda\in\RR^m$ such that
\begin{align*}
    & \| \mathrm{proj}_{T_{\cX}(x)} (-\nabla f(x) - \nabla c(x)^\top \lambda) \| = 0, \qquad \| c(x) \| = 0, \text{ and } \\
    & \min_{u \in T_\cX^{(c)}(x), \|u\|=1} u^\top \big( \nabla^2 f(x) + \ssum_{j=1}^{m} \lambda^{(j)} \nabla^2 c^{(j)}(x) \big) u \geq 0.
\end{align*}
More generally, for $\epsilong>0$ and $\epsilonH>0$, we say that $x\in\cX$ is an $(\epsilong,\epsilonH)$-accurate second-order stationary point if there exists $\lambda\in\RR^m$ such that
\begin{equation} \label{eq:eps-kkt}
\setlength{\arraycolsep}{2pt}
\begin{array}{cll}
    \mathrm{(i)} & \quad \| \mathrm{proj}_{T_{\cX}(x)} (-\nabla f(x) - \nabla c(x)^\top \lambda) \| \leq \epsilong, &\quad \text{($1^{\mathrm{st}}$-order stationarity condition)} \\
    \mathrm{(ii)} & \quad \| c(x) \| \leq \epsilong, &\quad \text{(feasibility condition)} \\
    \mathrm{(iii)} & \quad \min\limits_{u \in T_\cX^{(c)}(x), \|u\|=1} u^\top \big( \nabla^2 f(x) + \ssum_{j=1}^{m} \lambda^{(j)} \nabla^2 c^{(j)}(x) \big) u \geq -\epsilonH, &\quad \text{($2^{\mathrm{nd}}$-order stationarity condition)} 
\end{array}
\end{equation}
for some $\lambda \in \RR^m$, $\epsilong > 0$, and $\epsilonH > 0$; see e.g., \cite{mokhtari_escaping_2018, royer_newtoncg_2020, sahin_inexact_2019}.

\begin{remark}
    The first-order stationarity condition in \eqref{eq:eps-kkt}$\mathrm{(i)}$ is %
    equivalent to to the distance-based condition%
    \begin{align*}
        \mathrm{dist}(\nabla f(x) + \nabla c(x)^\top \lambda, -N_\cX(x)) \leq \epsilong,
    \end{align*}
    which is the form commonly used in literature; see, e.g., \cite{alacaoglu_complexity_2024,li_stochastic_2024,sahin_inexact_2019}.
    However, we use the projection-based form in \eqref{eq:eps-kkt}$\mathrm{(i)}$ because it simplified our analysis.
    {To our knowledge, %
    {theoretical analyses based on first-order stationarity conditions in this projection form are limited.} This condition makes explicit use of projections onto the tangent cone $T_{\mathcal{X}}(x)$. For a detailed discussion of the properties of the projection operator $\mathrm{proj}_{T_{\mathcal{X}}}(\cdot)$, see Remark~\ref{rmk:proj-tan-cone}.}

\end{remark}

In the stochastic settings, we may not be able to obtain approximate solutions that satisfy the deterministic stationarity conditions in \eqref{eq:eps-kkt}. Instead, we consider random iterates that satisfy these conditions either in expectation or with a prescribed probability. To this end, we define two types of approximate solutions as follows:
\begin{definition} \label{def:opt_cond_exp}
    A (possibly random) iterate $\tilde{x} \in \cX$ is an $(\epsilong, \epsilonH)$-accurate expected second-order stationary point of \eqref{eq:gen_formula}, if there exists $\tilde{\lambda} \in \RR^m$ such that, 
    \begin{equation} \label{eq:opt_cond_exp}
    \begin{aligned}
        \EE[\| \mathrm{proj}_{T_\cX(\tilde{x})} ( - \nabla f(\tilde{x}) - \nabla c(\tilde{x})^\top \tilde{\lambda} ) \|] &\leq \epsilong ,\\
        \EE[\|c(\tilde{x}) \|] &\leq \epsilong , \\
        \EE \lt[ \min\nolimits_{u \in T_\cX^{(c)}(\tilde{x}), \|u\|=1} u^\top (\nabla^2 f(\tilde{x}) + \ssum_{j=1}^m \tilde{\lambda}^{(j)} \nabla^2 c^{(j)}(\tilde{x})) u \rt]  &\geq - \epsilonH. 
    \end{aligned}
    \end{equation}
\end{definition}
\begin{definition} \label{def:opt_cond_prob}
    A (possibly random) iterate $\tilde{x} \in \cX$ is an $(\epsilong, \epsilonH)$-accurate $p$-probabilistic second-order stationary point of \eqref{eq:gen_formula}, if there exists $\tilde{\lambda} \in \RR^m$ such that,
    \begin{align}
        \PP
        \begin{pmatrix}
        \{ \| \mathrm{proj}_{T_\cX(\tilde{x})} ( - \nabla f(\tilde{x}) - \nabla c(\tilde{x})^\top \tilde{\lambda} ) \| \leq \epsilong \}~\cap \\
        \{ \| c(\tilde{x}) \| \leq \epsilon^{(g)} \}~\cap \\
        \big\{ \min\nolimits_{u \in T_\cX^{(c)}(\tilde{x}), \|u\|=1} u^\top (\nabla^2 f(\tilde{x}) + \ssum_{j=1}^m \tilde{\lambda}^{(j)} \nabla^2 c^{(j)}(\tilde{x})) u \geq - \epsilonH \big\}
        \end{pmatrix} \geq p. \label{eq:opt_cond_prob}
    \end{align}
\end{definition}
We utilize algorithms based on stochastic {function,} gradient and Hessian vector product evaluations to obtain these approximate solutions. To quantify the computational effort required by these algorithms, we define the sample complexity of an algorithm as follows:
\begin{definition}\label{def:sample_complexity}
    The sample complexity of an algorithm, denoted by $\cW$, is the total number of stochastic objective gradient evaluations $\nabla F(x,\xi)$, stochastic objective Hessian vector product evaluations $\nabla^2 F(x,\xi)v$, stochastic constraint function evaluations $C^{(j)}(x,\zeta)$, stochastic constraint gradient evaluations $\nabla C^{(j)}(x,\zeta)$, and stochastic constraint Hessian vector product evaluations $\nabla^2 C^{(j)}(x,\zeta)v$, for $j\in{1,\ldots,m}$ and an arbitrary vector $v$, required to obtain either an $(\epsilong,\epsilonH)$-accurate second-order stationary point of \eqref{eq:gen_formula}.
\end{definition}
\subsection{Contributions}

The primary contributions of this paper are three-fold. 

\begin{enumerate}

    \item We propose and analyze a novel algorithmic framework, named MESCAL (Method for Expectation-constraints with Second-order Convergence via Augmented Lagrangian), based on the inexact augmented Lagrangian framework for solving nonconvex expectation-constrained optimization problems of the form \eqref{eq:gen_formula}. The framework involves solving a sequence of nonconvex subproblems over the convex constraint set $\cX$. We develop inexactness conditions in both expectation and probability for solving these subproblems inexactly, which can be satisfied by general second-order stochastic solvers. We establish iteration complexity bounds %
    for MESCAL to obtain an $(\epsilong,\epsilonH)$-accurate expected or $p$-probabilistic second-order stationary point within $\cO\lt(\log\lt((\epsilong)^{-1}\rt)\rt)$ iterations (see \Cref{thm:outer-iteration-complexity}).

    \item We establish sample complexity results for MESCAL when general second-order solvers are used to solve the subproblems inexactly under reasonable assumptions on their theoretical properties, motivated by the lower bounds established in \cite{arjevani_secondorder_2020}, both in expectation and in probability (see \Cref{thm:theoretical-solver}). Furthermore, when $\cX=\RR^d$, we use three second-order solvers, SG-HV-NC \cite[Algorithm 4]{arjevani_secondorder_2020}, Natasha2 \cite{allen-zhu_natasha_2018}, and SPIDER \cite{fang_spider_2018}, within MESCAL to establish deterministic and probabilistic sample complexity bounds for obtaining an $(\epsilong,\epsilonH)$-accurate second-order stationary point in probability. \Cref{table:result_comparison} {contextualizes our results against prior work.}
    {To the best of our knowledge, such second-order stationarity guarantees for expectation-constrained problems of the form \eqref{eq:gen_formula} have not previously appeared in the literature.}

    \item We demonstrate the empirical performance of the proposed algorithmic framework on two relevant machine learning tasks: (i) a fairness-constrained optimization problem and (ii) a nonconvex Neyman--Pearson classification problem. We compare the performance of the three second-order solvers when used within the proposed MESCAL framework on these problems.

\end{enumerate}

\begin{table}[H]
    \centering
    \def\arraystretch{0.8}
    \caption{Summary of sample complexity $\cW$ results in the relevant literature under different problem settings. In all the works mentioned here, the objective considered is nonconvex, stochastic and the constraint set $\cX = \RR^d$. There are two types of constraints referred to here: 1) $\EE$ refers to expectation constraints, 2) None means no constraints. There are two types of iterates referred to here: 1) $\EE$ refers to random iterates satisfying stationarity conditions in \eqref{eq:opt_cond_exp}, 2) $\PP$ refers to random iterates satisfying stationarity conditions in \eqref{eq:opt_cond_prob}. Here, first-order sample complexity refers to work bounds in achieving first-order stationarity conditions and feasibility conditions, {and} second-order sample complexity refers to work bounds in achieving second-order stationarity conditions and feasibility conditions, if constraints present. {The sample complexity bounds denoted in our results are deterministic. We further prove sample complexity bounds in probability that improve upon these results by a factor of $\cO((\epsilong)^{-1})$; {see} %
    \Cref{sec:active_solvers}.} (Note: $^*$For marked paper, $\epsilong = \epsilon, \epsilonH = \sqrt{L_2 \epsilon}$, where $L_2$ is Lipschitz constant corresponding to Hessian of the objective function.)}
    \label{table:result_comparison}
    \begin{tabular}{l|c|c|c|c}
        \textbf{Reference} &
        \textbf{\makecell{Constraints}} & 
        \textbf{\makecell{Iterate\\Type}} & 
        \textbf{\makecell{1$^{\boldsymbol{\mathrm{st}}}$-Order Sample\\Complexity}} &
        \textbf{\makecell{2$^{\boldsymbol{\mathrm{nd}}}$-Order Sample\\Complexity}} \\
        \hline \hline
        ConEx \cite{boob_stochastic_2023}
        & $\EE$ 
        & $\EE$
        & $\cO \lt( (\epsilong)^{-6} \rt)$
        & - \\
        Stoch-iALM \cite{li_stochastic_2024}
        & $\EE$
        & $\EE$
        & $\cO \lt( (\epsilong)^{-5} \rt)$
        & - \\
        Stoch-SQP \cite{shen_sequential_2025}
        & $\EE$ 
        & $\EE$
        & $\cO \lt( (\epsilong)^{-4} \rt)$
        & - \\
        \hline
        SCR \cite{tripuraneni_stochastic_2018}
        & None 
        & $\PP$
        & - 
        & $\cO \lt( \epsilon^{-3.5} \rt)^*$ \\
        Natasha2 \cite{allen-zhu_natasha_2018}
        & None 
        & $\PP$
        & - 
        & \makecell{$\tilde{\cO}\big( (\epsilong)^{-3} (\epsilonH)^{-1}$\\$+(\epsilonH)^{-5} \big)$} \\
        SPIDER \cite{fang_spider_2018}
        & None 
        & $\PP$
        & -
        & $\tilde{\cO}\lt( (\epsilong)^{-3} + (\epsilonH)^{-5} \rt)$ \\
        SG-HV-NC \cite{arjevani_secondorder_2020}
        & None 
        & $\PP$
        & -
        & $\tilde{\cO}\lt( (\epsilong)^{-3} + (\epsilonH)^{-5} \rt)$ \\
        \hline
        \textbf{\makecell{MESCAL-\\SG-HV-NC\\
        (\Cref{thm:mescal_sgd})\\\vspace{-0.5em}\\MESCAL-\\SPIDER\\
        (\Cref{thm:mescal_spider})}}
        & $\EE$ 
        & $\PP$
        & -
        & \makecell{
            $\boldsymbol{\tilde{\cO} \big( (\epsilong)^{-6} (\epsilonH)^{-2}}$\\+ $\boldsymbol{(\epsilong)^{-5} (\epsilonH)^{-5} \big)}$} \\

        \textbf{\makecell{MESCAL-\\Natasha2\\
        (\Cref{thm:mescal_natasha2})}}
        & $\EE$  
        & $\PP$
        & -
        & \makecell{
            $\boldsymbol{\tilde{\cO} \big( (\epsilong)^{-7} (\epsilonH)^{-1}}$\\+ $\boldsymbol{(\epsilong)^{-5} (\epsilonH)^{-5} \big)}$} \\
    \end{tabular}
\end{table}

\subsection{Literature Review}

Deterministic augmented Lagrangian {(AL)} methods date back to the independent works of Hestenes \cite{hestenes_multiplier_1969} and Powell \cite{powell_method_1969} in 1969. The framework was subsequently developed and analyzed in, among others, the works of Rockafellar \cite{rockafellar_dual_1973} in 1973 and Bertsekas \cite{bertsekas_penalty_1976} in 1975. Our framework for expectation-constrained optimization is based in part on a stochastic version of the deterministic framework established by Bertsekas in \cite{bertsekas_penalty_1976}. Since then, the framework has been employed extensively in the optimization literature, and several methods have established complexity results under deterministic settings \cite{li_inexact_2021, lin_complexity_2022, xu_iteration_2021}. The literature most closely related to our work can be broadly divided into two categories: methods for expectation-constrained optimization and second-order methods for unconstrained nonconvex optimization.

\paragraph{Expectation-Constrained Optimization.} The development of algorithms for expectation-constrained optimization remains an active area of research \cite{boob_stochastic_2023, facchinei_stochastic_2025, huang_oracle_2023, lan_algorithms_2020, lew_sample_2024, ma_quadratically_2020, menhorn_trustregion_2022, yang_datadriven_2025, yu_online_2017}, with several works using AL and penalty-based methods \cite{alacaoglu_complexity_2024, li_stochastic_2024, liu_spidertype_2026, xiao_developing_2026, zhang_solving_2022}. For convex expectation-constrained problems, existing methods include primal subgradient switching \cite{lan_algorithms_2020}, primal subgradient descent with dual ascent \cite{yan_adaptive_2022}, and linearized proximal methods of multipliers \cite{zhang_solving_2022}, with sample complexity bounds of $\cO((\epsilong)^{-2})$ under their respective assumptions.

Existing work for nonconvex expectation-constrained optimization is more extensive but focuses primarily on first-order methods. Boob et al.~\cite{boob_stochastic_2023} employ a proximal-point method to transform the nonconvex problem into a sequence of convex subproblems. Under stronger assumptions, including strong feasibility, they establish an $\cO((\epsilong)^{-6})$ sample complexity bound. Ma et al.~\cite{ma_quadratically_2020} obtain the same bound using a related proximal-point approach based on the online stochastic subgradient routine of Yu et al.~\cite{yu_online_2017}. Huang et al.~\cite{huang_oracle_2023} also establish an $\cO((\epsilong)^{-6})$ bound using a switching-subgradient method related to that of Lan et al.~\cite{lan_algorithms_2020}. Beyond AL and penalty-based approaches, stochastic sequential quadratic programming methods have recently been considered for nonconvex expectation-constrained problems \cite{facchinei_stochastic_2025,shen_sequential_2025}; in particular, Shen et al.~\cite{shen_sequential_2025} establish an $\cO((\epsilong)^{-4})$ sample complexity bound under linear independence constraint qualification (LICQ).

Within AL and penalty-based first-order methods, Li et al.~\cite{li_stochastic_2024} provided an early result for nonconvex expectation-constrained optimization, establishing an $\cO((\epsilong)^{-5})$ sample complexity bound. Their method employs a proximal variant of the STORM estimator \cite{cutkosky_momentumbased_2019} to solve the primal subproblems. Subsequently, Alacaoglu and Wright~\cite{alacaoglu_complexity_2024} developed a single-loop penalty-based method using the STORM estimator and established the same sample complexity bound. Some of the conditions underlying our framework are motivated by \cite{alacaoglu_complexity_2024, li_stochastic_2024}. For weakly convex problems, Liu et al.~\cite{liu_spidertype_2026} developed a single-loop method based on an exact penalty formulation. Xiao et al.~\cite{xiao_developing_2026} developed an AL-based framework which embeds generalized subgradient methods for a single-step primal variable update.
\paragraph{Unconstrained Second-order Methods.}
Second-order methods are typically used to obtain approximate second-order stationary points. To the best of our knowledge, no existing work establishes second-order sample complexity guarantees for nonconvex expectation-constrained optimization. However, second-order methods have been extensively studied for unconstrained nonconvex optimization. In the deterministic setting, Cartis et al. \cite{cartis_complexity_2012} established a lower bound of $\cO(\max((\epsilong)^{-2}, (\epsilonH)^{-3}))$ evaluations for finding an $(\epsilong, \epsilonH)$-accurate second-order stationary point under reasonable assumptions, 
which is attained by cubic regularization methods \cite{cartisAdaptiveCubicRegularisation2011b}. Newton-type methods \cite{xu_newtontype_2020} require $\cO(\max((\epsilong)^{-2}(\epsilonH)^{-1},(\epsilonH)^{-3}))$ evaluations to obtain an $(\epsilong,\epsilonH)$-accurate second-order stationary point.

Second-order methods have also been developed for stochastic optimization problems \cite{allen-zhu_natasha_2018, arjevani_secondorder_2020, fang_spider_2018, tripuraneni_stochastic_2018}, with their sample complexity bounds summarized in \Cref{table:result_comparison}. These are typically stochastic extensions of existing deterministic second-order methods. In \cite{tripuraneni_stochastic_2018}, a stochastic cubic regularized Newton method is proposed with a sample complexity bound of $\cO(\epsilon^{-3.5})$ for obtaining an $(\epsilon,\sqrt{L\epsilon})$-accurate point, where $L$ is a Lipschitz constant. The negative curvature-based approaches in \cite{allen-zhu_natasha_2018, arjevani_secondorder_2020, fang_spider_2018} achieve a similar sample complexity bound of $\tilde{\cO}(\epsilon^{-3.5})$ under the same accuracy requirements. However, when the first-order and second-order tolerance parameters are of the same order, i.e., $\cO(\epsilon_g)=\cO(\epsilon_H)=\cO(\epsilon)$, the sample complexity bound for the stochastic cubic regularized Newton method becomes $\cO(\epsilon^{-7})$, whereas the negative curvature-based approaches achieve a bound of $\tilde{\cO}(\epsilon^{-5})$.

In this work, we consider \cite{allen-zhu_natasha_2018, arjevani_secondorder_2020, fang_spider_2018} as subproblem solvers within the MESCAL framework due to their favorable sample complexity bounds. For ease of comparison, \Cref{table:result_comparison} suppresses the dependence of these bounds on Lipschitz constants, variance bounds, and the initial optimality gap. However, these quantities play an important role in the sample complexity analysis of MESCAL; see the theoretical results in \Cref{sec:inner_loop} and \Cref{sec:active_solvers}.

\subsection{Paper Organization}

The paper is organized as follows. In \Cref{sec:outer_loop}, we present the MESCAL framework and establish iteration complexity guarantees for obtaining an $(\epsilong, \epsilonH)$-accurate expected or $p$-probabilistic second-order stationary point. In \Cref{sec:inner_loop}, we establish sample complexity results for MESCAL with a general solver to solve the subproblems arising within MESCAL.  In \Cref{sec:active_solvers}, under the restriction \(\cX=\RR^d\), we adapt three existing stochastic second-order solvers for use within MESCAL and establish their corresponding sample complexity results. In \Cref{sec:num_expts}, we present numerical results on two nonconvex optimization problems using MESCAL with these solvers. Finally, we provide concluding remarks in \Cref{sec:conclusion}.

%% file: 2-outer-algorithm.tex
\section{Algorithmic Framework} \label{sec:outer_loop} 
In this section, we propose an algorithmic framework for solving \eqref{eq:gen_formula}. The proposed framework, named MESCAL (Method for Expectation-constraints with Second-order Convergence via Augmented Lagrangian), is based on augmented Lagrangian methods. These methods rely on the property of \textit{augmentability}, which states that if a point $x^* \in \cX$ is a local solution of \eqref{eq:gen_formula}, then there exist some $\lambda^* \in \RR^m$ and a scalar $\alpha > 0$ such that $x^*$ is a local minimizer of the augmented Lagrangian problem $\min_{x \in \cX} \cL_\alpha(x, \lambda^*) := f(x) + (\lambda^*)^\top c(x) + \tfrac{\alpha}{2} c(x)^2$ \cite{hestenes_optimization_1975}. A typical augmented Lagrangian method iteratively updates estimates of the primal-dual solution $(x^*, \lambda^*)$ through the following two steps:
\begin{enumerate}[leftmargin=1.5cm]
    \item[Step 1:] Solve the primal subproblem by finding
    \begin{align}\label{eq:subprob}
        x_k \in \arg \min_{x \in \cX} \, \cL_{\alpha_k}(x, \lambda_k) := f(x) + \lambda_k^\top c(x) + \tfrac{\alpha_k}{2} \|c(x)\|^2,
    \end{align}
    where $\alpha_k > 0$ is the penalty parameter and $\lambda_k$ is the dual variable at iteration $k$. The function $\cL_{\alpha_k}$ is referred to as the augmented Lagrangian function.
    \item[Step 2:] Update the dual variable according to
    \begin{align*}
        \lambda_{k+1} = \lambda_k + \alpha_k c(x_k).
    \end{align*}
\end{enumerate}
However, implementing these steps in the nonconvex stochastic optimization setting considered in this paper poses significant challenges. Due to nonconvexity, finding an exact global solution to the primal subproblem is generally intractable. Moreover, even finding an exact local minimizer of the primal subproblem is challenging, as iterative nonlinear optimization algorithms typically require access to function and derivative evaluations of the objective function $f$ and constraint function $c$, both of which involve expectations of stochastic functions. While inexact augmented Lagrangian methods, which solve the primal subproblems only approximately up to a near local minimizer, are well established in the literature \cite{li_inexact_2021, xu_iteration_2021}, they typically impose deterministic inexactness conditions. Such conditions are generally impractical in stochastic settings. To overcome this limitation, we propose solving the primal subproblems inexactly while imposing inexactness conditions in expectation or in probability, rather than deterministically. 

Furthermore, the dual update step requires evaluating the exact constraint function $c(x_k)$ and may result in unbounded dual variables in nonconvex settings. We address this challenge by employing subsampled approximations of $c(x_k)$ and modifying the dual update using an upper-bounding sequence $\{\gamma_k\}$. For notational simplicity, unless otherwise specified, $\nabla$ and $\nabla^2$ applied to the augmented Lagrangian $\cL_\alpha(x,\lambda)${, and its estimators in forthcoming sections,} denote differentiation with respect to the primal variable $x$. In particular, $\nabla \cL_\alpha(x,\lambda):=\nabla_x\cL_\alpha(x,\lambda)$ and $\nabla^2 \cL_\alpha(x,\lambda):=\nabla_{xx}^2\cL_\alpha(x,\lambda)$. 
We also introduce notation for expectations and probabilities that will be used throughout the paper. Unless otherwise specified, let
\begin{align}
    \EE[\cdot] :&= \EE[\cdot\mid x_{-1},\lambda_0],
    \qquad
    \PP(\cdot) := \PP(\cdot\mid x_{-1},\lambda_0),
\end{align}
where the conditioning is with respect to the initial primal dual iterates and all randomness generated by the algorithm thereafter. At the beginning of iteration $k$, let $\mathcal{F}_k$ denote the $\sigma$-algebra generated by all the randomness available up to the point at which the primal dual iterates $(x_{k-1},\lambda_k)$ are obtained. We then use
\begin{align}
    \EE[\cdot\mid\mathcal{F}_k]
    \qquad\text{and}\qquad
    \PP(\cdot\mid\mathcal{F}_k),
\end{align}
to denote the conditional expectation and probability given the history of the algorithm up to the beginning of iteration $k$.

{The step-by-step overview of the MESCAL framework is presented in \Cref{alg:mescal-overview}.}

\begin{algorithm2e}[H]
\caption{Method for Expectation-constraints with Second-order Convergence via Augmented Lagrangian (MESCAL)}
\label{alg:mescal-overview}
\SetKwInOut{Input}{Input}
\DontPrintSemicolon
\Input{Initial iterate $x_{-1} \in \cX$; initial dual variable $\lambda_0 \in \RR^m$; initial penalty parameter $\alpha_0 > 0$; penalty parameter increase factor $\beta > 1$; primal subproblem solver $\cM$ satisfying either Option I (expectation conditions) or Option II (probabilistic conditions); primal subproblem error tolerance sequences $\epsilong_k, \epsilonH_k > 0$; success probability sequences $\{p_k\}$ for Option II, constraint sample size sequence $\{|S^{(d)}_k|\}$; and dual upper bound sequence $\{\gamma_k\}$}
\For{$k = 0,1, \dots$}{
    Update penalty parameter: $\alpha_k = \alpha_0 \, \beta^k$ \label{step:penalty_update}\;
    Starting from $x_{k-1}$, obtain primal iterate $x_k$ by applying  solver $\cM$ to solve subproblem \eqref{eq:subprob} satisfying either Option I or Option II below:
    \label{step:primal_prob} \;
    \textbf{Option I (expectation conditions):}%
    \vspace{-1em}
    \begin{equation}
    \begin{aligned}
        \label{cond:exp_primal_prob}
        & \EE[ \|\mathrm{proj}_{T_\cX(x_k)} (-\nabla \cL_{\alpha_k}(x_k, \lambda_k))\| \mid \cF_k ] \le \epsilong_k, \\
        & \EE\lt[
        \min\nolimits_{u \in T_\cX(x_k), \|u\|=1}
        u^\top \nabla^2 \cL_{\alpha_k}(x_k,\lambda_k) u
        \mid
        \cF_k
        \rt]
        \geq -\epsilonH_k.
    \end{aligned}
    \end{equation} \label{step:expect_primal_prob} \;
    \vspace{-1.5em}
    \textbf{Option II (probabilistic conditions):}%
    \vspace{-1em}
    \begin{equation}
    \begin{aligned}
        \label{cond:prob_primal_prob}
        \PP \lt( 
        \begin{matrix}
            \big\{ \| \mathrm{proj}_{T_\cX(x_k)} (- \nabla \cL_{\alpha_k} (x_k, \lambda_k)) \| \leq \epsilong_k \big\} ~\cap \\
            \big\{ \min\nolimits_{u \in T_\cX(x_k), \|u\|=1} u^\top \big( \nabla^2 \cL_{\alpha_k} (x_k, \lambda_k) \big) u \geq - \epsilonH_k \big\}
        \end{matrix} \middle| \cF_k \rt) \geq p_k.
    \end{aligned}
    \end{equation} \label{step:prob_primal_prob} \;
    \vspace{-2em}
    Draw $|S^{(d)}_k|$ i.i.d. samples $\{\zeta_i\}_{i=1}^{S^{(d)}_k}$ and compute stochastic constraint estimate
    \vspace{-0.5em}
    \begin{equation*}
        \overline{C}_{S^{(d)}_k}(x_k) = \tfrac{1}{|S^{(d)}_k|} \ssum_{i=1}^{|S^{(d)}_k|} C(x_k, \zeta_i)
    \end{equation*}\;
    \vspace{-2em}
    Update dual variable: %
    \begin{equation} \label{eq:mescal-dual-update}
        \lambda_{k+1} = \begin{cases} \lambda_k & \text{if } \|\overline{C}_{S^{(d)}_k}(x_k)\| = 0 \\
        \lambda_k + \min\lt\{ \alpha_k, \tfrac{\gamma_k}{\lt\|\overline{C}_{S^{(d)}_k}(x_k) \rt\|} \rt\} \overline{C}_{S^{(d)}_k}(x_k) & \text{otherwise} 
        \end{cases}
    \end{equation} \label{step:dual-update}
    \vspace{-1em}
}
\end{algorithm2e}
\begin{remark}
    We make a few remarks about the MESCAL framework as presented in \Cref{alg:mescal-overview}:

    \begin{itemize}
        \item The framework follows the structure of a typical augmented Lagrangian method: the primal subproblem is solved inexactly in Line~\ref{step:primal_prob}, followed by the dual variable update in Line~\ref{step:dual-update}. %
        Although augmented Lagrangian methods traditionally offer stronger practical and theoretical advantages over quadratic penalty methods by incorporating dual updates to improve feasibility and conditioning of the subproblem, these benefits are not fully reflected in existing theoretical analysis of nonconvex problems. In particular, nonasymptotic guarantees for augmented Lagrangian variants with a fixed penalty parameter and dual step size \cite{xie_complexity_2021} exhibit worse iteration complexity than those relying on increasing penalty parameters and decaying dual step sizes \cite{sahin_inexact_2019}, as surveyed by Alacaoglu and Wright \cite{alacaoglu_complexity_2024}. Accordingly, we increase the penalty parameter $\alpha_k$ at each iteration in Line~\ref{step:penalty_update}. Consequently, we modify the dual update to ensure that the sequence of dual variables remains uniformly bounded by using an upper-bounding sequence ${\gamma_k}$.

        \item We impose second-order inexactness conditions for the primal subproblem at each iteration, either in expectation as in~\eqref{cond:exp_primal_prob} or in probability as in~\eqref{cond:prob_primal_prob}. These conditions ensure that each primal subproblem is solved to sufficient accuracy for the overall framework to produce an approximate second-order stationary point of~\eqref{eq:gen_formula}. They also enable us to establish iteration complexity results for obtaining such an iterate; see \Cref{thm:outer-iteration-complexity}. We note that these inexactness conditions can be satisfied by a (stochastic) second-order solver $\cM$. For example, under $\cX = \RR^d$, the method in~\cite{berahas_exploiting_2026} uses a condition analogous to~\eqref{cond:exp_primal_prob} to define second-order stationarity in expectation. Similarly, the methods in~\cite{allen-zhu_natasha_2018, arjevani_secondorder_2020, berahas_exploiting_2026, fang_spider_2018} use conditions analogous to~\eqref{cond:prob_primal_prob} to define second-order stationarity in probability.
        
        \item In \Cref{sec:active_solvers}, we discuss three well-known second-order solvers with established theoretical complexity guarantees for solving primal subproblems of the form~\eqref{eq:subprob} under $\cX = \RR^d$ while satisfying the inexactness conditions given in \Cref{alg:mescal-overview}. Furthermore, for each of these solvers, we establish theoretical sample complexity results for obtaining approximate second-order stationary points. We note that we were unable to identify second-order solvers with established sample complexity guarantees for the case $\cX \subset \RR^d$. Consequently, we restrict our sample complexity results to the case $\cX = \RR^d$.

    \end{itemize}
\end{remark}

Next, we state the assumptions on the objective and constraint functions required to establish the iteration complexity results for the proposed framework.

\begin{assumption} \label{assn:abs-true-bounds}
    The functions $f: \cX \rightarrow \RR$, and $c: \cX \rightarrow \RR^m$ are twice continuously differentiable, where $\cX \subseteq \RR^d$ is a closed and convex set. Moreover, there exist positive constants $\kappa_f$, $\kappa_{\nabla f}$, $\kappa_c$, $\kappa_{\nabla c}$ such that
    \begin{equation*}
        \begin{aligned}
            && |f(x)| \leq \kappa_{f}, && \|\nabla f(x)\| \leq \kappa_{\nabla f}, &&
            \|c(x)\| \leq \kappa_{c}, && \|\nabla c(x)\| \leq \kappa_{\nabla c}, && \|\nabla^2 c(x)\| \leq \kappa_{\nabla^2 c}, && \forall  x \in \cX.
        \end{aligned}
    \end{equation*}
\end{assumption}

\begin{assumption} \label{assn:reg-c-nabla-c}
    There exists a positive constant $C_{\text{reg}}$ such that
    \begin{align*}
        C_{\text{reg}} \|c(x) \| \leq
        \| \mathrm{proj}_{T_\cX (x)}(-\nabla c(x)^\top c(x)) \|, \quad \forall x \in \cX.
    \end{align*}
\end{assumption}

\begin{remark}
The boundedness conditions imposed in \Cref{assn:abs-true-bounds} are standard assumptions in nonconvex constrained optimization; see, for example, \cite{alacaoglu_complexity_2024, li_rateimproved_2021, lin_complexity_2022, sahin_inexact_2019}.  \Cref{assn:reg-c-nabla-c} is a regularity condition on the constraints, and similar assumptions are commonly employed in the literature; see \cite{alacaoglu_complexity_2024, li_rateimproved_2021, lin_complexity_2022, lu_singleloop_2022, sahin_inexact_2019, xiao_developing_2026}. In the case $\cX = \RR^d$, \Cref{assn:reg-c-nabla-c} implies that $\|\nabla c(x)^T c(x)\| \geq C_{\text{reg}}\|c(x)\|$, which ensures that $\nabla c(x)^\top c(x)$ cannot be arbitrarily small relative to $c(x)$. Such regularity conditions, including Slater's condition, LICQ, and MFCQ (Mangasarian-Fromovitz constraint qualification), are useful in establishing the convergence of iterative optimization algorithms to local solutions in nonconvex optimization settings. We emphasize that the constants appearing in these assumptions need not be known to establish our theoretical convergence guarantees or to implement the algorithm in practice; it is only required that such constants exist. We also note that an equivalent condition to \Cref{assn:reg-c-nabla-c}, which is commonly used in the literature, is
\begin{align*}
    C_{\text{reg}} \|c(x) \| \leq
    \mathrm{dist}(-\nabla c(x)^\top c(x), N_\cX (x)) \quad \forall x \in \cX.
\end{align*}
However, we use the projection-based condition in \Cref{assn:reg-c-nabla-c} because it is more convenient for our analysis.
\end{remark}

We now present the main theoretical result of this section on the iteration complexity of \Cref{alg:mescal-overview}. 

\begin{theorem}[Iteration complexity]%
\label{thm:outer-iteration-complexity}
    Suppose Assumptions \ref{assn:abs-true-bounds} and \ref{assn:reg-c-nabla-c} hold. Let $\{x_k\}$ be the sequence of iterates generated by \Cref{alg:mescal-overview} with initial iterate $x_{-1} \in \cX$ and initial dual iterate $\lambda_0 \in \RR^m$. Suppose that the dual bound parameters $\gamma_k > 0$ be a bounded sequence satisfying $\sum_{k=0}^{\infty} \gamma_k = C_\gamma < \infty$, and that the dual sample sizes satisfy $|S^{(d)}_k| \geq 1$ for all $k \in \NN_0$. If the error bounds are chosen such that $\epsilong_k {\leq} \epsilong < 1$ and $\epsilonH_k {\leq} \epsilonH < 1$ for any $k \in \NN_0$, then,  for any $k = K \geq K^*$, where $\Lambda:= \|\lambda_0\| + C_\gamma$ and 
    \begin{equation}\label{eq:iter_comp_log}
    K^* := \lt\lceil \log_{\beta} \lt( \tfrac{1 + \kappa_{\nabla f} + \kappa_{\nabla c} \Lambda}{C_{\text{reg}} \alpha_0 \epsilong} \rt) \rt\rceil = \cO(\log(\tfrac{1}{\epsilong})),
    \end{equation}
    $x_K$ satisfies the following:
    \begin{enumerate}
    \item[(i)] If \textbf{Option I} is chosen in \Cref{alg:mescal-overview}, then %
    $x_K$ is an $(\epsilong, \epsilonH)$-accurate expected second-order stationary point satisfying \eqref{eq:opt_cond_exp}, with $\tilde{\lambda} := \lambda_K + \alpha_K c(x_K)$. %

    \item[(ii)] If \textbf{Option II} is chosen in \Cref{alg:mescal-overview} with $p_k = p$ for all $k \in \NN_0$, then %
    $x_K$ is an $(\epsilong, \epsilonH)$-accurate $p$-probabilistic second-order stationary point satisfying \eqref{eq:opt_cond_prob}, with $\tilde{\lambda} := \lambda_K + \alpha_K c(x_K)$. %
    \end{enumerate}
\end{theorem}

\begin{proof}

We first show that the sequence $\{\lambda_k\}$ is uniformly bounded. From \eqref{eq:mescal-dual-update}, for any $k \in \NN_0$, if $\lt\| \overline{C}_{S^{(d)}_k}(x_k) \rt\| = 0$, then $\|\lambda_{k+1}\| = \|\lambda_k\|$. If $\lt\|\overline{C}_{S^{(d)}_k}(x_k)\rt\| \neq 0$, then 
\begin{align*}
    \|\lambda_{k+1}\| \leq \|\lambda_k\| + \lt\|\min \lt\{ \alpha_k, \tfrac{\gamma_k}{\| \overline{C}_{S^{(d)}_k}(x_k) \|} \rt\} \overline{C}_{S^{(d)}_k}(x_k)\rt\| \leq \|\lambda_k\| + \gamma_k.
\end{align*}
Therefore, $\|\lambda_{k+1}\| \leq \|\lambda_k\| + \gamma_k$ for all $k \in \NN_0$. Summing this inequality from $k=0$ to $k=t-1$ for any $t\geq 1$ yields
\begin{align}
    \|\lambda_t\| \leq \|\lambda_0\| + \sum_{k=0}^{t-1}\gamma_k \leq \Lambda < \infty.
\end{align}
Therefore, we have $\|\lambda_k\| \leq \Lambda$ for any $k \in \NN_0$. Using this bound, the feasibility error at any $k \in \NN_0$ can be bounded as
\begin{align}
        \|c(x_k) \| &\leq \tfrac{1}{C_{\text{reg}}} \| \mathrm{proj}_{T_\cX (x_k)}(-\nabla c(x_k)^\top c(x_k)) \| \notag \\
        &= \tfrac{1}{C_{\text{reg}}}  \lt\| \mathrm{proj}_{T_\cX (x_k)} \lt( \tfrac{- \nabla \cL_{\alpha_k}(x_k, \lambda_k) + \nabla f(x_k) + \nabla c(x_k)^\top \lambda_k }{\alpha_k} \rt) \rt\|  \notag \\
        &= \tfrac{1}{C_{\text{reg}} \alpha_k} \| \mathrm{proj}_{T_\cX (x_k)} \lt( - \nabla \cL_{\alpha_k}(x_k, \lambda_k) + \nabla f(x_k) + \nabla c(x_k)^\top \lambda_k  \rt) \| \notag \\
        &\leq \tfrac{1}{C_{\text{reg}} \alpha_k} \lt(\| \mathrm{proj}_{T_\cX (x_k)} (- \nabla \cL_{\alpha_k}(x_k, \lambda_k)) \| + \| \nabla f(x_k) + \nabla c(x_k)^\top \lambda_k  \| \rt)   \notag \\
        & \leq \tfrac{1}{C_{\text{reg}} \alpha_k} \lt(\| \mathrm{proj}_{T_\cX (x_k)} (- \nabla \cL_{\alpha_k}(x_k, \lambda_k)) \| + \| \nabla f(x_k) \| + \| \nabla c(x_k) \| \| \lambda_k \|\rt)  \notag \\
        &\leq \tfrac{\| \mathrm{proj}_{T_\cX (x_k)} (- \nabla \cL_{\alpha_k}(x_k, \lambda_k)) \| + \kappa_{\nabla f} + \kappa_{\nabla c} \Lambda }{C_{\text{reg}} \alpha_k}, 
        \label{eq:feas-equiv}
\end{align}
where the first inequality follows from \Cref{assn:reg-c-nabla-c}, the second equality follows from the positive homogeneity property of the projection onto the tangent cone, i.e., $\mathrm{proj}_{T_\cX (x_k)}(a v) = a ~\mathrm{proj}_{T_\cX (x_k)}(v)$ for any $a > 0$ and $v \in \RR^d$, the third inequality follows from \Cref{lemm:proj-bound}, and the last inequality follows from \Cref{assn:abs-true-bounds} and $\|\lambda_k\| \leq \Lambda$.

    Let $\tilde{\lambda}_k := \lambda_k + \alpha_k c(x_k)$ for all $k \in \NN_0$. Then, we have
    \begin{align}
    \| \mathrm{proj}_{T_\cX(x_k)}(- \nabla f(x_k) - \nabla c(x_k)^\top \tilde{\lambda}_k) \|
        &= \| \mathrm{proj}_{T_\cX(x_k)}(- \nabla f(x_k) - \nabla c(x_k)^\top (\lambda_k + \alpha_k c(x_k))) \| \notag \\
        &= \| \mathrm{proj}_{T_\cX(x_k)}(- \nabla \cL_{\alpha_k}(x_k, \lambda_k)) \| \label{eq:1st-ord-equiv}.
    \end{align}

For the second-order condition, consider
\begin{align}
        &\min_{\substack{
        u \in T_\cX^{(c)}(x_k), \,
        \|u\|=1
        }}
        u^\top \lt(\nabla^2 f(x_k) + \ssum_{j=1}^m\tilde{\lambda}^{(j)}_k \nabla^2 c^{(j)}(x_k)\rt) u
        \nonumber \\
       &\qquad = \min_{\substack{
        u \in T_\cX^{(c)}(x_k), 
        \|u\|=1
        }} u^\top \lt(\nabla^2 f(x_k) + \ssum_{j=1}^m(\lambda^{(j)}_k + \alpha_k c^{(j)}(x_k)) \nabla^2 c^{(j)}(x_k)\rt) u + \alpha_k u^\top \nabla c(x_k)^\top \nabla c(x_k) u \nonumber \\
        &\qquad= \min_{\substack{
        u \in T_\cX^{(c)}(x_k), \,
        \|u\|=1}} u^\top (\nabla^2 \cL_{\alpha_k}(x_k, \lambda_k)) u \nonumber \\
        &\qquad\geq \min_{\substack{
        u \in T_\cX(x_k), \,
        \|u\|=1}} u^\top (\nabla^2 \cL_{\alpha_k}(x_k, \lambda_k)) u,  \label{eq:2nd-ord-equiv}%
\end{align}
where the first equality follows because $u \in T_\cX^{(c)}(x_k)$ implies $u \in \mathrm{Null}(\nabla c(x_k))$, and hence $\nabla c(x_k)u=0$, and the inequality is due to the fact that $T_\cX^{(c)}(x_k) \subseteq T_\cX(x_k)$.

Case (i): If \textbf{Option I} is chosen in \Cref{alg:mescal-overview}, then, taking expectations in \eqref{eq:1st-ord-equiv},and  using the tower property of conditional expectations and the first inequality in \eqref{cond:exp_primal_prob}, we have, for all $k \in \NN_0$,
\begin{align}\label{eq:1st-order-exp}
\EE[\| \mathrm{proj}_{T_\cX(x_k)} ( - \nabla f(x_k) - \nabla c(x_k)^\top \tilde{\lambda}_k )\|] &= \EE[\EE[\| \mathrm{proj}_{T_\cX(x_k)} ( - \nabla \cL_{\alpha_k} (x_k, \lambda_k)) \| \mid \cF_k]] \leq \epsilong_k. 
\end{align}

Using \eqref{eq:1st-order-exp}, taking expectations in \eqref{eq:feas-equiv},  and choosing $k=K\geq K^*$, we obtain
\begin{align}\label{eq:feas-exp}
   \EE[\|c(x_K) \|] &\leq  \tfrac{\epsilong_K + \kappa_{\nabla f} + \kappa_{\nabla c} \Lambda }{C_{\text{reg}} \alpha_K} \leq \tfrac{1 + \kappa_{\nabla f} + \kappa_{\nabla c} \Lambda }{C_{\text{reg}} \alpha_K} \leq \epsilong,
\end{align}
where the last inequality is due to \eqref{eq:iter_comp_log}. Next, taking expectations in \eqref{eq:2nd-ord-equiv}, and using the tower property of conditional expectations and the second inequality in \eqref{cond:exp_primal_prob}, we have, for all $k \in \NN_0$,
\begin{align}
&\EE\lt[
        \min_{\substack{
        u \in T_\cX^{(c)}(x_k), \,
        \|u\|=1
        }}
        u^\top \lt(\nabla^2 f(x_k) + \ssum_{j=1}^m\tilde{\lambda}^{(j)}_k \nabla^2 c^{(j)}(x_k)\rt) u \rt] \nonumber \\
        &\qquad \geq \EE \lt[\EE \lt[\min_{\substack{
        u \in T_\cX(x_k), \,
        \|u\|=1}} u^\top (\nabla^2 \cL_{\alpha_k}(x_k, \lambda_k)) u \mid \cF_k \rt]\rt] \geq - \epsilonH_k.\label{eq:2nd-ord-equiv-expect}
\end{align}
Therefore, using $\epsilong_k {\leq} \epsilong$, $\epsilonH_k {\leq} \epsilonH$,  \eqref{eq:feas-exp}, and setting $k=K$ in \eqref{eq:1st-order-exp} and \eqref{eq:2nd-ord-equiv-expect}, we conclude that $(x_K, \tilde{\lambda}_K)$ satisfies \eqref{eq:opt_cond_exp}.

Case (ii): If \textbf{Option II} is chosen in \Cref{alg:mescal-overview}, then, define the following events:
\begin{align*}
    P_{1, k} &:= \{ \| \mathrm{proj}_{T_\cX(x_k)} ( - \nabla f(x_k) - \nabla c(x_k)^\top \tilde{\lambda}_k) \| \leq \epsilong_k \}, \\
    P_{2, k} &:= \lt\{\min_{\substack{
        u \in T_\cX^{(c)}(x_k), \,
        \|u\|=1
        }}
        u^\top \lt(\nabla^2 f(x_k) + \ssum_{j=1}^m\tilde{\lambda}^{(j)}_k \nabla^2 c^{(j)}(x_k)\rt) u \geq-\epsilonH_k\rt\}, \\
    P_{3, k} &:=\lt\{ \|c(x_k)\| \leq \epsilong_k\rt\}, \\
     E_{1, k} &:= \{ \| \mathrm{proj}_{T_\cX(x_k)} ( -\nabla \cL_{\alpha_k} (x_k, \lambda_k)) \| \leq \epsilong_k \}, \\
      E_{2, k} &:= \lt\{ \min_{\substack{
        u \in T_\cX(x_k), \,
        \|u\|=1}} u^\top (\nabla^2 \cL_{\alpha_k}(x_k, \lambda_k)) u \geq -\epsilonH_k\rt\}.
\end{align*}
From \eqref{eq:1st-ord-equiv}, we have $P_{1,k}=E_{1,k}$, while \eqref{eq:2nd-ord-equiv} implies $E_{2,k}\subseteq P_{2,k}$. Moreover, by \eqref{eq:feas-equiv}, we have $E_{1,K}\subseteq P_{3,K}$ for any $K\geq K^*$. Therefore,
\begin{align*}
    \PP\lt( P_{1, K} \cap P_{2, K} \cap P_{3, K} \rt) &=  \PP\lt( E_{1, K} \cap P_{2, K} \cap P_{3, K} \rt) \geq \PP\lt( E_{1, K} \cap E_{2,K} \rt).
\end{align*}
Consider the indicator variable $\mathbf{1}_{E_{1, K} \cap E_{2,K}}$ for event $E_{1, K} \cap E_{2,K}$. %
{By} the tower property of conditional expectations and \eqref{cond:prob_primal_prob},
\begin{align*}
    \EE[\mathbf{1}_{E_{1, K} \cap E_{2,K}}] = \EE[\EE[\mathbf{1}_{E_{1, K} \cap E_{2,K}} \mid \cF_K]] \geq p_K.
\end{align*}
Therefore, using $p_K=p$, $\epsilong_k {\leq}\epsilong$, and $\epsilonH_k {\leq}\epsilonH$, we conclude that $x_K$ satisfies \eqref{eq:opt_cond_prob}, with $\tilde{\lambda} := \tilde{\lambda}_K$.

\end{proof}

\begin{remark} \label{rmk:iter-complexity}
We make the following remarks about \Cref{thm:outer-iteration-complexity}.

    \begin{itemize}
        \item The dual upper-bounding sequence $\{\gamma_k\}$ is introduced to ensure that the dual variables remain uniformly bounded. With this bound, our analysis follows a structure similar to that of penalty-based methods. In particular, due to the nonconvexity of the problem, our analysis does not exploit the potential benefits of the dual updates in establishing the iteration complexity result. In practice, however, dual updates can still be beneficial. Accordingly, in our numerical experiments in \Cref{sec:num_expts}, we set $\gamma_k = \infty$ for all $k \in \NN_0$, thereby removing the explicit upper bound on the dual updates. We observe that this choice provides good empirical performance while the resulting dual variables remain bounded in our experiments.

        \item The iteration complexity result in \Cref{thm:outer-iteration-complexity} holds for any dual sample size $|S^{(d)}_k|\geq 1$, and the effect of this sample size is not reflected in the complexity bound. This is a consequence of the dual upper-bounding sequence $\{\gamma_k\}$, which allows the analysis to accommodate arbitrary dual sample sizes. In practice, however, sufficiently large sample sizes can provide more accurate constraint estimates and, consequently, more accurate dual updates. Inspired by \cite{li_stochastic_2024} and {the fact that we can maintain near-accurate dual updates without compromising sample complexity bounds,}
        we choose $S^{(d)}_k = \Theta( (\epsilong)^{-2} )$,  which, under the standard variance assumption, gives $\EE[\| \overline{C}_{S^{(d)}_k}(x_k) - c(x_{k}) \|^2] = \cO( 1 / |S^{(d)}_k| ) = \cO((\epsilong)^2)$. Thus, the sampled constraint function provides a sufficiently accurate approximation of $c(x_{k})$ in expectation. This choice of dual sample size does not alter the overall sample complexity results established in \Cref{sec:inner_loop} or \Cref{sec:active_solvers}.

        \item Although \Cref{thm:outer-iteration-complexity} requires $\cO(\log(1/\epsilong))$ outer iterations, the same approximate second-order stationarity guarantees can be obtained using a single outer iteration by choosing the initial penalty parameter as $\alpha_0 = \cO(\tfrac{1}{\epsilong})$. {Nevertheless, using single iteration solves can yield worse sample complexity bounds in expectation and in probability, similar to the deterministic sample complexity bounds observed in \Cref{sec:active_solvers} (see \Cref{lemm:1-lower-bound-finite} and \Cref{rmk:opt-gap-discuss}).
        }
        Moreover, the use of multiple outer iterations can provide both theoretical and practical benefits in convex settings \cite{bollapragada_adaptive_2023, bollapragada_augmented_2026}. 
        Finally, although the theorem is stated with 
        $p_k=p$ for all $k\in\NN_0$, this sequence may instead be chosen to decrease gradually, provided that 
        $p_K\leq p$ for $K\geq K^*$; the same conclusions then follow \cite{bollapragada_adaptive_2023, bollapragada_augmented_2026}.
    \end{itemize}
\end{remark}

%% file: 3-inner-algorithm.tex
\section{Sample Complexity} \label{sec:inner_loop}

In this section, we establish sample complexity, as defined in \Cref{def:sample_complexity}, required to obtain {an} $(\epsilong, \epsilonH)$-accurate expected or $p$-probabilistic second-order stationary point of \eqref{eq:gen_formula} for the proposed MESCAL framework. We consider a general solver $\cM$ for solving the primal subproblem inexactly at each iteration $k \in \NN_0$ in \Cref{alg:mescal-overview} to achieve iterates $x_k$ satisfying \eqref{cond:exp_primal_prob} or \eqref{cond:prob_primal_prob} under reasonable assumptions on the theoretical properties of the solver $\cM$ (see \Cref{assn:theoretical-sample-size}). In \Cref{sec:active_solvers}, we discuss three specific existing stochastic second-order solvers that can be employed when $\cX = \RR^d$ and \textbf{Option II} is chosen in \Cref{alg:mescal-overview}, as we are not aware of existing solvers with theoretical guarantees for the more general settings considered here.

We first introduce additional notation to denote the stochastic estimates of the gradient and Hessian of the augmented Lagrangian function used by the primal solver $\cM$, as their exact counterparts may be unavailable or expensive to compute. To simplify notation throughout the analysis, we define a joint random vector $\theta = (\xi,\zetaone,\zetatwo)$, where $\xi \sim \cP_\xi$ and $\zetaone,\zetatwo \simiid \cP_\zeta$. We define the probability space on which $\theta$ is defined as  $(\cS_\theta,\cG_\theta,\cP_\theta) := \lt( \cS_\xi \times \cS_{\zeta}\times\cS_{\zeta}, \cG_\xi \otimes \cG_\zeta \otimes \cG_\zeta, \cP_\xi \otimes \cP_\zeta \otimes \cP_\zeta \rt)$. Next, we define the following stochastic gradient and Hessian estimators of the augmented Lagrangian function:
\begin{align}
    & \nabla \bar{\cL}_\alpha(x, \lambda; \theta) := \nabla F(x, \xi) + \nabla C(x, \zetaone)^\top (\lambda + \alpha C(x, \zetatwo)), \label{eq:sample_gradient_estimator} \\
    & \nabla^2 \bar{\cL}_\alpha (x, \lambda; \theta) := \nabla^2 F(x, \xi) + \sum_{j=1}^m (\lambda^{(j)} + \alpha C^{(j)}(x, \zetatwo)) \nabla^2 C^{(j)}(x, \zetaone) + \alpha \nabla C(x, \zetaone)^\top \nabla C(x, \zetatwo)
    \label{eq:sample_hessian_estimator}
\end{align}
By the independence of $\zetaone$ and $\zetatwo$, the estimators in \eqref{eq:sample_gradient_estimator} and \eqref{eq:sample_hessian_estimator} are unbiased. That is, 
$\EE[\nabla \bar{\cL}_\alpha (x, \lambda; \theta)] = \nabla \cL_\alpha (x, \lambda)$ and $\EE[\nabla^2 \bar{\cL}_\alpha (x, \lambda; \theta)] = \nabla^2 \cL_\alpha (x, \lambda)$.

Next, we state the assumptions used to establish sample complexity for the MESCAL framework. 

\begin{assumption} \label{assn:lip-nabla-f-nabla-c}
    The stochastic objective function gradient $\nabla F(x, \xi)$ is $L_{\nabla f}$-Lipschitz continuous $\cP_\xi$-almost surely on $\cX$, and each stochastic constraint gradient $\nabla C^{(j)}(x, \zeta)$ is $L_{\nabla c}$-Lipschitz continuous $\cP_\zeta$-almost surely on $\cX$. That is,
    \begin{align*}
        \| \nabla F(x, \xi) - \nabla F(y, \xi) \| &\leq L_{\nabla f} \|x - y \|, \quad \forall x, y \in \cX, \quad \cP_\xi\text{-a.s., } \quad \text{and} \\
        \| \nabla C^{(j)}(x, \zeta) - \nabla C^{(j)}(y, \zeta) \| &\leq L_{\nabla c} \|x - y \|, \quad \forall j \in \{1,2, \cdots, m\}, \quad \forall x, y \in \cX, \quad \cP_\zeta\text{-a.s.}
    \end{align*}
\end{assumption}

\begin{assumption} \label{assn:lip-nabla2-f-nabla2-c}
    The stochastic objective function Hessian $\nabla^2 F(x, \xi)$ is $L_{\nabla^2 f}$-Lipschitz continuous $\cP_\xi$-almost surely on $\cX$, and each stochastic constraint Hessian $\nabla^2 C^{(j)}(x, \zeta)$ is $L_{\nabla^2 c}$-Lipschitz continuous $\cP_\zeta$-almost surely on $\cX$. That is,
    \begin{align*}
        \| \nabla^2 F(x, \xi) - \nabla^2 F(y, \xi) \| &\leq L_{\nabla^2 f} \|x - y \|, \quad \forall x, y \in \cX, \quad \cP_\xi\text{-a.s., } \quad \text{and}\\
         \| \nabla^2 C^{(j)}(x, \zeta) - \nabla^2 C^{(j)}(y, \zeta) \| &\leq L_{\nabla^2 c} \|x - y \|, \quad \forall j \in \{1,2, \cdots, m\}, \quad \forall x, y \in \cX, \quad \cP_\zeta\text{-a.s.}
    \end{align*}
\end{assumption}

\begin{assumption} \label{assn:var-bounds}
    The stochastic objective function gradient and Hessian satisfy uniform error bounds $\cP_\xi$-almost surely. That is, there exist constants $\sigma_{\nabla f}, \sigma_{\nabla^2 f} \in [0, \infty)$ such that
    \begin{align*}
        \|\nabla F(x, \xi) - \nabla f(x)\| &\leq \sigma_{\nabla f}, \quad \text{and} \quad \
        \|\nabla^2 F(x, \xi) - \nabla^2 f(x)\| \leq \sigma_{\nabla^2 f} \quad \forall x \in \cX, \quad \cP_\xi\text{-a.s.}
    \end{align*}
    In addition, the stochastic constraint function, constraint Jacobian and constraint Hessian satisfy uniform error bounds $\cP_\zeta$-almost surely. That is, there exist constants $\sigma_{c}, \sigma_{\nabla c}, \sigma_{\nabla^2 c} \in [0,\infty)$ such that
    \begin{align*}
        & \| C(x, \zeta) - c(x) \| \leq \sigma_{c},
        \| \nabla C(x, \zeta) - \nabla c(x) \| \leq \sigma_{\nabla c}, \text{ and }
        \| \nabla^2 C(x, \zeta) - \nabla^2 c(x) \| \leq \sigma_{\nabla^2 c} \quad \forall x \in \cX, \cP_\zeta\text{-a.s.}
    \end{align*}
\end{assumption}
The following lemma provides useful results as a consequence of these assumptions.
\begin{lemma} \label{lemm:lip-c}
    Suppose Assumptions \ref{assn:abs-true-bounds} and \ref{assn:var-bounds} hold. Then, the stochastic objective function gradient and stochastic constraint function, gradient, and Hessian are bounded. That is, %
    \begin{align*}
       & \|\nabla F(x,\xi)\| \leq \kappa_{\nabla f} + \sigma_{\nabla f}, \quad \forall x \in \cX, \quad \cP_\xi\text{-a.s.},\\ 
       & \|C(x, \zeta)\| \leq \kappa_c + \sigma_c, \|\nabla C(x, \zeta)\| \leq \kappa_{\nabla c} + \sigma_{\nabla c}, \text{ and } \|\nabla^2 C(x, \zeta)\| \leq \kappa_{\nabla^2 c} + \sigma_{\nabla^2 c}, \quad \forall x \in \cX, \quad \cP_\zeta\text{-a.s.}
    \end{align*}
    Furthermore, each stochastic constraint function $C^{(j)}$ is Lipschitz continuous $\cP_\zeta$-almost surely on $\cX$. That is, %
    $$|C^{(j)}(x, \zeta)-C^{(j)}(y, \zeta)| \leq L_c \|x-y\|, \qquad \forall j \in \{1,2, \cdots, m\}, \quad \forall x, y \in \cX, \quad \cP_\zeta\text{-a.s.,}$$
    where $L_c := \kappa_{\nabla c} + \sigma_{\nabla c} < \infty$.
\end{lemma}
The proof of this lemma is provided in \Cref{sec:appendix_lemm:lip-c}.

\begin{remark}
    Assumptions~\ref{assn:lip-nabla-f-nabla-c}, \ref{assn:lip-nabla2-f-nabla2-c}, and \ref{assn:var-bounds} are used to establish that the stochastic augmented Lagrangian {gradient and Hessian estimators} are Lipschitz continuous and that their errors relative to the corresponding exact quantities are uniformly bounded, as shown in the following technical lemmas. We note that these are relatively strong properties and may not be needed by all solvers that can be employed to solve the primal subproblem. However, solvers that employ variance reduction techniques typically require these properties to hold for the stochastic estimators themselves {\cite{arjevani_lower_2023}}, rather than only for the corresponding expected quantities. Moreover, the almost sure error bounds in Assumption~\ref{assn:var-bounds}, or the almost sure boundedness of the stochastic quantities established in~\Cref{lemm:lip-c}, are commonly utilized in constrained stochastic optimization; see, for example, \cite{curtis_stochasticgradientbased_2025,yu_online_2017}. To provide a unified presentation across different potential solvers, including those considered in \Cref{sec:active_solvers}, which operate under varying degrees of strictness in their requirements on the augmented Lagrangian properties, we adopt these assumptions. 
    {F}or some solvers, it may suffice to assume Lipschitz continuity of the true augmented Lagrangian gradients and Hessians and boundedness of the errors in expectation. In such cases, the assumptions above can be equivalently imposed on the corresponding expected quantities, with the error bounds required only in expectation.

\end{remark}

Next, we provide technical lemmas establishing the properties of the stochastic augmented Lagrangian {gradient and Hessian estimators.}

\begin{lemma} \label{lemm:2-lipschitz-grad-aug-lag}
    Suppose Assumptions \ref{assn:abs-true-bounds}, \ref{assn:lip-nabla-f-nabla-c}, and \ref{assn:var-bounds} hold. At every iteration $k \in \NN_0$ of \Cref{alg:mescal-overview}, suppose that the dual variable satisfies $\| \lambda_k \| \leq \Lambda < \infty$. Let
    \begin{align*}
        L^{(g)}_1 := L_{\nabla f} + \sqrt{m} L_{\nabla c} \Lambda, \text{ and } 
        L^{(g)}_2 := m (L_c (\sigma_{\nabla c} + \kappa_{\nabla c}) + L_{\nabla c} (\sigma_{c} + \kappa_{c})).
    \end{align*}
    Then, at each iteration $k \in \NN_0$, the function $\nabla \bar{\cL}(x, \lambda_k; \theta)$ is $L^{(g)}_{\cL, k}$-Lipschitz continuous on $\cX$ $\cP_\theta$-almost surely. That is,
    \begin{align*}
        \| \nabla \bar{\cL}_{\alpha_k} (x, \lambda_k; \theta) - \nabla \bar{\cL}_{\alpha_k} (y, \lambda_k; \theta) \| \leq L^{(g)}_{\cL, k} \|x - y\|, \quad \forall k \in \NN_0, \quad \forall x, y \in \cX, \quad \cP_\theta\text{-a.s.}.,
    \end{align*}
    where $L^{(g)}_{\cL, k} \leq L^{(g)}_1 + \alpha_k L^{(g)}_2$.
\end{lemma}
\begin{proof}
    Consider the difference between stochastic gradient estimators of augmented Lagrangian function defined in \eqref{eq:sample_gradient_estimator} at $x,y \in \cX$,
    \begin{align}
        \| \nabla \bar{\cL}_{\alpha_k} (x, \lambda_k; \theta) &- \nabla \bar{\cL}_{\alpha_k} (y, \lambda_k; \theta) \| \nonumber \\
        &= \Big\| (\nabla F(x, \xi) - \nabla F(y, \xi)) + \ssum_{j=1}^m \lambda^{(j)}_k  (\nabla C^{(j)}(x, \zetaone) - \nabla C^{(j)}(y, \zetaone))  \nonumber \\
        & \qquad + \alpha_k (\nabla C(x, \zetaone)^\top C(x, \zetatwo) - \nabla C(y, \zetaone)^\top C(y, \zetatwo))\Big\|  \nonumber \\
        &\leq \| \nabla F(x, \xi) - \nabla F(y, \xi) \| + \ssum_{j=1}^m |\lambda^{(j)}_k| \| \nabla C^{(j)}(x, \zetaone) - \nabla C^{(j)}(y, \zetaone) \|  \nonumber \\
        & \qquad + \alpha_k \| \nabla C(x, \zetaone)^\top C(x, \zetatwo) - \nabla C(y, \zetaone)^\top C(y, \zetatwo) \| \nonumber \\
        &\leq (L_{\nabla f} + \sqrt{m} L_{\nabla c} \Lambda) \|x-y\| + \alpha_k \| \nabla C(x, \zetaone)^\top C(x, \zetatwo) - \nabla C(y, \zetaone)^\top C(y, \zetatwo) \|,  \label{eq:lem33_proof_1} 
    \end{align}
where the last inequality follows from \Cref{assn:lip-nabla-f-nabla-c} and $\|\lambda_k\|_1 \leq \sqrt{m}\|\lambda_k\| \leq \sqrt{m} \Lambda$. Now, consider, 
    \begin{align}
        \| (\nabla C(x, \zetaone)^\top C(x, \zetatwo) &- \nabla C(y, \zetaone)^\top C(y, \zetatwo))\| \nonumber\\
        &= \| \nabla C(x, \zetaone)^\top (C(x, \zetatwo) - C(y, \zetatwo)) + (\nabla C(x, \zetaone) - \nabla C(y, \zetaone))^\top C(y, \zetatwo)  \| \nonumber \\
        &\leq \ssum_{j=1}^m \lt\| \nabla C^{(j)}(x, \zetaone) \rt\| \lt\| C^{(j)}(x, \zetatwo) - C^{(j)}(y, \zetatwo) \rt\| \nonumber\\
        & \qquad + \ssum_{j=1}^m \lt\| \nabla C^{(j)}(x, \zetaone) - \nabla C^{(j)}(y, \zetaone)\rt\| \lt\| C^{(j)}(y, \zetatwo) \rt\| \nonumber \\
        &\leq m (L_c (\sigma_{\nabla c} + \kappa_{\nabla c}) + L_{\nabla c} (\sigma_{c} + \kappa_{c})) \lt\|x-y\rt\|, \label{eq:lem33_proof_2} 
    \end{align}
where the last inequality follows from Lemma~\ref{lemm:lip-c}. Substituting \eqref{eq:lem33_proof_2} in \eqref{eq:lem33_proof_1} completes the proof.     

\end{proof}

\begin{lemma} \label{lemm:3-lipschitz-hess-aug-lag}
    Suppose Assumptions \ref{assn:abs-true-bounds}, \ref{assn:lip-nabla-f-nabla-c}, \ref{assn:lip-nabla2-f-nabla2-c}, and \ref{assn:var-bounds} hold. At every iteration $k \in \NN_0$ of \Cref{alg:mescal-overview}, suppose that the dual variable satisfies $\| \lambda_k \| \leq \Lambda < \infty$. Let %
    \begin{align*}
        L^{(H)}_1 &:= L_{\nabla^2 f} + \sqrt{m} L_{\nabla^2 c} \Lambda, \text{ and } \\
        L^{(H)}_2 &:= m \lt( 2 L_{\nabla c} (\sigma_{\nabla c} + \kappa_{\nabla c})  + L_c (\kappa_{\nabla^2 c} + \sigma_{\nabla^2 c}) + L_{\nabla^2 c} (\sigma_{c} + \kappa_{c}) \rt).
    \end{align*}
    Then, at each iteration $k \in \NN_0$, the function $\nabla^2 \bar{\cL}(x, \lambda_k; \theta)$ is $L^{(H)}_{\cL, k}$-Lipschitz continuous on $\cX$ $\cP_\theta$-almost surely. That is,
    \begin{align*}
        \| \nabla^2 \bar{\cL}_{\alpha_k} (x, \lambda_k; \theta) - \nabla^2 \bar{\cL}_{\alpha_k} (y, \lambda_k; \theta) \| \leq L^{(H)}_{\cL, k} \|x - y\|, \quad \forall k \in \NN_0, \quad \forall x, y \in \cX, \quad \cP_\theta\text{-a.s.,}
    \end{align*}
    where $L^{(H)}_{\cL, k} \leq L^{(H)}_1 + \alpha_k L^{(H)}_2$.
\end{lemma}

\begin{proof}
Consider the difference between stochastic Hessian estimators of augmented Lagrangian function defined in \eqref{eq:sample_hessian_estimator} at $x,y \in \cX$,
\begin{align}
    \| \nabla^2 \bar{\cL}_{\alpha_k} (x, \lambda_k; \theta) &- \nabla^2 \bar{\cL}_{\alpha_k} (y, \lambda_k; \theta) \| \nonumber \\
    &= \Big\| (\nabla^2 F(x, \xi) - \nabla^2 F(y, \xi)) + \sum_{j=1}^m \lambdaj_k (\nabla^2 C^{(j)} (x, \zetaone) - \nabla^2 C^{(j)} (y, \zetaone))
    \nonumber\\ 
    & \qquad + \alpha_k \sum_{j=1}^m  ( \nabla C^{(j)} (x, \zetaone) \nabla C^{(j)} (x, \zetatwo)^\top -  \nabla C^{(j)} (y, \zetaone) \nabla C^{(j)} (y, \zetatwo)^\top) 
    \nonumber\\ 
    & \qquad + \alpha_k \sum_{j=1}^m  ( \nabla^2 C^{(j)} (x, \zetaone) C^{(j)} (x, \zetatwo) -  \nabla^2 C^{(j)} (y, \zetaone) C^{(j)} (y, \zetatwo)) \Big\| \nonumber\\
    &\leq \underbrace{\| \nabla^2 F(x, \xi) - \nabla^2 F(y, \xi) \| + \sum_{j=1}^m |\lambda_k^{(j)}|\| (\nabla^2 C^{(j)} (x, \zetaone) - \nabla^2 C^{(j)} (y, \zetaone)) \|}_{T_1} \nonumber\\ 
    & \qquad + \alpha_k \sum_{j=1}^m\underbrace{ \|  ( \nabla C^{(j)} (x, \zetaone) \nabla C^{(j)} (x, \zetatwo)^\top -  \nabla C^{(j)} (y, \zetaone) \nabla C^{(j)} (y, \zetatwo)^\top)\|}_{T_2} \nonumber\\ 
    & \qquad + \alpha_k \sum_{j=1}^m\underbrace{ \|  ( \nabla^2 C^{(j)} (x, \zetaone) C^{(j)} (x, \zetatwo) -  \nabla^2 C^{(j)} (y, \zetaone) C^{(j)} (y, \zetatwo)) \|}_{T_3}. \label{eq:lem34_proof_0}
\end{align}

We now bound each of the norm terms $T_1$, $T_2$, and $T_3$ defined above separately. First, %
for $T_1$, %
by \Cref{assn:lip-nabla2-f-nabla2-c} and $\|\lambda_k\|_1 \leq \sqrt{m}\|\lambda_k\| \leq \sqrt{m} \Lambda$, we have,
\begin{align}
    T_1 &\leq L_{\nabla^2 f} \|x - y\| + \sum_{j=1}^m  |\lambda_k^{(j)} | L_{\nabla^2 c} \|x - y\| \leq (L_{\nabla^2 f} + \sqrt{m}L_{\nabla^2 c}\Lambda)\|x-y\|. \label{eq:lem34_proof_1}
\end{align}

Next, we bound $T_2$ for any $j \in \{1,\dots,m\}$ as follows:
\begin{align}
    \|\nabla C^{(j)} (x, \zetaone)& \nabla C^{(j)} (x, \zetatwo)^\top - \nabla C^{(j)} (y, \zetaone) \nabla C^{(j)} (y, \zetatwo)^\top\| \nonumber\\
    &= \| \nabla C^{(j)} (x, \zetaone) \nabla C^{(j)} (x, \zetatwo)^\top - \nabla C^{(j)} (x, \zetaone) \nabla C^{(j)} (y, \zetatwo)^\top \nonumber\\
    & \qquad + \nabla C^{(j)} (x, \zetaone) \nabla C^{(j)} (y, \zetatwo)^\top - \nabla C^{(j)} (y, \zetaone) \nabla C^{(j)} (y, \zetatwo)^\top\| \nonumber\\
    &\leq \| \nabla C^{(j)} (x, \zetaone) ( \nabla C^{(j)} (x, \zetatwo) - \nabla C^{(j)} (y, \zetatwo))^\top \| \nonumber\\
    & \qquad + \|( \nabla C^{(j)} (x, \zetaone) - \nabla C^{(j)} (y, \zetaone)) \nabla C^{(j)} (y, \zetatwo)^\top \| \nonumber\\
    &\leq \| \nabla C^{(j)} (x, \zetaone) \| \| \nabla C^{(j)} (x, \zetatwo) - \nabla C^{(j)} (y, \zetatwo) \| \nonumber\\ 
    & \qquad + \| \nabla C^{(j)} (x, \zetaone) - \nabla C^{(j)} (y, \zetaone) \| \| \nabla C^{(j)} (y, \zetatwo) \| \nonumber \\
    & \leq 2(\kappa_{\nabla c} + \sigma_{\nabla c}) L_{\nabla c} \| x - y \|, \label{eq:lem34_proof_2}
\end{align}
where the final inequality follows from \Cref{assn:lip-nabla-f-nabla-c} and \Cref{lemm:lip-c}.

Finally, we bound $T_3$ for any $j \in \{1,\dots,m\}$ as follows:
\begin{align}
    \|\nabla^2 C^{(j)} (x, \zetaone) & C^{(j)} (x, \zetatwo) - \nabla^2 C^{(j)} (y, \zetaone) C^{(j)} (y, \zetatwo)\| \nonumber\\
    &= \Big\| \nabla^2 C^{(j)} (x, \zetaone) C^{(j)} (x, \zetatwo) - \nabla^2 C^{(j)} (x, \zetaone) C^{(j)} (y, \zetatwo) \nonumber\\
    & \qquad + \nabla^2 C^{(j)} (x, \zetaone) C^{(j)} (y, \zetatwo) - \nabla^2 C^{(j)} (y, \zetaone) C^{(j)} (y, \zetatwo)\Big\| \nonumber\\
    &\leq \| \nabla^2 C^{(j)} (x, \zetaone) (C^{(j)} (x, \zetatwo) - C^{(j)} (y, \zetatwo)) \| \nonumber\\
    & \qquad + \| ( \nabla^2 C^{(j)} (x, \zetaone) - \nabla^2 C^{(j)} (y, \zetaone)) C^{(j)} (y, \zetatwo) \| \nonumber\\
    & \leq \| \nabla^2 C^{(j)}(x, \zetaone) \| | C^{(j)}(x, \zetatwo) - C^{(j)}(y, \zetatwo) | \nonumber\\
    & \qquad + \| \nabla^2 C^{(j)}(x, \zetaone) - \nabla^2 C^{(j)}(y, \zetaone) \| | C^{(j)}(y, \zetatwo)| \nonumber\\
    & \leq (L_c \lt( \sigma_{\nabla^2 c} + \kappa_{\nabla^2 c} \rt) +  L_{\nabla^2 c} \lt( \sigma_{c} + \kappa_{c} \rt)) \|x - y \|, \label{eq:lem34_proof_3} 
\end{align}
where the final inequality follows from \Cref{assn:lip-nabla2-f-nabla2-c} and \Cref{lemm:lip-c}. Substituting \eqref{eq:lem34_proof_1}, \eqref{eq:lem34_proof_2}, and \eqref{eq:lem34_proof_3} in \eqref{eq:lem34_proof_0} completes the proof. 

\end{proof}

\begin{lemma} \label{lemm:4-finite-bounded-variance}
    Suppose Assumptions \ref{assn:abs-true-bounds} and \ref{assn:var-bounds} hold. At every iteration $k \in \NN_0$ of \Cref{alg:mescal-overview}, suppose that the dual variable satisfies %
    $\| \lambda_k \| \leq \Lambda < \infty$. Let %
    \begin{align*}
        \sigma^{(g)}_1 := \sigma_{\nabla f} + \sigma_{\nabla c} \Lambda, 
        \quad \text{and} \quad
        \sigma^{(g)}_2 :=  \sigma_{\nabla c} \sigma_c + \sigma_{\nabla c}\kappa_c + \kappa_{\nabla c} \sigma_c.
    \end{align*}
    Then, at each iteration $k \in \NN_0$, %
    the stochastic gradient error of $\nabla \bar{\cL}_{\alpha_k} (x, \lambda_k; \theta) $ is uniformly bounded by $\sigma^{(g)}_{\cL, k}$ on $\cX$ $\cP_\theta$-almost surely. That is,
    \begin{align*}
        \| \nabla \bar{\cL}_{\alpha_k} (x, \lambda_k; \theta) - \nabla \cL_{\alpha_k} (x, \lambda_k) \| \leq \sigma^{(g)}_{\cL, k}, \quad \forall k \in \NN_0, \quad \forall x \in \cX, \quad \cP_\theta\text{-a.s.,}
    \end{align*}
    where $\sigma^{(g)}_{\cL, k} \leq \sigma^{(g)}_1 + \alpha_k \sigma^{(g)}_2$.
\end{lemma}
\begin{proof}
    For each iteration $k \in \NN_0$, using \Cref{assn:abs-true-bounds}, \Cref{assn:var-bounds}, \Cref{lemm:lip-c}, and $\|\lambda_k\| \leq \Lambda$, we get,
    \begin{align*}
        & \| \nabla \bar{\cL}_{\alpha_k}(x, \lambda_k; \theta) - \nabla \cL_{\alpha_k}(x, \lambda_k) \| \\ 
        & \quad = \| \nabla F(x, \xi) - f(x) + (\nabla C(x, \zetaone) - \nabla c(x))^\top \lambda_k  + \alpha_k (\nabla C(x, \zetaone)^\top C(x, \zetatwo) - \nabla c(x)^\top c(x)) \| \\
        & \quad \leq \| \nabla F(x, \xi) - f(x) \| + \| \nabla C(x, \zetaone) - \nabla c(x) \| \| \lambda_k \|  + \alpha_k \| \nabla C(x, \zetaone)^\top C(x, \zetatwo) - \nabla c(x)^\top c(x) \| \\
        & \quad \leq \sigma_{\nabla f} + \sigma_{\nabla c} \Lambda + \alpha_k \| (\nabla C(x, \zetaone) - \nabla c(x))^\top C(x, \zetatwo) + \nabla c(x)^\top (C(x, \zetatwo) - c(x)) \| \\
        & \quad \leq \sigma_{\nabla f} + \sigma_{\nabla c} \Lambda + \alpha_k ( \| \nabla C(x, \zetaone) - \nabla c(x) \| \| C(x, \zetatwo) \| + \| \nabla c(x) \| \| C(x, \zetatwo) - c(x) \|) \\
        & \quad \leq \sigma_{\nabla f} + \sigma_{\nabla c} \Lambda + \alpha_k ( \sigma_{\nabla c} (\sigma_c + \kappa_c) + \kappa_{\nabla c} \sigma_c).
    \end{align*}
\end{proof}

Finally, we state an existing result on the error bounds of stochastic Hessian estimators, provided that the stochastic gradient estimators are Lipschitz continuous almost surely.

\begin{lemma}[Page 4, \cite{tripuraneni_stochastic_2018}; \cite{tropp_introduction_2015}] \label{lemm:5-finite-hess-variance}
    Suppose Assumptions \ref{assn:abs-true-bounds}, \ref{assn:lip-nabla-f-nabla-c}, and \ref{assn:var-bounds} hold. At every iteration $k \in \NN_0$ of \Cref{alg:mescal-overview}, suppose that the dual variable satisfies $\| \lambda_k \| \leq \Lambda < \infty$. Let
    \begin{align*}
        & L^{(g)}_1 := L_{\nabla f} + m L_{\nabla c} \Lambda, \:
        L^{(g)}_2 := m (L_{\nabla c} (\sigma_{c} + \kappa_{c}) + L_c (\sigma_{\nabla c} + \kappa_{\nabla c})), \\
        & \sigma^{(H)}_1 := 2 L^{(g)}_1, \: \text{ and } \:
        \sigma^{(H)}_2 := 2 L^{(g)}_2.
    \end{align*}
    Then, at each iteration $k \in \NN_0$, the stochastic Hessian error of $\nabla^2 \bar{\cL}_{\alpha_k} (x, \lambda_k; \theta) $ is uniformly bounded by $\sigma^{(H)}_{\cL, k}$ on $\cX$ $\cP_\theta$-almost surely. That is,
    \begin{align*}
        \| \nabla^2 \bar{\cL}_{\alpha_k} (x, \lambda_k; \theta) - \nabla^2 \cL_{\alpha_k} (x, \lambda_k) \| \leq \sigma^{(H)}_{\cL, k}, \quad \forall k \in \NN_0, \quad \forall x \in \cX, \quad \cP_\theta\text{-a.s.,}
    \end{align*}
    where $\sigma^{(H)}_{\cL, k} \leq \sigma^{(H)}_1 + \alpha_k \sigma^{(H)}_2$.
\end{lemma}

\begin{remark}
    We note that each of the terms ${L^{(g)}_{\cL,k}}$, ${L^{(H)}_{\cL,k}}$, ${\sigma^{(g)}_{\cL,k}}$, and ${\sigma^{(H)}_{\cL,k}}$ defined in \cref{lemm:2-lipschitz-grad-aug-lag,lemm:3-lipschitz-hess-aug-lag,lemm:4-finite-bounded-variance,lemm:5-finite-hess-variance} is bounded above by an expression of the form $c_1+\alpha_k c_2$ for all $k \in \NN_0$, where $c_1$ and $c_2$ are constants independent of $k$. Since ${\alpha_k}$ is an increasing sequence, these upper bounds are also increasing in $k$. Consequently, the quantities governing the sample complexity of the primal solver $\cM$ can be bounded in terms of the penalty parameter $\alpha_k$.
\end{remark}
We now establish bounds on the initial optimality gap of the primal subproblems at each iteration $k$.
\begin{lemma} \label{lemm:1-lower-bound-finite}
    Suppose \Cref{assn:abs-true-bounds} holds. At every iteration $k \in \NN_0$ of \Cref{alg:mescal-overview}, suppose that the dual variable satisfies $\| \lambda_k \| \leq \Lambda < \infty$ and $\epsilong_k < 1$. For all $k \in \NN_0$, let $x_k^* \in \arg \min_{x \in \cX} \cL_{\alpha_k} (x, \lambda_k)$, and define
        $$\Delta_k := \cL_{\alpha_k}(x_{k-1}, \lambda_k) - \cL_{\alpha_k}(x_k^*, \lambda_k).$$
        
    Further, define, 
        $$\Delta^*_k := 2 \kappa_{f} + \tfrac{\Lambda^2}{\alpha_0} + \alpha_k \kappa_{c}^2, \quad \text{and} \quad \Delta^* := 2 \kappa_{f} + \tfrac{\Lambda^2}{\alpha_0} + \tfrac{\kappa_c \beta(1 + \kappa_{\nabla c} + \Lambda \kappa_{\nabla f})}{C_{\text{reg}}}.$$
        
    Then, for all $k \in \NN_0$, 
    \begin{equation}\label{eq:det_delta_bnd}
        \Delta_k \leq \Delta^*_k = {\cO(\alpha_k)}.
    \end{equation}
    
    Further suppose \Cref{assn:reg-c-nabla-c} holds. Then, for \Cref{alg:mescal-overview}
    \begin{enumerate}[label=(\roman*)]
        \item under \textbf{Option I}, %
        for all $k \geq 1$,
        \begin{align}\label{eq:exp_delta_bnd}
            \EE[\Delta_k 
            ] \leq \Delta^*.
        \end{align}
        
        \item under \textbf{Option II}, %
        for all $k \geq 1$,
        \begin{align}\label{eq:prob_delta_bnd}
            \PP(\Delta_k \leq \Delta^* 
            \mid \cF_{k-1}
            ) \geq p_{k-1}.
        \end{align}
        Furthermore, for any $K\geq 1$, 
        \begin{align}\label{eq:delta-prob-overall-bound_init}
            \PP(\cap_{k=1}^{K}\{ \Delta_k \leq \Delta^* \}
            ) \geq 
            {\prod_{k=0}^{K-1}p_k. }
        \end{align}
    \end{enumerate}
\end{lemma}
\begin{proof}
    Using Young's inequality, \Cref{assn:abs-true-bounds}, {$\|\lambda_k\| \leq \Lambda$}, and $\alpha_k \geq \alpha_0$ for all $k\in \NN_0$, we have, 
    \begin{align}
        \cL_{\alpha_k}(x_{k-1}, \lambda_k) &- \cL_{\alpha_k}(x_k^*, \lambda_k) \notag\\
        & = f(x_{k-1}) + \lambda_k^\top c(x_{k-1}) + \tfrac{\alpha_k}{2} \|c(x_{k-1})\|^2 - f(x_k^*) - \lambda_k^\top c(x_k^*) - \tfrac{\alpha_k}{2}  \|c(x_k^*)\|^2 \notag \\
        &\leq f(x_{k-1}) + \lambda_k^\top c(x_{k-1}) + \tfrac{\alpha_k}{2} \|c(x_{k-1})\|^2 - f(x_k^*) + \tfrac{1}{2 \alpha_k} \| \lambda_k \|^2  \notag \\
        & \leq 2 \kappa_{f}  + \| \lambda_k \| \| c(x_{k-1}) \| + \tfrac{\alpha_k}{2} \|c(x_{k-1})\|^2 + \tfrac{\| \lambda_k \|^2}{2 \alpha_k}  \notag \\
        &\leq 2 \kappa_{f}  + \Lambda \| c(x_{k-1}) \| + \tfrac{\alpha_k}{2} \|c(x_{k-1})\|^2 + \tfrac{\Lambda^2}{2 \alpha_k}  \notag \\
        &\leq 2 \kappa_{f} + \alpha_k \|c(x_{k-1})\|^2 + \tfrac{\Lambda^2}{\alpha_0} \label{eq:opt_gap_general} \\
        & \leq 2 \kappa_{f} + \alpha_k \kappa_{c}^2 + \tfrac{\Lambda^2}{\alpha_0}. \label{eq:opt_gap_upper}
    \end{align}
    Therefore, $\Delta_k \leq \Delta^*_k = \cO(\alpha_k)$ for all $k \in \NN_o$, the equality is due to the fact that $\alpha_k$ keeps growing with $k$, whereas $\kappa_f,\kappa_c, \alpha_0$, and $\Lambda$ are constant with respect to $k$.
    
    Case (i): If \textbf{Option I} is chosen in \Cref{alg:mescal-overview}, then, taking expectation in \eqref{eq:opt_gap_general}, and using \Cref{assn:abs-true-bounds}, \eqref{eq:feas-equiv}, \eqref{eq:1st-order-exp} and $\epsilong_k < 1$,  it holds for all $k \geq 1$ that
    \begin{align*}
        \EE [\cL_{\alpha_k}(x_{k-1}, \lambda_k) - \cL_{\alpha_k}(x_k^*, \lambda_k) 
        ] &\leq 2 \kappa_{f} + \tfrac{\Lambda^2}{\alpha_0} +  \alpha_k \EE[\|c(x_{k-1})\|^2 
        ] \\
        & \leq 2 \kappa_{f} + \tfrac{\Lambda^2}{\alpha_0} +  \kappa_c \beta \alpha_{k-1} \EE[\|c(x_{k-1})\|
        ] \\
        & \leq 2 \kappa_{f} + \tfrac{\Lambda^2}{\alpha_0} + \tfrac{\kappa_c \beta(1 + \kappa_{\nabla c} + \Lambda \kappa_{\nabla f})}{C_{\text{reg}}}.
    \end{align*}
    Case (ii): Similarly, if \textbf{Option II} is chosen in \Cref{alg:mescal-overview}, then, using \eqref{eq:opt_gap_general}, \Cref{assn:abs-true-bounds}, \eqref{eq:feas-equiv}, and \eqref{cond:prob_primal_prob}, we have, {for all $k \geq 1$ that}
    \begin{align}
        \PP(\cL_{\alpha_k}(x_{k-1}, \lambda_k) &- \cL_{\alpha_k}(x_k^*, \lambda_k) \leq \Delta^* 
        \mid \cF_{k - 1}
        ) \notag\\
        &\geq \PP(2 \kappa_{f} + \alpha_k \|c(x_{k-1})\|^2 + \tfrac{\Lambda^2}{\alpha_0} \leq \Delta^*
        \mid \cF_{k - 1}
        ) \notag\\
        &\geq \PP(\| \mathrm{proj}_{T_\cX (x_{k-1})} (- \nabla \cL_{\alpha_{k-1}}(x_{k-1}, \lambda_{k-1})) \| \leq \epsilong_{k-1}\mid 
        \cF_{k - 1} %
        )  \notag\\
        &\geq p_{k-1} \label{eq:lem36_prob_1}
    \end{align}
    For $k\in\{1,\ldots,K\}$, define the events 
    \begin{align*}
        D_k := \{ \Delta_k \leq \Delta^* \}, \quad D:=\cap_{k=1}^K D_k.
    \end{align*}
    Using \eqref{eq:lem36_prob_1} and indicator function, we have
    \begin{align}
        \PP(D_k \mid 
        \cF_{k - 1} %
        ) = \EE[\mathbf{1}_{D_k} \mid 
        \cF_{k - 1} %
        ] \geq p_{k-1} 
        ,\quad \forall k = \{1, \cdots, K\}. \label{eq:single-Delta-prob}
    \end{align}
    Considering $\PP\lt(D 
    \rt)$ and using \eqref{eq:single-Delta-prob}, we obtain 
    \begin{align}
        \PP\lt(D
        \rt) &= \EE[\pprod_{k=1}^K \mathbf{1}_{D_k} 
        ] \\ %
        &= \EE[\mathbf{1}_{D_K} \pprod_{k=1}^{K-1} \mathbf{1}_{D_k} 
        ] \notag \\
        &= \EE[\EE[\mathbf{1}_{D_K} \pprod_{k=1}^{K-1} \mathbf{1}_{D_k} 
        \mid \cF_{K-1} %
        ] 
        ] \notag \\
        &= \EE[\EE[\mathbf{1}_{D_K} \mid
        \cF_{K-1} %
        ] \pprod_{k=1}^{K-1} \mathbf{1}_{D_k} 
        ] \notag \\
        &\geq p_{K-1} \EE[ \pprod_{k=1}^{K-1} \mathbf{1}_{D_k} 
        ]. \label{eq:delta-unroll-end} 
    \end{align}
    where the fourth equality follows because $\pprod_{k=1}^{K-1}\mathbf{1}_{D_k}$ is determined by $x_{K-2}$ and $\lambda_{K-1}$.
    Applying the same argument recursively yields
    \begin{align*}%
        \PP(D
        ) \geq \prod_{k=0}^{K-1}p_k.
    \end{align*}
\end{proof}

\begin{remark} \label{rmk:opt-gap-discuss}
    We established three bounds on the initial optimality gap $\Delta_k$ of the primal subproblem at each iteration in \Cref{lemm:1-lower-bound-finite}. These bounds are useful for establishing the sample complexity results. Among the three bounds, the deterministic bound in \eqref{eq:det_delta_bnd} provides the strongest type of guarantee, as it holds deterministically; however, the corresponding bound $\Delta^*_k$ increases with $k$ and is therefore less refined than the bounds obtained in expectation and in probability. We use the deterministic bound to establish deterministic sample complexity results for the different primal solvers considered in \Cref{sec:active_solvers}. In contrast, the expectation bound in \eqref{eq:exp_delta_bnd} and the probability bound in \eqref{eq:prob_delta_bnd} provide weaker types of guarantees, but yield the improved uniform bound 
    {$\Delta_k \leq \Delta^*$.}
    These bounds are therefore used to establish stronger sample complexity results in expectation and in probability, respectively. Such expectation-based initial optimality gap bounds have also been utilized in the literature for analyzing stochastic subproblem solvers; see, e.g., \cite[Lemma 5]{li_stochastic_2024}. The improvement in our expectation and probability bounds follows from exploiting the fact that $\|c(x_{k-1})\|$ decreases as the iterations progress, either in expectation or in probability, respectively. Nevertheless, these bounds remain weaker than those established in the strongly convex setting in \cite{bollapragada_adaptive_2023}, where the optimality gap is shown to decrease as the iterations progress due to the strong convexity of the primal subproblems considered therein.
\end{remark}

Based on the technical \Cref{lemm:1-lower-bound-finite,lemm:2-lipschitz-grad-aug-lag,lemm:3-lipschitz-hess-aug-lag,lemm:4-finite-bounded-variance,lemm:5-finite-hess-variance}, we introduce the following definition for ease of presentation.  

\begin{definition} \label{def:complexity_factors}
    Suppose Assumptions \ref{assn:abs-true-bounds}, {\ref{assn:reg-c-nabla-c}}, \ref{assn:lip-nabla-f-nabla-c}, \ref{assn:lip-nabla2-f-nabla2-c}, and \ref{assn:var-bounds} hold. Then, for each iteration $k \in \NN_0$ of \Cref{alg:mescal-overview}, we define the following quantities: %
    \begin{itemize}
        \item Lipschitz constant $L^{(g)}_{\cL, k}$ for $\nabla \bar{\cL}_{\alpha_k}(\cdot, \lambda_k; \theta)$ as per \Cref{lemm:2-lipschitz-grad-aug-lag}.

        \item Lipschitz constant $L^{(H)}_{\cL, k}$ for $\nabla^2 \bar{\cL}_{\alpha_k}(\cdot, \lambda_k;\theta)$ as per \Cref{lemm:3-lipschitz-hess-aug-lag}.

        \item Variance bound $\sigma^{(g)}_{\cL, k}$ for $\nabla \bar{\cL}_{\alpha_k}(\cdot, \lambda_k;\theta)$ as per \Cref{lemm:4-finite-bounded-variance}.

        \item Variance bound $\sigma^{(H)}_{\cL, k}$ for $\nabla^2 \bar{\cL}_{\alpha_k}(\cdot, \lambda_k; \theta)$ as per \Cref{lemm:5-finite-hess-variance}.

        \item Initial optimality gap $\Delta_k$ as per \Cref{lemm:1-lower-bound-finite}.
    \end{itemize}
    
\end{definition}

We are now ready to provide the sample complexity results for the MESCAL framework. We first define the work complexity of a primal solver $\cM$.

\begin{definition}\label{def:primal_work_at_k} 
    We define the work complexity of a primal solver $\cM$ employed at iteration $k \in \NN_0$ in \Cref{alg:mescal-overview}, denoted by $\cW^{(p)}_{\cM,k}$, as the total number of stochastic augmented Lagrangian gradient evaluations $\nabla \bar{\cL}_{\alpha_k}(x,\lambda_k;\theta)$ and stochastic augmented Lagrangian Hessian vector product evaluations $\nabla^2 \bar{\cL}_{\alpha_k}(x,\lambda_k;\theta)v$ for arbitrary vectors $v$.
\end{definition}

We next state a technical lemma establishing a bound on the sample complexity of \Cref{alg:mescal-overview} in terms of the work complexity of the primal solver $\cM$. We set the dual sample sizes to $|S^{(d)}_k| = \Theta((\epsilong)^{-2})$ for all $k \in \NN_0$, as described in \Cref{rmk:iter-complexity}.

\begin{lemma} \label{lemm:gen-complexity-solver}
    Suppose Assumptions \ref{assn:abs-true-bounds} and \ref{assn:reg-c-nabla-c} hold. Let $\{x_k\}$ be the sequence of iterates generated by \Cref{alg:mescal-overview} with initial iterate $x_{-1} \in \cX$ and initial dual iterate $\lambda_0 \in \RR^m$. Suppose that the dual bound parameters $\gamma_k > 0$ be a bounded sequence satisfying $\sum_{k=0}^{\infty} \gamma_k = C_\gamma < \infty$, and that the dual sample sizes satisfy $|S^{(d)}_k|=\Theta((\epsilong)^{-2})$ for all $k \in \NN_0$. If the error bounds are chosen such that $\epsilong_k = \epsilong < 1$ and $\epsilonH_k = \epsilonH < 1$ for any $k \in \NN_0$, then, for any $k=K \geq K^*$, where $K^*$ is defined in \eqref{eq:iter_comp_log}, the sample complexity $\cW$, as defined in \Cref{def:sample_complexity}, satisfies 
    \begin{equation}\label{eq:sample-complexity-bound}
        \cW = \tilde{\cO}\lt( \lt(\sum_{k=0}^K \cW_{\cM, k}^{(p)}\rt) + (\epsilong)^{-2} \rt),
    \end{equation}
    where $\cW_{\cM,k}^{(p)}$ is the work complexity of solver $\cM$ defined in \Cref{def:primal_work_at_k}.
\end{lemma}

\begin{proof}
    Let $\cW^{(p)}_{\cM}$ denote the total number of stochastic objective gradient, objective Hessian vector product, constraint function, constraint gradient, and constraint Hessian vector product evaluations used by the solver $\cM$ over iterations $k=0,\ldots,K$. Let $\cW^{(d)}$ denote the total number of stochastic constraint function evaluations used in the dual updates over the same iterations. Then, \begin{equation}\label{eq:total-work} 
        \cW = \cW^{(p)}_{\cM} + \cW^{(d)}. 
    \end{equation} 
    We first consider the work associated with the dual updates. At each iteration $k$, the dual update uses $m|S^{(d)}_k|$ stochastic constraint function evaluations. Thus, 
    \begin{align} 
        \cW^{(d)} = \sum_{k=0}^K m|S^{(d)}_k| = \sum_{k=0}^K \Theta\lt((\epsilong)^{-2}\rt) = \cO\lt((K+1)(\epsilong)^{-2}\rt)=\tilde{\cO}\lt((\epsilong)^{-2}\rt), \label{eq:dual-work} 
    \end{align}
    where the last equality is due to $K^*=\cO(\log(1/\epsilong))$. 
    
    Next, {we} consider the work associated with the primal subproblems. From \eqref{eq:sample_gradient_estimator} and \eqref{eq:sample_hessian_estimator}, each stochastic augmented Lagrangian gradient evaluation requires a constant number of objective gradient, constraint function, and constraint gradient evaluations, while each stochastic augmented Lagrangian Hessian vector product evaluation requires a constant number of objective and constraint Hessian vector product evaluations, constraint function evaluations, and constraint gradient evaluations. Since the number of constraint components is $m$, which is fixed, the total number of underlying stochastic oracle evaluations required at each outer iteration is of the same order as $\cW^{(p)}_{\cM,k}$. Therefore, 
    \begin{align} 
        \cW^{(p)}_{\cM} = \cO\lt(\sum_{k=0}^{K} \cW^{(p)}_{\cM,k}\rt). \label{eq:primal-work} 
    \end{align}
    Substituting \eqref{eq:dual-work} and \eqref{eq:primal-work} in \eqref{eq:total-work} completes the proof. 
\end{proof}

We make the following assumption on the theoretical sample complexity guarantees of a general solver $\cM$. 

\begin{assumption} \label{assn:theoretical-sample-size}
    Suppose that Assumptions \ref{assn:abs-true-bounds}, {\ref{assn:reg-c-nabla-c}}, \ref{assn:lip-nabla-f-nabla-c}, \ref{assn:lip-nabla2-f-nabla2-c}, and \ref{assn:var-bounds} hold. For each $k \in \NN_0$, define
    \begin{equation}\label{eq:assm34_bound_total}
        {\mathcal{B}_{\cM, k}} = \min \lt\{ \tfrac{\Delta_k \sigma^{(g)}_{\cL, k}  \sigma^{(H)}_{\cL, k}}{(\epsilong_k)^3}, \tfrac{ \Delta_k\sigma^{(g)}_{\cL, k}(L^{(H)}_{\cL, k})^{0.5} }{(\epsilong_k)^{3.5}}, \tfrac{\Delta_k  \sigma^{(g)}_{\cL, k}L^{(g)}_{\cL, k}}{(\epsilong_k)^4} \rt\} + \tfrac{\Delta_k (\sigma^{(H)}_{\cL, k})^2 (L^{(H)}_{\cL, k})^2}{(\epsilonH_k)^5},
    \end{equation}
    where $\{\Delta_k\}$, $\{L^{(g)}_{\cL, k}\}$, $\{L^{(H)}_{\cL, k}\}$, $\{\sigma^{(g)}_{\cL, k}\}$, $\{\sigma^{(H)}_{\cL, k}\}$ are defined in \Cref{def:complexity_factors}. We assume that one of the following holds:
    \begin{enumerate}[label=(\roman*)] 
        \item Under \textbf{Option I}, $\cM$ returns an iterate satisfying \eqref{cond:exp_primal_prob} with \begin{align*} 
            \cW_{\cM,k}^{(p)} = \tilde{\cO}( {\mathcal{B}_{\cM, k}}).
        \end{align*} 
        \item Under \textbf{Option II}, $\cM$ returns an iterate satisfying \eqref{cond:prob_primal_prob} for any $r\geq 1$, with 
        \begin{align*} 
            \cW_{\cM,k}^{(p)} = \cO\lt( {\mathcal{B}_{\cM, k}} \log^r\lt(\tfrac{1}{1-p_k}\rt) \rt).
        \end{align*}
    \end{enumerate}
\end{assumption}
\begin{remark}
    We make the following remarks {regarding \Cref{assn:theoretical-sample-size}}.
    \begin{itemize}
        \item %
        We are unaware of a stochastic second-order solver $\cM$ with established work complexity guarantees for the general convex set $\cX$ considered in this paper. 
        We therefore formulate work complexity bounds %
        guided by the lower bound result in \cite[Eq.~(13)]{arjevani_secondorder_2020}. In \cite{arjevani_secondorder_2020}, Arjevani et al. %
        established lower bounds %
        for computing an %
        $(\epsilong, \epsilonH)$-accurate stationary point satisfying probabilistic conditions \eqref{cond:prob_primal_prob}, %
        for $p\geq0.5$ and $\cX = \RR^d$. These lower bounds are %
        consistent with the upper bounds %
        available for SG-HV-NC \cite[Algorithm 4]{arjevani_secondorder_2020}, Natasha2 \cite{allen-zhu_natasha_2018}, and SPIDER \cite{fang_spider_2018}. 
        The corresponding MESCAL complexity results for these methods are established in \Cref{sec:active_solvers}, under the restriction $\cX=\RR^d$.

        \item The lower bound in \cite[Eq.~(13)]{arjevani_secondorder_2020} does not characterize the dependence on the desired success probability $p$. We therefore include a polylogarithmic dependence on the failure probability $1-p$, so that the work complexity increases as $p$ tends to one. Such dependence is common in high-probability stochastic optimization guarantees; see, for example, \cite{cutkosky_highprobability_2021,ghadimi_stochastic_2013,roosta-khorasani_subsampled_2019,sadiev_highprobability_2023}. Furthermore, the dependence on initial optimality gap, Lipschitz constants, and variance bounds is essential within the MESCAL framework for accurate calculations of sample complexity bounds, as these terms may contribute to increasing the complexity bound. {Accordingly, these constants are included explicitly in the work complexity bounds. }  

        \item Although the %
        lower bound in \cite[Eq. (13)]{arjevani_secondorder_2020} is stated for $\cX = \RR^d$, several proximal methods attain work complexities comparable to those of their unconstrained counterparts; see, e.g., \cite{bollapragada_convergence_2026,li_simple_2018,liang_almost_2024}. %
        In addition, high-probability guarantees often differ from their expectation-based counterparts only by logarithmic factors under suitable assumptions; see \cite{cutkosky_highprobability_2021,sadiev_highprobability_2023}.
        \item The %
        {method }proposed and analyzed in \cite{berahas_exploiting_2026} %
        computes an $(\epsilong, \epsilonH)$-accurate %
        second-order stationary point %
        under $\cX = \RR^d$ {satisfying \eqref{cond:exp_primal_prob}}, but its work complexity is not established therein. We therefore also analyze MESCAL under \textbf{Option I}, assuming that the solver $\cM$ satisfies the expected work complexity bound in \Cref{assn:theoretical-sample-size}. The resulting sample complexity guarantees consequently apply to any solver satisfying these conditions. %
        
      \end{itemize}  
\end{remark}

\begin{theorem} \label{thm:theoretical-solver}
    
    Suppose Assumptions \ref{assn:abs-true-bounds}, \ref{assn:reg-c-nabla-c}, \ref{assn:lip-nabla-f-nabla-c}, \ref{assn:lip-nabla2-f-nabla2-c}, \ref{assn:var-bounds}, and \ref{assn:theoretical-sample-size} hold. Let $\{x_k\}$ be the sequence of iterates generated by \Cref{alg:mescal-overview} with initial iterate $x_{-1} \in \cX$ and initial dual iterate $\lambda_0 \in \RR^m$. Suppose that the dual bound parameters $\gamma_k > 0$ be a bounded sequence satisfying $\sum_{k=0}^{\infty} \gamma_k = C_\gamma < \infty$, and that the dual sample sizes satisfy $|S^{(d)}_k|=\Theta((\epsilong)^{-2})$ for all $k \in \NN_0$. If the error bounds are chosen such that $\epsilong_k = \epsilong < 1$ and $\epsilonH_k = \epsilonH < 1$ for any $k \in \NN_0$, then, for any $k=K \geq K^*$, where $K^*$ is defined in \eqref{eq:iter_comp_log}, the sample complexity $\cW$, as defined in \Cref{def:sample_complexity}, satisfies the following guarantees: 
    \begin{enumerate}[label=(\roman*)]
        \item under \textbf{(Option I)}, when \Cref{assn:theoretical-sample-size}$\mathrm{(i)}$ holds, the expected sample complexity to obtain an $(\epsilong,\epsilonH)$-accurate expected second-order stationary point $\tilde{x}:=x_K$ satisfying \Cref{def:opt_cond_exp} for solving \eqref{eq:gen_formula}, with $\tilde{\lambda} := \lambda_K + \alpha_K c(x_K)$, is %
        \begin{align*}
            \EE[\cW] = \tilde{\cO}((\epsilong)^{-5} + (\epsilong)^{-4} (\epsilonH)^{-5}).
        \end{align*}

        \item under \textbf{(Option II)} with $K = K^*$, when \Cref{assn:theoretical-sample-size}$\mathrm{(ii)}$ holds and $p_k = p^{1/K}$, the sample complexity to obtain an $(\epsilong, \epsilonH)$-accurate $p$-probabilistic second-order stationary point $\tilde{x}:=x_K$ satisfying \Cref{def:opt_cond_prob} for solving \eqref{eq:gen_formula}, with $\tilde{\lambda} := \lambda_K + \alpha_K c(x_K)$, is%
        \begin{align*}
            \cW = \tilde{\cO}((\epsilong)^{-5} + (\epsilong)^{-4} (\epsilonH)^{-5}),
        \end{align*}
        with probability at least $p$.
    \end{enumerate}
\end{theorem}

\begin{proof}
    By Lemmas \ref{lemm:2-lipschitz-grad-aug-lag}-\ref{lemm:5-finite-hess-variance}, we have,
    \begin{equation} \label{eq:aug-constants-complexity-bounds}
    \begin{aligned}
        L^{(g)}_{\cL, k} &\leq L^{(g)}_1 + \alpha_k L^{(g)}_2 = \cO(\alpha_k), \quad L^{(H)}_{\cL, k} \leq L^{(H)}_1 + \alpha_k L^{(H)}_2 = \cO(\alpha_k), \\
        \sigma^{(g)}_{\cL, k} &\leq \sigma^{(g)}_1 + \alpha_k \sigma^{(g)}_2 = \cO(\alpha_k), \quad\sigma^{(H)}_{\cL, k} \leq \sigma^{(H)}_1 + \alpha_k \sigma^{(H)}_2 = \cO(\alpha_k),
    \end{aligned}
    \end{equation}
    where the equalities are due to the fact that $\alpha_k$ keeps growing with $k$, whereas $L^{(g)}_1$, $L^{(g)}_2$, $L^{(H)}_1$, $L^{(H)}_2$, $\sigma^{(g)}_1$, $\sigma^{(g)}_2$, $\sigma^{(H)}_1$, and $\sigma^{(H)}_2$ are constant with respect to $k$. Let 
    \begin{equation} \label{eq:work_bound_1}
        \mathscr{C}_{\cM, k} = \min \lt\{ \tfrac{ \sigma^{(g)}_{\cL, k}  \sigma^{(H)}_{\cL, k}}{(\epsilong_k)^3}, \tfrac{ \sigma^{(g)}_{\cL, k}(L^{(H)}_{\cL, k})^{0.5} }{(\epsilong_k)^{3.5}}, \tfrac{ \sigma^{(g)}_{\cL, k}L^{(g)}_{\cL, k}}{(\epsilong_k)^4} \rt\} + \tfrac{(\sigma^{(H)}_{\cL, k})^2 (L^{(H)}_{\cL, k})^2}{(\epsilonH_k)^5}.
    \end{equation}
    Therefore, from \eqref{eq:assm34_bound_total}, we have $\mathcal{B}_{\cM, k}=\Delta_k\mathscr{C}_{\cM, k}$. Considering $\mathscr{C}_{\cM, k}$ %
    and using \eqref{eq:aug-constants-complexity-bounds}, $\epsilong_k = \epsilong$, and $\epsilonH_k = \epsilonH$, we obtain 
    \begin{align}
        \mathscr{C}_{\cM, k} &= \min \lt\{ \tfrac{ \sigma^{(g)}_{\cL, k}  \sigma^{(H)}_{\cL, k}}{(\epsilong_k)^3}, \tfrac{ \sigma^{(g)}_{\cL, k}(L^{(H)}_{\cL, k})^{0.5} }{(\epsilong_k)^{3.5}}, \tfrac{ \sigma^{(g)}_{\cL, k}L^{(g)}_{\cL, k}}{(\epsilong_k)^4} \rt\} + \tfrac{(\sigma^{(H)}_{\cL, k})^2 (L^{(H)}_{\cL, k})^2}{(\epsilonH_k)^5} \notag \\
        & = \cO \lt( \min \lt\{ \tfrac{(\alpha_k)^2}{(\epsilong)^3}, \tfrac{ (\alpha_k)^{1.5} }{(\epsilong)^{3.5}}, \tfrac{(\alpha_k)^2}{(\epsilong)^4} \rt\} + \tfrac{(\alpha_k)^4}{(\epsilonH)^5} \rt). \label{eq:work-scrC-bound}
    \end{align}
    Case (i): Under \textbf{Option I}, from \Cref{assn:theoretical-sample-size}(i), we establish that the iterate $x_k$ satisfies Condition \eqref{cond:exp_primal_prob} in \Cref{alg:mescal-overview}. Therefore, by \Cref{thm:outer-iteration-complexity}, it holds that $\tilde{x} := x_K$ is an $(\epsilong, \epsilonH)$-accurate expected second-order stationary point as defined in \Cref{def:opt_cond_exp}.
    
    {At} iteration $k=0$, the work complexity of solver $\cM$ is given by
    \begin{align}
        \cW^{(p)}_{\cM,0} =\tilde{\cO}(\Delta^*_0 \mathscr{C}_{\cM, 0}) = \tilde{\cO} \lt(  (\epsilong)^{-3} + (\epsilonH)^{-5} \rt), \label{eq:exp-k-zero-primal-work}
    \end{align} 
    where $\Delta^*_0$ is defined in \Cref{lemm:1-lower-bound-finite}. 
    
    Using %
    {\eqref{eq:exp_delta_bnd} from \Cref{lemm:1-lower-bound-finite}(i), }in expectation, the total work complexity over iterations $1, 2, \cdots, K$ for solver $\cM$ is
    \begin{align}
        \sum_{k=1}^K \EE[\cW^{(p)}_{\cM,k}] &=\tilde{\cO}\lt(\sum_{k=1}^K\EE[\Delta_k \mathscr{C}_{\cM, k}]\rt) =\tilde{\cO}\lt(\sum_{k=1}^K\EE[\Delta_k] \mathscr{C}_{\cM, k}\rt) \nonumber \\
        &=\tilde{\cO}\lt(\sum_{k=1}^K \min \lt\{ \tfrac{(\alpha_k)^2}{(\epsilong)^3}, \tfrac{ (\alpha_k)^{1.5} }{(\epsilong)^{3.5}}, \tfrac{(\alpha_k)^2}{(\epsilong)^4} \rt\} + \tfrac{(\alpha_k)^4}{(\epsilonH)^5}\rt) \nonumber \\        
        & = \tilde{\cO} \lt( \min \lt\{ \tfrac{\beta^{2K}}{(\epsilong)^3}, \tfrac{ \beta^{1.5K}}{(\epsilong)^{3.5}}, \tfrac{\beta^{2K}}{(\epsilong)^4} \rt\} + \tfrac{\beta^{4K}}{(\epsilonH)^5} \rt) \notag \\
        & = \tilde{\cO} \lt( (\epsilong)^{-5} + (\epsilong)^{-4} (\epsilonH)^{-5} \rt), \label{eq:exp-primal-work}
    \end{align}
    where the last equality is due to the fact that $\beta^K = \cO((\epsilong)^{-1})$. Combining \eqref{eq:exp-k-zero-primal-work} and \eqref{eq:exp-primal-work}, and substituting in \eqref{eq:sample-complexity-bound} yields the result. 
    
    Case (ii): Under \textbf{Option II}, from \Cref{assn:theoretical-sample-size}(ii), we establish that the iterate $x_k$ satisfies Condition \eqref{cond:prob_primal_prob} in \Cref{alg:mescal-overview}. Therefore, by \Cref{thm:outer-iteration-complexity}, it holds that $\tilde{x} := x_K$ is an $(\epsilong, \epsilonH)$-accurate $p$-probabilistic second-order stationary point as defined in \Cref{def:opt_cond_prob}, since $p \leq p^{1 / K}$. %

    {At} iteration $k=0$, the work complexity of solver $\cM$ is given by
    \begin{align}
        \cW^{(p)}_{\cM,0} =\cO\lt(\Delta^*_0 \mathscr{C}_{\cM, 0} \log^r \lt( \tfrac{1}{1-p^{1/K}} \rt) \rt) = \tilde{\cO} \lt(  (\epsilong)^{-3} + (\epsilonH)^{-5} \rt), \label{eq:prob-k-zero-primal-work}
    \end{align} 
    where $\Delta^*_0$ is defined in \Cref{lemm:1-lower-bound-finite} and  the final equality follows from \Cref{lemm:log-prob-term}.

    Let $D:=\cap_{k=1}^K \{\Delta_k \leq \Delta^*\}$. From \eqref{eq:delta-prob-overall-bound_init}, we have
    \begin{align}\label{eq:delta-prob-overall-bound}
        \PP(D 
        ) \geq 
        {\prod_{k=0}^{K-1}p_k = }
        \left(p^{1/K}\right)^K = p.
    \end{align}
    
    Now, define the event
    {
    \begin{align*}
        E_{\cW} := \lt\{
                \sum_{k=1}^K \cW_{\cM,k}^{(p)} =
                \tilde{\cO} \lt( (\epsilong)^{-5} + (\epsilong)^{-4} (\epsilonH)^{-5} \rt)
                \rt\}.
    \end{align*}}
    Suppose that event $D$ holds. Then, by \Cref{assn:theoretical-sample-size}(ii), {\eqref{eq:work_bound_1}}, and \Cref{lemm:log-prob-term}, we obtain
    {
    \begin{align*}
        \sum_{k=1}^K \cW^{(p)}_{\cM,k}
        &=
        \sum_{k=1}^K
        \tilde{\cO}
        \lt(
            \Delta_k \mathscr{C}_{\cM, k}
            \log^{r} \lt( \frac{1}{1-p^{1/K}} \rt)
        \rt) =
        \sum_{k=1}^K
        \tilde{\cO}
        \lt(
            \Delta^* \mathscr{C}_{\cM, k}
        \rt) =
        \tilde{\cO}
        \lt(
            (\epsilong)^{-5} + (\epsilong)^{-4} (\epsilonH)^{-5}
        \rt),
    \end{align*}}
    where the final inequality follows from \eqref{eq:exp-primal-work}. Consequently, $D \subseteq E_{\cW}$. Therefore, by \eqref{eq:delta-prob-overall-bound},
    {
    \begin{align}
        \PP \lt( \sum_{k=1}^K \cW_{\cM,k}^{(p)} \leq
            \tilde{\cO} \lt( (\epsilong)^{-5} + (\epsilong)^{-4} (\epsilonH)^{-5} \rt)
            \rt) \geq
        \PP (D 
        ) \geq p.
        \label{eq:work-1-K-primal-prob}
    \end{align}}
    
    Since the bound in \eqref{eq:prob-k-zero-primal-work} holds almost surely, combining it with \eqref{eq:work-1-K-primal-prob} and \eqref{eq:sample-complexity-bound} yields
    \begin{equation}
        \PP\lt( \cW = \tilde{\cO}((\epsilong)^{-5} + (\epsilong)^{-4} (\epsilonH)^{-5}) \rt) \geq p. \label{eq:complexity-general-prob}
    \end{equation}

\end{proof}

%% file: 4-existing-solvers.tex
\section{Existing Primal Solvers}  \label{sec:active_solvers}
In this section, we consider existing solvers for solving the primal subproblem at each outer iteration $k \in \NN_{0}$ of \Cref{alg:mescal-overview} and producing iterates $x_k$ that satisfy the inexactness conditions specified therein. 
For the general setting considered in this paper, we are unaware of stochastic second-order solvers with theoretical complexity guarantees for obtaining iterates that satisfy expectation conditions~\eqref{cond:exp_primal_prob} or probabilistic conditions~\eqref{cond:prob_primal_prob}. However, theoretical complexity results are available for several stochastic second-order solvers under more specialized settings. We consider three such solvers: SG-HV-NC \cite[Algorithm 4]{arjevani_secondorder_2020}, Natasha2 \cite{allen-zhu_natasha_2018}, and SPIDER \cite{fang_spider_2018}. We describe their implementation within the MESCAL framework and establish the corresponding sample complexity results. For these three solvers, the existing theoretical results are established under two restrictions: (i) the subproblem is unconstrained, i.e., $\cX=\RR^d$, and (ii) the solver obtains an $(\epsilong_k,\epsilonH_k)$-accurate iterate satisfying the probabilistic conditions in  \eqref{cond:prob_primal_prob}. Therefore, our sample complexity results for these solvers establish the complexity of obtaining an $(\epsilong, \epsilonH)$-accurate $p$-probabilistic second-order stationary point $\tilde{x}$ satisfying \eqref{eq:opt_cond_prob} under the restriction $\cX=\RR^d$. However, we note that the theoretical analysis framework established in \Cref{sec:outer_loop} and \Cref{sec:inner_loop} ensures that any new second-order solvers, that can handle general closed and convex constraint sets $\cX$ and obtain iterates satisfying \eqref{cond:exp_primal_prob}, can be analyzed within the MESCAL framework to establish the corresponding sample complexity results.

The methods considered share a common framework in that they employ two distinct types of steps to update the primal iterate toward an $(\epsilong, \epsilonH)$-accurate point. The first is a gradient-based first-order step that leverages variance-reduction techniques to control the error in the gradient estimator. The second is a negative-curvature-based second-order step involving movement along directions of negative curvature, i.e., directions along which the Hessian exhibits sufficiently negative curvature.
Although all three methods employ randomized negative-curvature steps, they differ in their negative-curvature detection procedures, the rule used to select between a first-order and a second-order step, and the variance-reduction technique used to construct the stochastic gradient estimators.
In the following subsections, we describe these algorithms and establish the corresponding sample complexity results for their use within the MESCAL framework.

%% file: 4-SGD-NCS-RVR.tex
\subsection{Primal Solver: SG-HV-NC} \label{subsec:sgd_ncs}

The SG-HV-NC (stochastic gradient method with variance reduction using Hessian vector products and negative curvature steps) \cite[Algorithm 4]{arjevani_secondorder_2020} is a second-order solver that alternates between stochastic gradient (SG) steps and negative-curvature %
steps. We implement this solver as the primal solver $\cM$ under \textbf{Option II} of \Cref{alg:mescal-overview}, with the step-by-step implementation provided in \Cref{alg:sgd-ncs}. At each inner iteration $t \in \NN$, a Bernoulli random variable determines whether an SG step or a negative-curvature step is performed. The SG steps use the variance-reduced gradient estimator HVP-RVR-GE (Hessian vector product-based recursive variance reduced gradient estimator), described in \Cref{alg:sgd-hvp}, while the negative-curvature steps use Oja's algorithm, as described in \cite{allen-zhu_natasha_2018}, to efficiently identify a direction of negative curvature. For it's implementation in \Cref{step:sgd_oja} of \Cref{alg:sgd-ncs}, we refer to \cite[Lemma 16]{arjevani_secondorder_2020}. Further details and theoretical analysis of SG-HV-NC and Oja's algorithm can be found in \cite{allen-zhu_natasha_2018, allen-zhu_follow_2017, arjevani_secondorder_2020, shamir_convergence_2016}.

We next state the input-output guarantee of the Oja's algorithm that we use in our analysis. In particular, given a point $x$ and a stochastic function class $h$, Oja either returns a certificate that the Hessian at $x$ has no sufficiently negative curvature or returns a unit-norm direction of sufficiently negative curvature with some probability $1 - p_f$.

\begin{definition} \label{def:oja}
    The Oja's algorithm takes as input a point $x \in \RR^d$; a stochastic function class $h$ satisfying 1) a finite initial optimality gap, 2) Lipschitz continuous stochastic gradient, 3) Lipschitz continuous Hessian, 4) stochastic gradient with a uniform error bound in expectation, and 5) stochastic Hessian with a uniform error bound almost surely; a precision parameter $\epsilonH > 0$; and a failure probability $p_f \in (0, 1)$. It outputs $u \in \RR^d \cup \{ \perp \}$ such that, with probability at least $1 - p_f$, either
    \begin{enumerate}
        \item $u = \perp$, and $\sigma_{\min} (\nabla^2 h(x)) \geq -2 \epsilonH $;
        \item $u \neq \perp$, $\|u\|=1$ and $u^\top \nabla^2 h(x) u \leq -\epsilonH$.
    \end{enumerate}
\end{definition}

\begin{algorithm2e}[H]
    \caption{SG-HV-NC}\label{alg:sgd-ncs}
    \DontPrintSemicolon
    
    \textbf{Input (for all $k \in \NN_0$):} function class $\cL_{\alpha_k}$; iterate $x_{k-1}$; dual variable $\lambda_k$; optimality gap $\Delta_k$; smoothness constants $L^{(g)}_{\cL, k}$, $L^{(H)}_{\cL, k}$; variance constants $\sigma^{(g)}_{\cL, k}$, $\sigma^{(H)}_{\cL, k}$; primal subproblem tolerances $\epsilong_k$, $\epsilonH_k > 0$ \;
    
    set $\eta_k = \min \lt\{ \tfrac{\epsilonH_k}{L^{(H)}_{\cL, k}}, \tfrac{1}{2\sqrt{(L^{(g)}_{\cL, k})^2 + (\sigma^{(H)}_{\cL, k})^2 + \epsilong_k L^{(H)}_{\cL, k}}} \rt\}$, $T_k = \lt\lceil \tfrac{20\Delta_k (L^{(H)}_{\cL, k})^2}{(\epsilonH_k)^3} + \tfrac{2\Delta_k}{\eta_k (\epsilong_k)^2} \rt\rceil$, $q_k = \tfrac{(\epsilonH_k)^3}{(\epsilonH_k)^3 + 10\Delta_k (L^{(H)}_{\cL, k})^2 \eta_k (\epsilong_k)^2}$, $\delta_k = \tfrac{\epsilonH_k}{40^2 L^{(H)}_{\cL, k}}$ \;
    
    set $b^{(g)}_k = \min\lt\{1, \tfrac{\eta_k \sqrt{(\sigma^{(H)}_{\cL, k})^2 + \epsilong_k L^{(H)}_{\cL, k}}}{\sigma^{(g)}_{\cL, k}} \rt\}$ and $b^{(H)}_k = \min\lt\{1, \tfrac{\epsilonH_k\sqrt{(\sigma^{(H)}_{\cL, k})^2 + \epsilong_k L^{(H)}_{\cL, k}}}{\sigma^{(g)}_{\cL, k} L^{(H)}_{\cL, k}} \rt\}$ \;
    
    initialize $x_{k, 0} = x_{k - 1}$, $x_{k, 1} = x_{k - 1}$ \; 
    initialize $g_{k, 1} = \text{HVP-RVR-GE}(x_{k, 1}, x_{k, 0}, \emptyset, b^{(g)}_k, \epsilong_k, L^{(H)}_{\cL, k}, \sigma^{(g)}_{\cL, k}, \sigma^{(H)}_{\cL, k}, \lambda_k, \cL_{\alpha_k})$ \;
    
    \For{$t = 1$ \KwTo $T_k$}{
        sample $Q_{1, k, t} \sim \text{Bernoulli}(q_k)$ \;
        \eIf{$Q_{1, k, t} = 1$}{
            set $x_{k, t+1} = x_{k, t} - \eta_k g_{k, t}$ \tcp*{first-order step}
            set $g_{k, t+1} = \text{HVP-RVR-GE}(x_{k, t+1}, x_{k, t}, g_{k, t}, b^{(g)}_k, \epsilong_k, L^{(H)}_{\cL, k}, \sigma^{(g)}_{\cL, k}, \sigma^{(H)}_{\cL, k}, \lambda_k, \cL_{\alpha_k})$ \;
        }{
            set $u_{k, t} = \text{Oja}(x_{k, t}, \cL_{\alpha_k}(\cdot, \lambda_k), 2\epsilonH_k, \delta_k)$ \label{step:sgd_oja} \tcp*{negative curvature search}
            \eIf{$u_{k, t} = \perp$}{
                set $x_{k, t+1} = x_{k, t}$ \;
                set $g_{k, t+1} = g_{k, t}$ \;
            }{
                sample $r_{k, t} \sim \text{Uniform}(\{-1, 1\})$ \;%
                set $x_{k, t+1} = x_{k, t} + \tfrac{\epsilonH_k r_{k, t}}{L^{(H)}_{\cL, k}} u_{k, t}$ \tcp*{second-order step}%
                set $g_{k, t+1} = \text{HVP-RVR-GE}(x_{k, t+1}, x_{k, t}, g_{k, t}, b^{(H)}_k, \epsilong_k, L^{(H)}_{\cL, k}, \sigma^{(g)}_{\cL, k}, \sigma^{(H)}_{\cL, k}, \lambda_k, \cL_{\alpha_k})$ \;
            }
        }
    }
    sample $x_k \sim \mathrm{Uniform}(\{x_{k, t}\}_{t=1}^{T_k})$
    
    \Return $x_{k}$ \;
\end{algorithm2e}

\begin{algorithm2e}[H]
    \caption{HVP-RVR-GE} \label{alg:sgd-hvp}
    \DontPrintSemicolon
    \textbf{Input:} Iterates $\hat{x}_1, \hat{x}_0$; gradient estimators $g_0$; estimator selection probability $b$; error tolerance $\epsilon$; smoothness constant $L^{(H)}$; variance constants $\sigma^{(g)}$, $\sigma^{(H)}$; dual variable $\lambda$; function class $\cL_{\alpha}$ \;
    set $J = \lt\lceil\tfrac{5 ((\sigma^{(H)})^2 + L^{(H)} \epsilon)}{b \epsilon^2} \|\hat{x}_1 - \hat{x}_0 \|^2 \rt\rceil$, $B = \lt\lceil\tfrac{5 (\sigma^{(g)})^2}{\epsilon^2}\rt\rceil$ \; 
    
    sample $Q_2 \sim \text{Bernoulli}(b)$ \; 
    \If{$Q_2=1$ $\boldsymbol{\mathrm{or}}$ $g_0 = \emptyset$}{
        sample $\theta_i \sim \cP_\theta, \quad \forall i \in \{1, 2, \cdots, B\}$ \;
        set $g_1 = \tfrac{\sum_{\theta_i=1}^B \nabla \bar{\cL}_{\alpha} (\hat{x}_1, \lambda; \theta)}{B}$
    }
    \Else{
        set $\overline{x}_j = \tfrac{j}{J} \hat{x}_1 + \lt( 1 - \tfrac{j}{J} \rt) \hat{x}_0, \quad \forall j \in \{0, 1, \cdots, J\}$ \; 
        sample $\theta_j \sim \cP_\theta, \quad \forall j \in \{0, 1, \cdots, J\}$ \;
        set $g_1 = g_0 + \sum_{j=1}^{J} \nabla^2 \bar{\cL}_{\alpha} (\overline{x}_{j-1}, \lambda; \theta_j)(\overline{x}_{j}-\overline{x}_{j-1})$ \; 
    }
    \Return $g_1$
\end{algorithm2e}

\begin{remark}
    Second-order stochastic optimization algorithms typically involve several parameters. For \Cref{alg:sgd-ncs}, we briefly describe the roles of the main parameters: 
    \begin{itemize}
        \item $\eta_k$: The step size used for the first-order %
        steps.
        
        \item $T_k$: The prescribed number of inner iterations performed by the SG-HV-NC solver. 
        
        \item $q_k$: The probability parameter of the Bernoulli random variable $Q_{1,k,t}$, which determines whether a first-order %
        step ($Q_{1,k,t}=1$) or a negative-cuvature step ($Q_{1,k,t}=0$) is performed at inner iteration $t$.
        
        \item $\delta_k$: The failure probability parameter for Oja's algorithm which controls the probability with which the algorithm correctly identifies a direction of negative curvature or certifies the absence of sufficiently negative curvature.

        \item $b^{(g)}_k$: The probability 
        parameter of Bernoulli random variable $Q_2$ in \Cref{alg:sgd-hvp} when HVP-RVR-GE is called during a first-order step. It determines whether a subsampled gradient estimator or a recursive variance-reduced gradient estimator is used.
        \item $b^{(H)}_k$: The probability parameter of Bernoulli random variable $Q_2$ in \Cref{alg:sgd-hvp} when HVP-RVR-GE is called during a second-order step. It determines whether a subsampled gradient estimator or a recursive variance-reduced gradient estimator is used.
    \end{itemize}
    The values of these parameters are specified in \Cref{alg:sgd-ncs} to ensure that the resulting iterate $x_k$ is an $(\epsilong_k,\epsilonH_k)$-accurate iterate satisfying the probabilistic conditions in \eqref{cond:prob_primal_prob} at each outer iteration $k \in \NN$.
\end{remark}

We now state the work complexity at each iteration $k \in \NN_0$, as defined in \Cref{def:primal_work_at_k}, of the SG-HV-NC algorithm, denoted by $\cW^{(p)}_{SG, k}$, when used as the primal solver in \Cref{alg:mescal-overview}.

\begin{lemma}[Work complexity of SG-HV-NC]%
\label{lemm:iter-sample-complexity-sgd-ncs}
    Suppose $\cX = \RR^d$ and Assumptions \ref{assn:abs-true-bounds}, \ref{assn:reg-c-nabla-c}, \ref{assn:lip-nabla-f-nabla-c}, \ref{assn:lip-nabla2-f-nabla2-c}, and \ref{assn:var-bounds} hold. Let \Cref{alg:sgd-ncs} be the solver $\cM$ under \textbf{Option II}, with terms $L^{(g)}_{\cL, k}$, $L^{(H)}_{\cL, k}$, $\sigma^{(g)}_{\cL, k}$, $\sigma^{(H)}_{\cL, k}$, and $\Delta_k$ defined as per \Cref{def:complexity_factors}. If the primal subproblem tolerance inputs satisfy $\epsilong_k \leq \min \{1, \sigma^{(g)}_{\cL, k}, (\Delta_k L^{(g)}_{\cL, k})^{0.5} \}$, and $\epsilonH_k \leq \min \{1, \sigma^{(H)}_{\cL, k}, \allowbreak L^{(g)}_{\cL, k}, (\epsilong_k L^{(H)}_{\cL, k})^{0.5}\}$, then, at every iteration $k\in\NN_0$, \Cref{alg:sgd-ncs} returns an iterate $x_k$ satisfying the probabilistic conditions~\eqref{cond:prob_primal_prob} with probability at least $p_k = 5/8$, with work complexity   
    \begin{align*}
        \cW^{(p)}_{SG, k} = %
        \tilde{\cO} \lt( 
        \tfrac{\Delta_k \sigma^{(g)}_{\cL, k} \sigma^{(H)}_{\cL, k}}{(\epsilong_k)^3} + 
        \tfrac{\Delta_k L^{(H)}_{\cL, k} \sigma^{(g)}_{\cL, k} \sigma^{(H)}_{\cL, k}}{(\epsilonH_k)^2(\epsilong_k)^2} + 
        \tfrac{\Delta_k (L^{(H)}_{\cL, k})^2 (\sigma^{(H)}_{\cL, k} + L^{(g)}_{\cL, k})^2}{(\epsilonH_k)^5} + 
        \tfrac{\Delta_k L^{(g)}_{\cL, k}}{(\epsilong_k)^2}
        \rt).
    \end{align*}
\end{lemma}

The proof of this lemma is provided in \Cref{sec:appendix_lemm:sgd}. We now provide the sample complexity of the MESCAL framework with SG-HV-NC as the primal solver $\cM$.

\begin{theorem}[{Sample Complexity of MESCAL-SG-HV-NC}]%
\label{thm:mescal_sgd}
    Suppose {$\cX = \RR^d$}, and Assumptions \ref{assn:abs-true-bounds}, \ref{assn:reg-c-nabla-c}, \ref{assn:lip-nabla-f-nabla-c}, \ref{assn:lip-nabla2-f-nabla2-c} {and} \ref{assn:var-bounds} hold. %
    Let {$\{x_k\}$ be the sequence of iterates generated by \Cref{alg:mescal-overview} with }initial iterate $x_{-1} \in \cX$ and initial dual iterate $\lambda_0 \in \RR^m$. Suppose that the dual bound parameters $\gamma_k > 0$ be a bounded sequence satisfying $\sum_{k=0}^{\infty} \gamma_k = C_\gamma < \infty$, and that the dual sample sizes satisfy $|S^{(d)}_k| = \Theta((\epsilong)^{-2})$ for all $k \in \NN_0$. {Let} \Cref{alg:sgd-ncs} %
    {be the }solver $\cM$ under \textbf{Option II}. Suppose {the} primal subproblem tolerances $\epsilong_k$, $\epsilonH_k$ satisfy $\epsilong_k = \epsilong \leq \min \{1, \sigma^{(g)}_{\cL, k}, (\Delta_k L^{(g)}_{\cL, k})^{0.5} \}$ and $\epsilonH_k = \epsilonH \leq \min \{1, \sigma^{(H)}_{\cL, k}, \allowbreak L^{(g)}_{\cL, k}, (\epsilong_k L^{(H)}_{\cL, k})^{0.5}\}$ for all $k \in \NN_0$.
    Then, %
    {for any }$k = K \geq K^*$, where $K^*$ is %
    {defined in \eqref{eq:iter_comp_log}}, the iterate $\tilde{x}:= x_K$ is an $(\epsilong, \epsilonH)$-accurate $p$-probabilistic second-order stationary point satisfying \eqref{eq:opt_cond_prob} as per \Cref{def:opt_cond_prob}, %
    {with }$p=5/8$ %
    {and} $\tilde{\lambda} := \lambda_K + \alpha_K c(x_K)$. {Furthermore}, %
    the sample complexity of the algorithm, as %
    {defined in }\Cref{def:sample_complexity}, to obtain $\tilde{x}:= x_K$, is
    \begin{enumerate}[label=(\roman*)]
        \item \textbf{(Deterministic {Sample} Complexity):} 
        $$\cW = \tilde{\cO}( ( \epsilong )^{-6} ( \epsilonH )^{-2} + ( \epsilong )^{-5} ( \epsilonH )^{-5} ).$$

        \item \textbf{(Probabilistic {Sample} Complexity):}  
        $$\cW = \tilde{\cO}( ( \epsilong )^{-5} ( \epsilonH )^{-2} + ( \epsilong )^{-4} ( \epsilonH )^{-5})$$
        with probability at least $p^K$.
    \end{enumerate}
\end{theorem}

\begin{proof} 
    From \Cref{lemm:iter-sample-complexity-sgd-ncs}, we establish that the iterate $x_k$ satisfies \eqref{cond:prob_primal_prob} in \Cref{alg:mescal-overview}. {Therefore, }by \Cref{thm:outer-iteration-complexity}, it holds that $\tilde{x} := x_K$ is an $(\epsilong, \epsilonH)$-accurate $p$-probabilistic second-order stationary point %
    {as defined in }\Cref{def:opt_cond_prob}, with $p=5/8$ and %
    $\tilde{\lambda} := \lambda_K + \alpha_K c(x_K)$.

    At iteration $k=0$, using \Cref{lemm:iter-sample-complexity-sgd-ncs}, \eqref{eq:aug-constants-complexity-bounds}, \eqref{eq:det_delta_bnd} in \Cref{lemm:1-lower-bound-finite}, and setting $\epsilong_0 = \epsilong$, $\epsilonH_0 = \epsilonH$, the work complexity of solver $\cM$ is given by
    \begin{align}
        \cW^{(p)}_{SG, 0} &= \tilde{\cO} \lt( 
        \tfrac{\Delta^*_0 \sigma^{(g)}_{\cL, 0} \sigma^{(H)}_{\cL, 0}}{(\epsilong)^3} + 
        \tfrac{\Delta^*_0 L^{(H)}_{\cL, 0} \sigma^{(g)}_{\cL, 0} \sigma^{(H)}_{\cL, 0}}{(\epsilonH)^2(\epsilong)^2} + 
        \tfrac{\Delta^*_0 (L^{(H)}_{\cL, 0})^2 (\sigma^{(H)}_{\cL, 0} + L^{(g)}_{\cL, 0})^2}{(\epsilonH)^5} + 
        \tfrac{\Delta^*_0 L^{(g)}_{\cL, 0}}{(\epsilong)^2}
        \rt) \notag \\
        & = \tilde{\cO} ( (\epsilong)^{-3} + (\epsilong)^{-2} (\epsilonH)^{-2} + (\epsilonH)^{-5} ). \label{eq:0-complexity-sgd}
    \end{align}

    Case (i): {Using \Cref{lemm:iter-sample-complexity-sgd-ncs}, \eqref{eq:aug-constants-complexity-bounds}, \eqref{eq:det_delta_bnd} in \Cref{lemm:1-lower-bound-finite}, and setting $\epsilong_k = \epsilong$, $\epsilonH_k = \epsilonH$, the total work complexity of solver $\cM$ over iterations $k = 1, 2, \cdots, K$  is given by} 
    
    \begin{align}
        \sum_{k=1}^K\cW^{(p)}_{SG, k} &= \sum_{k=1}^K \tilde{\cO} \lt( 
        \tfrac{\Delta_k \sigma^{(g)}_{\cL, k} \sigma^{(H)}_{\cL, k}}{(\epsilong_k)^3} + 
        \tfrac{\Delta_k L^{(H)}_{\cL, k} \sigma^{(g)}_{\cL, k} \sigma^{(H)}_{\cL, k}}{(\epsilonH_k)^2(\epsilong_k)^2} + 
        \tfrac{\Delta_k (L^{(H)}_{\cL, k})^2 (\sigma^{(H)}_{\cL, k} + L^{(g)}_{\cL, k})^2}{(\epsilonH_k)^5} + 
        \tfrac{\Delta_k L^{(g)}_{\cL, k}}{(\epsilong_k)^2}
        \rt) \nonumber \\
        &= \sum_{k=1}^K \tilde{\cO} \lt( \tfrac{\beta^{3k}}{(\epsilong)^3} + \tfrac{\beta^{4k}}{(\epsilonH)^2 (\epsilong)^2} + \tfrac{\beta^{5k}}{(\epsilonH)^5} + \tfrac{\beta^{2k}}{(\epsilong)^2} \rt) \nonumber \\
        &= \; \tilde{\cO} \lt( \tfrac{\beta^{3K}}{(\epsilong)^3} + \tfrac{\beta^{4K}}{(\epsilonH)^2 (\epsilong)^2} + \tfrac{\beta^{5K}}{(\epsilonH)^5} + \tfrac{\beta^{2K}}{(\epsilong)^2} \rt) \nonumber \\
        &= \; \tilde{\cO} \lt( ( \epsilong )^{-6} ( \epsilonH )^{-2} + ( \epsilong )^{-5} ( \epsilonH )^{-5} \rt), \label{eq:1-K-complexity-sgd-det}
    \end{align}
    where the last equality is due to the fact that $\beta^K = \cO((\epsilong)^{-1})$. Combining \eqref{eq:0-complexity-sgd}, \eqref{eq:1-K-complexity-sgd-det}, and \eqref{eq:sample-complexity-bound} yields the result.

     Case (ii): Let $D:=\cap_{k=1}^K \{\Delta_k \leq \Delta^*\}$. From \eqref{eq:delta-prob-overall-bound_init}, we have $\PP(D 
     ) \geq \prod_{k=0}^{K-1}p_k = p^K$. Suppose that event 
     $D$
     holds, following the same approach as in the proof of Case (ii) in \Cref{thm:theoretical-solver}, we have 
    \begin{align*}
       \sum_{k=1}^K \cW^{(p)}_{\cM, k} &=  \tilde{\cO} \lt( \tfrac{\beta^{2K}}{(\epsilong)^3} + \tfrac{\beta^{3K}}{(\epsilonH)^2 (\epsilong)^2} + \tfrac{\beta^{4K}}{(\epsilonH)^5} + \tfrac{\beta^{K}}{(\epsilong)^2} \rt)  = \; \tilde{\cO} \lt( ( \epsilong )^{-5} ( \epsilonH )^{-2} + ( \epsilong )^{-4} ( \epsilonH )^{-5} \rt). 
    \end{align*}
    Therefore, 
    \begin{align}
        \PP\lt( \sum_{k=1}^K \cW^{(p)}_{\cM, k} = \tilde{\cO}((\epsilong)^{-5} (\epsilonH)^{-2} + (\epsilong)^{-4} (\epsilonH)^{-5}) \rt) \geq p^K. \label{eq:1-K-complexity-sgd-prob} 
    \end{align}
    Since the bound in \eqref{eq:0-complexity-sgd} holds almost surely, combining it with \eqref{eq:1-K-complexity-sgd-prob} and \eqref{eq:sample-complexity-bound} yields the result.

\end{proof}

\begin{remark}
The deterministic sample complexity result in \Cref{thm:mescal_sgd}, which uses the deterministic bound on $\Delta_k$ in \eqref{eq:det_delta_bnd}, is worse by a factor of $(\epsilong)^{-1}$ than the probabilistic bound, which instead uses the constant upper bound on $\Delta_k$ that holds probabilistically given in \eqref{eq:prob_delta_bnd}. A limitation of the latter result is that the available theoretical analysis of SG-HV-NC guarantees success only with the fixed probability $p_k=5/8$ at each outer iteration $k$. This differs from \Cref{thm:theoretical-solver}, where the per-iteration success probabilities can be selected as $p_k=p^{1/K}$ to obtain an overall success probability of at least $p$. An extension of the
SG-HV-NC theoretical analysis to an arbitrary prescribed success probability could therefore yield a higher-probability counterpart of the present result. Finally, in the 
nonconvex 
setting, ignoring the $\epsilonH$ factors, the resulting probabilistic sample complexity matches the best-known first-order
complexity bounds for stochastic augmented-Lagrangian methods; see, for example, \cite{li_stochastic_2024}.
\end{remark}

%% file: 4-Natasha2.tex
\subsection{Primal Solver: Natasha2}
Natasha2 \cite{allen-zhu_natasha_2018} is a second-order method that alternates between negative-curvature steps and proximal variance-reduced gradient %
steps to obtain second-order %
{stationary points of }nonconvex unconstrained optimization problems. The step-by-step implementation of this algorithm as solver $\cM$ under \textbf{Option II} in \Cref{alg:mescal-overview}, is provided in \Cref{alg:natasha2}. This is primarily a stochastic solver, i.e., {it typically employs single-sample realizations of $\theta$ rather than a mini-batch within the variance-reduced gradient estimator step (see Line 11 in \Cref{alg:natasha1_5}).}  
Natasha2 utilizes two algorithmic {subroutines. }%
{The first is Oja's algorithm, for which a brief description is provided in~\Cref{subsec:sgd_ncs}. }
The {second} %
is the Natasha1.5 subroutine, {described in \Cref{alg:natasha1_5}}, which is intended to {obtain an $\epsilon$-accurate first-order stationary point. }%
Once these standard iterate updates are completed, %
{Natasha2 } 
utilizes an SGD variant \cite{allen-zhu_how_2018}, denoted SGD3SC, on a strongly convex approximation of the objective function to update the iterates. %
{This additional step is }necessary for {the }theoretical convergence but {is} not important in practice; its usage is established in \cite[Theorem 5.7]{allen-zhu_natasha_2018}. 

We next state the input-output guarantee of SGD3SC that we use in our analysis. Specifically, given a point $x$ and a stochastic function class $h$, SGD3SC generates an iterate with a bounded gradient norm.

\begin{definition}
    The SGD3SC algorithm takes as input a stochastic function class $h$ satisfying 1) strong convexity, 2) finite initial optimality gap, 3) Lipschitz continuous stochastic gradient, 4) Lipschitz continuous Hessian, 5) stochastic gradient with a uniform error bound in expectation. After $T$ iterations, the algorithm outputs $\bar{x} \in \RR^d$ such that $\EE[\|\nabla h(\bar{x}) \|] = \cO(1 / \sqrt{T})$.
\end{definition}

\begin{algorithm2e}[H]
\caption{Natasha2}
\label{alg:natasha2}
\SetKwInOut{Input}{Input}
\DontPrintSemicolon

\textbf{Input (for all $k \in \NN_0$):} function class $\cL_{\alpha_k}$; iterate $x_{k-1}$; dual variable $\lambda_k$; optimality gap $\Delta_k$; smoothness constants $L^{(g)}_{\cL, k}$, $L^{(H)}_{\cL, k}$; variance constant $\sigma^{(g)}_{\cL, k}$; primal subproblem tolerances $\epsilong_k$, $\epsilonH_k > 0$; solver-specific positive constants $c_L$, $c_\iota$, $c_B$, $c_q$, $c_\eta$, $c_N$, $c_{\tilde{L}}$, $c_T$ \;

initialize $x_{k, 0} = x_{k - 1}$ 

\eIf{$L^{(H)}_{\cL, k} \geq \tfrac{L^{(g)}_{\cL, k} \epsilonH_k}{3(\sigma^{(g)}_{\cL, k})^{2/3} (\epsilong_k)^{1/3}}$}{
    set $\tilde{L}_k = \max \lt( L^{(g)}_{\cL, k}, \tfrac{c_L L^{(H)}_{\cL, k} \sqrt[3]{(\sigma^{(g)}_{\cL, k})^{2} \epsilong_k}}{\epsilonH_k} \rt)$, $\iota_k = \tilde{L}_k$  \;
    
}{
    set $\tilde{L}_k = L^{(g)}_{\cL, k}$, $\iota_k = \min \lt(L^{(g)}_{\cL, k}, \max \lt\{ \tfrac{\epsilonH_k}{3}, \tfrac{c_\iota (\sigma^{(g)}_{\cL, k})^2 (\epsilong_k) (L^{(H)}_{\cL, k})^3}{(L^{(g)}_{\cL, k})^2 (\epsilonH_k)^3}, \tfrac{c_\iota (\epsilong_k) L^{(g)}_{\cL, k}}{\sigma^{(g)}_{\cL, k}} \rt\} \rt)$ 
    
}
set $B_k = \lt\lceil \tfrac{c_B (\sigma^{(g)}_{\cL, k})^2}{(\epsilong_k)^2} \rt\rceil$, $q_k = \lt\lceil c_q \lt( \tfrac{ B_k \iota_k^2}{\tilde{L}_k^2} \rt)^{1/3} \rt\rceil$, $\eta_k = \tfrac{c_\eta \iota_k}{q_k^2 \tilde{L}_k^2} $, $N_{k, 1} = \lt\lceil \tfrac{c_N \iota_k \Delta_k}{q_k (\epsilong_k)^2} \rt\rceil $, $t_N=0$  \; 

\For{$t = 0$ \KwTo $\infty$}{
    set $v = \text{Oja}\lt( x_{k, t}, \cL_{\alpha_k}(\cdot, \lambda_k), \tfrac{\epsilonH_k}{3}, \tfrac{1}{12} \rt)$  \tcp*{negative curvature search}
    
    \eIf{$v \neq \perp$}{
        sample $r_{k, t} \sim \mathrm{Uniform}(\{-1, 1\})$  \;
        set $x_{k, t+1} = x_{k, t} + \tfrac{r_{k, t} \epsilonH_k}{3 L^{(H)}_{\cL, k}} v $  \tcp*{second-order step}
    }{
        update $t_N \gets t_N + 1$  \;
        let $V_{k, t}(y) := \cL_{\alpha_k}(y, \lambda_k) + L^{(g)}_{\cL, k} \Big(\max \big\{0, \|y - x_{k, t}\| - \tfrac{\epsilonH_k}{3 L^{(H)}_{\cL, k}} \big\} \Big)^2$  \;
        set $(\hat{x}_{k, t}, x_{k, t+1}) = \text{Natasha1.5} (V_{k, t}, x_{k, t}, B_k, q_k, 1, \eta_k, L^{(g)}_{\cL, k})$  \tcp*{first-order step}
        store the vector pair $(x_{k, t}, \hat{x}_{k, t})$  \;
        go to \Cref{step:nat-outside-for} if $t_N = N_{k, 1}$  \;
    }
}
sample $(x, \hat{x}) \sim \mathrm{Uniform} (\{(x_{k, t}, \hat{x}_{k, t})\}_{t = 0}^{N_{k, 1}})$  \label{step:nat-outside-for} \;
let $W(y) := \cL_{\alpha_k}(y, \lambda_k) + L^{(g)}_{\cL, k} \Big(\max \big\{0, \|y - x\| - \tfrac{\epsilonH_k}{3 L^{(H)}_{\cL, k}}  \big\} \Big)^2 + \iota_k \| y - \hat{x}\|^2$  \;

set $y^{\text{out}} = \mathrm{SGD3SC} \lt( W, \hat{x}, \iota_k, c_{\tilde{L}} \tilde{L}_k, \tfrac{c_T(\sigma^{(g)}_{\cL, k})^2}{(\epsilong_k)^2} \log^3(\tfrac{\tilde{L}_k}{\iota_k}) \rt)$  \;
\Return $y^{\text{out}}$ 
\end{algorithm2e}

\begin{algorithm2e}[H]
\caption{$\text{Natasha1.5}$}
\label{alg:natasha1_5}
\SetKwInOut{Input}{Input}
\DontPrintSemicolon
\DontPrintSemicolon

\Input{function $V(\cdot)$, starting vector $y^\emptyset$, epoch length $B$, sub-epoch count $q$, epoch count $T' \geq 1$, learning rate $\eta > 0$, Lipschitz constant $L^{(g)}$ }

initialize $\hat{y}_{1, 1} = y^\emptyset$  \; 
set $B' = \lt\lceil B/q \rt\rceil$  \;

\For{$t' = 1$ \KwTo $T'$}{
    sample $\theta_j \sim \cP_\theta$ for $j \in \{0, 1, \cdots, B\}$  \;
    set $\tilde{y} = \hat{y}_{(t', 1)}$, $\mu = \tfrac{1}{B} \sum_{\theta_j = 0}^B \nabla V(\tilde{y}, \theta)$  \;
    \For{$t_q = 1$ \KwTo $q$}{
        initialize $y_0 = \hat{y}_{(t', t_q)}$  \;
        store the vector $\hat{y}_{(t', t_q)}$  \;
        \For{$b' = 0$ to $B' - 1$}{
            sample $\theta \sim \cP_\theta$  \;
            set $g = \nabla V (y_{b'}, \theta) - \nabla V (\tilde{y}, \theta) + \mu + 2 L^{(g)} (y_{b'} - \hat{y}_{(t', t_q)})$  \;
            $y_{b'+1} = y_{b'} - \eta g$ 
        }
        sample $\hat{y}_{(t', t_q + 1)} \sim \mathrm{Uniform}(\{y_0, y_1, \cdots, y_{B'}\})$  \;
    }
}
sample $\hat{x} \sim \mathrm{Uniform}(\{\hat{y}_{(t',t_q)} \mid 1 \leq t' \leq T', 1 \leq t_q \leq q \})$; set $x^+ = \hat{y}_{(T', q)}$  \;
\Return $(\hat{x}, x^+)$  
\end{algorithm2e}

\begin{remark}
    We briefly describe the roles of the main parameters in \Cref{alg:natasha2}:
    \begin{itemize}
        \item $B_k$: The epoch mini-batch length for the Natasha1.5 subroutine.
        \item $q_k$: The sub-epoch count in the Natasha1.5 subroutine.
        \item $\eta_k$: The step size used to update the iterate in the Natasha1.5 subroutine.
        \item $\iota_k$: The strong convexity parameter utilized for the SGD3SC subroutine.
        \item {$N_{k,1}:$ The number of first-order steps.}
        \item $V_{k, t}$: The approximation of the augmented Lagrangian function.
    \end{itemize}
\end{remark}
These parameters are specified in \Cref{alg:natasha2} to ensure that the resulting iterate $x_k$ is an $(\epsilong_k,\epsilonH_k)$-accurate iterate satisfying the probabilistic conditions in \eqref{cond:prob_primal_prob} at each outer iteration $k \in \NN$.

We now state the work complexity at each iteration $k \in \NN_0$, as defined in \Cref{def:primal_work_at_k}, of the Natasha2 algorithm, denoted by $\cW^{(p)}_{Nat2,k}$, when used as the primal solver in \Cref{alg:mescal-overview}.

\begin{lemma}[Work complexity of Natasha2] \label{lemm:iter-sample-complexity-natasha2}
    Suppose $\cX = \RR^d$, and Assumptions \ref{assn:abs-true-bounds}, \ref{assn:reg-c-nabla-c}, \ref{assn:lip-nabla-f-nabla-c}, \ref{assn:lip-nabla2-f-nabla2-c} and \ref{assn:var-bounds} hold. Let \Cref{alg:natasha2} be the solver $\cM$ under \textbf{Option II}, with terms $L^{(g)}_{\cL, k}$, $L^{(H)}_{\cL, k}$, $\sigma^{(g)}_{\cL, k}$, and $\Delta_k$ defined as per \Cref{def:complexity_factors}. If the primal subproblem tolerance inputs satisfy $\epsilong_k = \cO(\min\{1, \sigma^{(g)}_{\cL, k}\})$, and $\epsilonH_k = \cO \Big(\min \Big\{1, 3 L^{(g)}_{\cL, k}, \allowbreak \Big( \tfrac{(\sigma^{(g)}_{\cL, k})^2 \epsilong_k (L^{(H)}_{\cL, k})^3}{L^{(g)}_{\cL, k}} \Big)^{0.25}  \Big\} \Big)$, then at every iteration $k \in \NN_0$, \Cref{alg:natasha2} returns an iterate $x_k$ satisfying the probabilistic conditions~\eqref{cond:prob_primal_prob} with probability at least $p_k = 2/3$, with work complexity
    \begin{align*}
        \cW^{(p)}_{Nat2, k} = \tilde{\cO} \lt( 
        \tfrac{(\sigma^{(g)}_{\cL, k})^4}{(\epsilong_k)^2}
        + \tfrac{\Delta_k(L^{(H)}_{\cL, k})^2 (L^{(g)}_{\cL, k})^2 }{(\epsilonH_k)^5}
        + \tfrac{\Delta_kL^{(H)}_{\cL, k} (L^{(g)}_{\cL, k})^2}{\epsilong_k (\epsilonH_k)^3}
        + \tfrac{\Delta_kL^{(H)}_{\cL, k} (\sigma^{(g)}_{\cL, k})^2}{(\epsilong_k)^3 \epsilonH_k}
        + \tfrac{\Delta_k(L^{(g)}_{\cL, k})^3}{\epsilong_k (\epsilonH_k)^2 \sigma^{(g)}_{\cL, k}}
        \rt).
    \end{align*}
\end{lemma}

The proof of this lemma is provided in \Cref{sec:appendix_lemm:nat2}. %
{We now establish} the sample complexity %
{of} MESCAL-Natasha2, i.e., the MESCAL framework utilizing %
Natasha2 {as the solver $\cM$.} %

\begin{theorem}[Sample Complexity of MESCAL-Natasha2] \label{thm:mescal_natasha2}
    Suppose $\cX = \RR^d$, and Assumptions \ref{assn:abs-true-bounds}, \ref{assn:reg-c-nabla-c}, \ref{assn:lip-nabla-f-nabla-c}, \ref{assn:lip-nabla2-f-nabla2-c} and \ref{assn:var-bounds} hold. Let $\{x_k\}$ be the sequence of iterates generated by \Cref{alg:mescal-overview} with initial iterate $x_{-1} \in \cX$ and initial dual iterate $\lambda_0 \in \RR^m$. Suppose that the dual bound parameters $\gamma_k > 0$ be a bounded sequence satisfying $\sum_{k=0}^{\infty} \gamma_k = C_\gamma < \infty$, and that the dual sample sizes satisfy $|S^{(d)}_k| = \Theta((\epsilong)^{-2})$ for all $k \in \NN_0$. Let \Cref{alg:natasha2} be the solver $\cM$ under \textbf{Option II}. Suppose the primal subproblem tolerances $\epsilong_k$, $\epsilonH_k$ satisfy $\epsilong_k = \epsilong = \cO(\min\{1, \sigma^{(g)}_{\cL, k}\})$ and $\epsilonH_k = \epsilonH = \cO \Big(\min \Big\{1, 3 L^{(g)}_{\cL, k}, \allowbreak \Big( \tfrac{(\sigma^{(g)}_{\cL, k})^2 \epsilong_k (L^{(H)}_{\cL, k})^3}{L^{(g)}_{\cL, k}} \Big)^{0.25}  \Big\} \Big)$ for all $k \in \NN_0$. Then, for any $k = K \geq K^*$, where $K^*$ is defined in \eqref{eq:iter_comp_log}, the iterate $\tilde{x}:= x_K$ is an $(\epsilong, \epsilonH)$-accurate $p$-probabilistic second-order stationary point satisfying \eqref{eq:opt_cond_prob} as per \Cref{def:opt_cond_prob} with $p = 2/3$ and $\tilde{\lambda} := \lambda_K + \alpha_K c(x_K)$. Furthermore, the sample complexity of the algorithm as defined in \Cref{def:sample_complexity}, to obtain $\tilde{x}:= x_K$, is
    \begin{enumerate}[label=(\roman*)]
        \item \textbf{(Deterministic Sample Complexity):} 
        $$\cW = \tilde{\cO} ( ( \epsilong )^{-7} ( \epsilonH )^{-1} + ( \epsilong )^{-5} ( \epsilonH )^{-5} ).$$

        \item \textbf{(Probabilistic Sample Complexity):} 
        $$\cW = \tilde{\cO} ( ( \epsilong )^{-6} ( \epsilonH )^{-1} + ( \epsilong )^{-4} ( \epsilonH )^{-5} )$$
        with probability at least $p^K$.
    \end{enumerate}
\end{theorem}

\begin{proof}
    From \Cref{lemm:iter-sample-complexity-natasha2}, we establish that the iterate $x_k$ satisfies \eqref{cond:prob_primal_prob} in \Cref{alg:mescal-overview}. Therefore, by \Cref{thm:outer-iteration-complexity}, it holds that the iterate $\tilde{x} := x_K$ is an $(\epsilong, \epsilonH)$-accurate $p$-probabilistic second-order stationary point as defined in \Cref{def:opt_cond_prob}, with $p=2/3$ and where $\tilde{\lambda} := \lambda_K + \alpha_K c(x_K)$.

    Following the same approach as in \eqref{eq:prob-k-zero-primal-work} and \eqref{eq:0-complexity-sgd}, we have
    \begin{align}
        \cW^{(p)}_{Nat2, 0} = \tilde{\cO} \lt( (\epsilonH)^{-5} + (\epsilong)^{-1} (\epsilonH)^{-3} + (\epsilong)^{-3} (\epsilonH)^{-1} \rt). \label{eq:0-complexity-nat2}
    \end{align}

    Case (i): By the same reasoning as in \eqref{eq:1-K-complexity-sgd-det}, we have
    \begin{align}
        \sum_{k=1}^K \cW^{(p)}_{Nat2, k} = \tilde{\cO} \lt( ( \epsilong )^{-7} ( \epsilonH )^{-1} + ( \epsilong )^{-5} ( \epsilonH )^{-5} \rt). \label{eq:1-K-complexity-nat2}
    \end{align}
    Substituting \eqref{eq:0-complexity-nat2} and \eqref{eq:1-K-complexity-nat2} in \eqref{eq:sample-complexity-bound} completes the proof.

    Case (ii): For probabilistic bounds, by following the same approach as that for \eqref{eq:1-K-complexity-sgd-prob}, we have
    \begin{align}
        \PP\lt( \sum_{k=1}^K \cW^{(p)}_{Nat2, k} = \tilde{\cO}((\epsilong)^{-6} (\epsilonH)^{-1} + (\epsilong)^{-4} (\epsilonH)^{-5}) \rt) \geq p^K. \label{eq:1-K-complexity-nat2-prob} 
    \end{align}
    We use \eqref{eq:0-complexity-nat2} and \eqref{eq:1-K-complexity-nat2-prob} to complete the proof, as in \eqref{eq:complexity-general-prob}.

\end{proof}

\begin{remark}

The deterministic sample complexity results %
for MESCAL-Natasha2 and MESCAL-SG-HV-NC differ in one term. In particular, the bound for MESCAL-Natasha2 contains the term $(\epsilong)^{-7}(\epsilonH)^{-1}$, whereas the corresponding term for MESCAL-SG-HV-NC is $(\epsilong)^{-6}(\epsilonH)^{-2}$. The same distinction arises in the $p^K$-probabilistic sample complexity results. %
The bound for MESCAL-Natasha2 contains $(\epsilong)^{-6}(\epsilonH)^{-1}$, while that for MESCAL-SG-HV-NC contains $(\epsilong)^{-5}(\epsilonH)^{-2}$.

\end{remark}

%% file: 4-SPIDER.tex
\subsection{Primal Solver: SPIDER}
SPIDER--SFO+ (stochastic path-integrated differential estimator--stochastic first-order+) \cite[Algorithm 2]{fang_spider_2018}, concisely denoted as SPIDER, {is another second-order algorithm. The step-by-step implementation of this algorithm as the primal solver $\cM$ under \textbf{Option II} in \Cref{alg:mescal-overview}, is provided in \Cref{alg:spider_sfo_plus}.} 
It combines SARAH-based \cite{nguyen_sarah_2017} stochastic gradient tracking with a negative-curvature search procedure to efficiently escape saddle points. 
{Similar to }Natasha2, SPIDER first identifies directions of negative curvature using the Neon2 algorithm, a variant of Oja's algorithm. When a direction of negative curvature is detected, {rather than taking a single large step,} SPIDER performs multiple mini-steps along that direction, allowing {it to continuously update the gradient estimator without recomputing a highly accurate gradient.} If no negative curvature is {detected, it similarly performs multiple first-order updates while continuously updating the gradient estimator.} The algorithm terminates when the norm of the gradient estimator falls below the prescribed threshold.

We next state the input-output guarantee of the Neon2 algorithm{, which is similar to that of Oja's algorithm.} %
For information on the Neon2 procedure, as utilized in \Cref{step:neon2_spider} of \Cref{alg:spider_sfo_plus}, refer to \cite[Theorem 9]{fang_spider_2018}. A more detailed analysis {of Neon2} is provided in \cite{allen-zhu_neon2_2018}.

\begin{definition} \label{def:neon2}
    The Neon2 algorithm takes as input a point $x \in \RR^d$, stochastic function class $h$ that has 1) a finite initial optimality gap, 2) Lipschitz continuous stochastic gradient, 3) Lipschitz continuous stochastic Hessian, 4) stochastic gradient with uniform error bound almost surely, a precision parameter $\epsilonH > 0$ and a failure probability $p_f \in (0, 1)$, and outputs $w \in \RR^d \cup \{ \perp \}$ such that with probability at least $1 - p_f$, either

    \begin{enumerate}
        \item $w = \perp$, $\sigma_{\min} (\nabla ^2 h(x)) \geq -\epsilonH$.

        \item $w \neq \perp$, $\|w\| = 1$, and $w^\top \nabla^2 h(x) w \leq \tfrac{\epsilonH}{2}$.
    \end{enumerate}

\end{definition}

\begin{algorithm2e}[H]
\caption{SPIDER}
\label{alg:spider_sfo_plus}
\SetKwInOut{Input}{Input}
\DontPrintSemicolon

\textbf{Input (for all $k \in \NN_0$):} function class $\cL_{\alpha_k}$; initial iterate $x_{k-1}$; primal subproblem error tolerances $\overline{\epsilon}^{(g)}_k, \epsilonH_k > 0$; optimality gap $\Delta_k$; smoothness constants $L^{(g)}_{\cL, k}$, $L^{(H)}_{\cL, k}$; variance constant $\sigma^{(g)}_{\cL, k}$; primal parameter $\overline{\epsilon}^{(g)}_k$, primal subproblem tolerance $\epsilonH_k > 0$; solver-specific constant  $c_{n_0}$\;

set $c_{n_0} \in [1, 20 \sigma^{(g)}_{\cL, k}/\overline{\epsilon}^{(g)}_k]$, $B^{(p)}_k=\tfrac{20 \sigma^{(g)}_{\cL, k}}{\overline{\epsilon}^{(g)}_k c_{n_0}}$, $B^{\text{reset}}_k = \tfrac{200 (\sigma^{(g)}_{\cL, k})^2}{(\overline{\epsilon}^{(g)}_k)^2}$, $q_k=\tfrac{10 \sigma^{(g)}_{\cL, k} c_{n_0}}{\overline{\epsilon}^{(g)}_k}$, $\eta_k = \min \lt( \tfrac{\overline{\epsilon}^{(g)}_k}{10 L^{(g)}_{\cL, k} c_{n_0} \| g_t \|}, \tfrac{1}{L^{(g)}_{\cL, k} c_{n_0}} \rt), J_k = 4 \lt\lfloor \max \lt( \tfrac{81 L^{(H)}_{\cL, k} \Delta_k}{(\epsilonH_k)^3}, \tfrac{120 \Delta_k L^{(H)}_{\cL, k}}{\epsilonH_k \overline{\epsilon}^{(g)}_k} \rt) \rt\rfloor, \cT_k = \lt\lceil \tfrac{10 \epsilonH_k L^{(g)}_{\cL, k} c_{n_0}}{3 L^{(H)}_{\cL, k} \overline{\epsilon}^{(g)}_k} \rt\rceil$\;

set primal subproblem tolerance $\epsilong_k =  \overline{\epsilon}^{(g)}_k \log(64 (\cT_k J_k + 1))$ \; 

set $t = 0$; initialize $x_{k, 0} = x_{k - 1}$ \;

\For{$j = 0$ \KwTo $J_k$}{
    set $w_1 = \text{Neon2}\lt( \cL_{\alpha_k}(\cdot, \lambda_k), x_{k, t}, 2\epsilonH_k, \tfrac{1}{16J_k} \rt)$ \label{step:neon2_spider} \tcp*{negative curvature search}
    
    \If{$w_1 \neq \perp$}{
        sample $b_j \sim \mathrm{Uniform}(\{-1, 1\})$ \;
    }
    
    \While{$t < (j + 1) \cT_k$}{
        \eIf{$t \bmod q_k = 0$}{
            sample $\theta_i \sim \cP_\theta, \quad \forall i \in \{1, 2, \cdots, B^{\text{reset}}_k\}$ \;
            set $g_t = \tfrac{\sum_{\theta_i} \nabla \bar{\cL}_{\alpha_k} (x_{k, t}, \lambda_k; \theta_i)}{B^{\text{reset}}_k}$ \;
        }{
            sample $\theta_i \sim \cP_\theta, \quad \forall i \in \{1, 2, \cdots, B^{(p)}_k\}$ \;
            set $g_t = \tfrac{\sum_{\theta_i} \nabla \bar{\cL}_{\alpha_k} (x_{k, t}, \lambda_k; \theta_i)}{B^{(p)}_k} - \tfrac{\sum_{\theta_i} \nabla \bar{\cL}_{\alpha_k} (x_{k, t-1}, \lambda_k; \theta_i)}{B^{(p)}_k} + g_{t-1}$ \;
        }
        \eIf{$w_1 \neq \perp$}{
            set $x_{k, t+1} = x_{k, t} - b_j \eta_k w_1$ \tcp*{second-order step}
        }{
            \If{$\|g_t\| \leq 2 \epsilong_k$}{
                \Return $x_{k, t}$ \;
            }
            set $x_{k, t+1} = x_{k, t} - \eta_k (g_t / \|g_t\|)$ \tcp*{first-order step}
        }
        update $t \gets t + 1$ \;
    }
}
\Return $x_{k, t}$
\end{algorithm2e}

\begin{remark}
    We briefly describe the roles of the main parameters in \Cref{alg:spider_sfo_plus}:
    \begin{itemize}
        \item $B^{(p)}_k$: The %
        smaller %
        {sample size} employed {in gradient estimator}.
        \item $B^{\text{reset}}_k$: %
        {The larger sample size employed in computing occasional highly accurate gradient approximation used in gradient estimator}.
        \item $q_k$: The number of iterations before the {highly accurate} gradient {computation is performed. }%
        \item $\eta_k$: The step size for both first-order and second-order steps.
        \item $J_k$: The total number of {times the negative curvature search step is performed.}%
        \item $\cT_k$: The number of %
        first-order or second-order %
        steps %
        {performed} {in between successive negative curvature search steps.}
        \item $\overline{\epsilon}^{(g)}_k$: The primal parameter on which the primal subproblem tolerance $\epsilong$ is dependent on, with $\epsilong_k = \tilde{\cO} ( \overline{\epsilon}^{(g)}_k )$.
    \end{itemize}
\end{remark}

\begin{lemma}[Work complexity of SPIDER] \label{lemm:iter-sample-complexity-spider}
    Suppose $\cX = \RR^d$, and Assumptions \ref{assn:abs-true-bounds}, \ref{assn:reg-c-nabla-c}, \ref{assn:lip-nabla-f-nabla-c}, \ref{assn:lip-nabla2-f-nabla2-c} and \ref{assn:var-bounds} hold. Let \Cref{alg:spider_sfo_plus} be the solver $\cM$ under \textbf{Option II}, with terms $L^{(g)}_{\cL, k}$, $L^{(H)}_{\cL, k}$, $\sigma^{(g)}_{\cL, k}$, and $\Delta_k$ defined as per \Cref{def:complexity_factors}. Suppose for $\overline{\epsilon}^{(g)}_k \leq 1$, $\epsilonH_k \leq 1$, for every iteration $k \in \NN_0$ such that 
    primal subproblem tolerances $\epsilong_k = \overline{\epsilon}^{(g)}_k \log(64 (\cT_k J_k + 1)) \leq 1$ and $\epsilonH_k \leq 1$, {where $\cT_k$, $J_k$ are defined in \Cref{alg:spider_sfo_plus}. Then, }\Cref{alg:spider_sfo_plus} returns an iterate $x_k$ satisfying the probabilistic conditions~\eqref{cond:prob_primal_prob} with probability at least $p_k = 1/2$, with work complexity
    \begin{align*}
        &\cW^{(p)}_{SPDR, k} = \tilde{\cO} \lt( 
        \tfrac{\Delta_k L^{(g)}_{\cL, k} \sigma^{(g)}_{\cL, k}}{(\overline{\epsilon}^{(g)}_k)^3} + \tfrac{\Delta_k L^{(g)}_{\cL, k} L^{(H)}_{\cL, k} \sigma^{(g)}_{\cL, k}}{(\overline{\epsilon}^{(g)}_k)^2 (\epsilonH_k)^2} 
        + \tfrac{\Delta_k (L^{(g)}_{\cL, k})^2 (L^{(H)}_{\cL, k})^2}{(\epsilonH_k)^5} 
        + \tfrac{\Delta_k (L^{(g)}_{\cL, k})^2 L^{(H)}_{\cL, k}}{\overline{\epsilon}^{(g)}_k (\epsilonH_k)^3} + \tfrac{(\sigma^{(g)}_{\cL, k})^2}{(\overline{\epsilon}^{g}_k)^2} 
        + \tfrac{L^{(g)}_{\cL, k} \sigma^{(g)}_{\cL, k} \epsilonH_k}{(\overline{\epsilon}^{g}_k)^2 L^{(H)}_{\cL, k}} \rt).
    \end{align*}

\end{lemma}

The proof of this lemma is provided in \Cref{sec:appendix_lemm:spdr}. We now establish the sample complexity of MESCAL-SPIDER, i.e., the MESCAL framework utilizing SPIDER as the solver $\cM$. 

\begin{theorem}[Sample Complexity of MESCAL-SPIDER] \label{thm:mescal_spider}
    Suppose $\cX = \RR^d$, and Assumptions \ref{assn:abs-true-bounds}, \ref{assn:reg-c-nabla-c}, \ref{assn:lip-nabla-f-nabla-c}, \ref{assn:lip-nabla2-f-nabla2-c} and \ref{assn:var-bounds} hold. Let $\{x_k\}$ be the sequence of iterates generated by \Cref{alg:mescal-overview} with initial iterate $x_{-1} \in \cX$ and initial dual iterate $\lambda_0 \in \RR^m$. Suppose that the dual bound parameters $\gamma_k > 0$ be a bounded sequence satisfying $\sum_{k=0}^{\infty} \gamma_k = C_\gamma < \infty$, and that the dual sample sizes satisfy $|S^{(d)}_k| = \Theta((\epsilong)^{-2})$ for all $k \in \NN_0$. Let \Cref{alg:spider_sfo_plus} be the solver $\cM$ under \textbf{Option II}. Suppose for $\overline{\epsilon}^{(g)}_k = \overline{\epsilon}^{(g)} \leq 1$, $\epsilonH_k \leq 1$, such that primal subproblem tolerance inputs $\epsilong_k = \overline{\epsilon}^{(g)}_k \log(64 (\cT_k J_k + 1))$
    , $\epsilong = \epsilong_K \leq 1$ and $\epsilonH = \epsilonH_k \leq 1$, with number of iterations $k = K = K^*$, where $K^*$ as per \Cref{thm:outer-iteration-complexity} and $\cT_k$, $J_k$ defined in \Cref{alg:spider_sfo_plus}. Then, $\tilde{x}:= x_K$ is an $(\epsilong, \epsilonH)$-accurate $p$-probabilistic second-order stationary point satisfying \eqref{eq:opt_cond_prob} as per \Cref{def:opt_cond_prob} with $p=1/2$, and $\tilde{\lambda} := \lambda_K + \alpha_K c(x_K)$. Furthermore, {if there exists a constant $q>0$ such that $\log(1/\epsilong) \geq \log^q(1/\epsilonH)$, then} the sample complexity of the algorithm, as defined in \Cref{def:sample_complexity}, to obtain $\tilde{x}:= x_K$ is
    \begin{enumerate}
        \item \textbf{(Deterministic Sample Complexity):} 
        $$\cW = \tilde{\cO}( ( \epsilong )^{-6} ( \epsilonH )^{-2} + ( \epsilong )^{-5} ( \epsilonH )^{-5} ).$$

        \item \textbf{(Probabilistic Sample Complexity):}  
        $$\cW = \tilde{\cO}( ( \epsilong )^{-5} ( \epsilonH )^{-2} + ( \epsilong )^{-4} ( \epsilonH )^{-5})$$
        with probability at least $p^K$.
    \end{enumerate}

\end{theorem}

\begin{proof}
    From \Cref{lemm:iter-sample-complexity-spider}, we establish that the iterate $x_k$ satisfies \eqref{cond:prob_primal_prob}. As $\epsilong_k \leq \epsilong_K = \epsilong$ for all $k \leq K$ by definition; it holds from \Cref{thm:outer-iteration-complexity} that the iterate $\tilde{x} := x_K$ is an $(\epsilong, \epsilonH)$-accurate $p$-probabilistic second-order stationary point, according to \Cref{def:opt_cond_prob}, with $p=1/2$, and where $\tilde{\lambda} := \lambda_K + \alpha_K c(x_K)$.

    From \eqref{eq:aug-constants-complexity-bounds} and the fact that that $\beta^K = \cO((\epsilong)^{-1})$, we have,
    \begin{align*}
        \epsilong = \overline{\epsilon}^{(g)} \log(64 (\cT_K J_K + 1)) = \overline{\epsilon}^{(g)} \Theta(\log( \max\{ (\epsilonH)^{-3} (\epsilong)^{-3} (\overline{\epsilon}^{(g)})^{-1}, (\epsilonH)^{-1} (\epsilong)^{-3}  (\overline{\epsilon}^{(g)})^{-2} \}))
    \end{align*}
    Then, from \Cref{lemm:epsilong-overline-equal}, we have for any constant $r > 0$ that 
    \begin{align}
        \tilde{\cO}((\overline{\epsilon}^{(g)})^{-r}) = \tilde{\cO}((\epsilong)^{-r}). \label{eq:spdr-equivalence}
    \end{align}

    Following the same approach as in \eqref{eq:prob-k-zero-primal-work} and \eqref{eq:0-complexity-sgd}, and using \eqref{eq:spdr-equivalence} we have
    \begin{align}
        \cW^{(p)}_{SPDR, 0} &= \tilde{\cO} \lt( (\overline{\epsilon}^{(g)})^{-3} + 
        (\overline{\epsilon}^{(g)})^{-2}(\epsilonH)^{-2} + 
        (\epsilonH)^{-5} +
        (\overline{\epsilon}^{(g)})^{-1} (\epsilonH)^{-3} 
        \rt) \notag \\
        &= \tilde{\cO} \lt( (\epsilong)^{-3} + 
        (\epsilong)^{-2}(\epsilonH)^{-2} + 
        (\epsilonH)^{-5} +
        (\epsilong)^{-1} (\epsilonH)^{-3} 
        \rt). \label{eq:0-complexity-spdr}
    \end{align}

    Case (i): By the same reasoning as in \eqref{eq:1-K-complexity-sgd-det} and using \eqref{eq:spdr-equivalence}, we have
    \begin{align}
        \sum_{k=1}^K \cW^{(p)}_{SPDR, k} = \tilde{\cO} \lt( ( \epsilong )^{-6} ( \epsilonH )^{-2} + ( \epsilong )^{-5} ( \epsilonH )^{-5} \rt). \label{eq:1-K-complexity-spdr}
    \end{align}
    Substituting \eqref{eq:0-complexity-spdr} and \eqref{eq:1-K-complexity-spdr} in \eqref{eq:sample-complexity-bound} completes the proof.
    
    Case (ii): For probabilistic bounds, by following the same approach as that for \eqref{eq:1-K-complexity-sgd-prob} and using \eqref{eq:spdr-equivalence}, we have
    \begin{align}
        \PP\lt( \sum_{k=1}^K \cW^{(p)}_{\cM, k} = \tilde{\cO}((\epsilong)^{-5} (\epsilonH)^{-2} + (\epsilong)^{-4} (\epsilonH)^{-5}) \rt) \geq p^K. \label{eq:1-K-complexity-spdr-prob} 
    \end{align}
    We use \eqref{eq:0-complexity-spdr} and \eqref{eq:1-K-complexity-spdr-prob} to complete the proof as in \eqref{eq:complexity-general-prob}.

\end{proof}
\begin{remark}
{We note that} %
    the sample complexity of MESCAL-SPIDER %
    {matches that of} %
    MESCAL-SG-HV-NC. However, we have an additional mild restriction on the parameters $\epsilong, \epsilonH$ such that $\epsilonH$ cannot decrease exponentially faster than $\epsilong$. 
\end{remark}

%% file: 5-numerical-experiments.tex
\section{Numerical Experiments} \label{sec:num_expts}

In this section, we demonstrate the empirical performance of the proposed MESCAL framework for solving a fairness constrained problem and a Neyman-Pearson classification problem. To solve the primal subproblem in \Cref{step:primal_prob} of MESCAL, we consider the three solver choices discussed in \Cref{sec:active_solvers}: SG-HV-NC \cite{arjevani_secondorder_2020}, Natasha2 \cite{allen-zhu_natasha_2018}, and SPIDER \cite{fang_spider_2018}.

We make a few modifications to the MESCAL framework in \Cref{alg:mescal-overview} to enhance its computational efficiency in practical implementations. The practical variant differs from \Cref{alg:mescal-overview} in a few respects.
\begin{itemize}
    
    \item \textbf{Primal variable update:} To conduct Step \ref{step:primal_prob} of \Cref{alg:mescal-overview}, we terminate the primal solver and set $x_k := x_{k-1, t}$ after $t$ iterations of the primal problem subsolver, if $x_{k-1, t}$ satisfies the conditions 
    in \Cref{alg:mescal-overview}. We also impose a maximum number of inner iterations $t \leq t^*$. If $x_{k-1, t^*}$ does not satisfy the conditions in \Cref{alg:mescal-overview},
    then $x_k := x_{k-1, t^*}$. Finally, we impose a maximum number of outer iterations $K$, such that the outer loop terminates when $k=K$. %

    \item \textbf{Dual variable update:} Instead of the bounded dual update as provided in Line \ref{step:dual-update} of Algorithm \ref{alg:mescal-overview}, we employ the full dual update by setting $\gamma_k = \infty$ for all $k \in \NN_0$. Although the full dual update increases the theoretical computational complexity of expectation-constrained optimization \cite{alacaoglu_complexity_2024}, 
    it leads to superior practical performance.
    
    \item \textbf{Subsampled gradient estimator}: Theoretically, in \Cref{sec:inner_loop} and \Cref{sec:active_solvers}, we utilize a sample average approximation of the augmented Lagrangian gradient estimator in \eqref{eq:sample_gradient_estimator} for mini-batch gradient {estimators}. For %
    {improved} practical performance, {we instead use the modified estimator}%
    \begin{align}
    \label{eq:mod_grad_estimator}
        &\nabla \hat{\cL}_\alpha (x, \lambda{, S^{(p)}}) := \frac{\sum_{\theta_i \in S^{(p)}} \nabla F(x, \xi_i)}{|S^{(p)}|} +  \frac{\sum_{\theta_i \in S^{(p)}} \nabla C(x, \zetaone_i)}{|S^{(p)}|}^\top \lt( \lambda + \alpha \frac{\sum_{\theta_i \in S^{(p)}} C(x, \zetatwo_i)}{|S^{(p)}|} \rt),
    \end{align}
    where $S^{(p)} = \{(\xi_i, \zetaone_i, \zetatwo_i)\}_{i=1}^|S^{(p)}|$, and $\zetaone_{i_1}, \zetatwo_{i_2} \simiid \cP_\zeta$ for all $i_1, i_2$. We note that this new estimator is also unbiased. %
    {Moreover, }the %
    {additional }cross-product terms {reduce the variance relative to subsampling the estimator in \eqref{eq:sample_gradient_estimator}. }
    We prove this variance reduction in %
    \Cref{lemm:grad-estimator-variance}.
    
    \item \textbf{Subsampled Hessian estimator}: Theoretically, we utilize a sample average approximation of the augmented Lagrangian Hessian estimator in \eqref{eq:sample_hessian_estimator} for mini-batch Hessian estimators. For %
    {improved} practical performance, we first consider a modified estimator, similar to %
    \eqref{eq:mod_grad_estimator}, %
    {that benefits from variance reduction:}
    \begin{equation} \label{eq:mod_hess_estimator}
    \begin{aligned}
        \nabla^2 \ddot{\cL}_\alpha (x, \lambda{, S^{(p)}}) &:= \frac{\sum\limits_{\theta_i \in S^{(p)}} \nabla^2 F(x, \xi_i)}{|S^{(p)}|} +  \sum_{j=1}^m \lt( \lambda^{(j)} + \alpha \frac{\sum\limits_{\theta_i \in S^{(p)}} C_j(x, \zetatwo_i)}{|S^{(p)}|} \rt) \frac{\sum\limits_{\theta_i \in S^{(p)}} \nabla^2 C_j(x, \zetaone_i)}{|S^{(p)}|}^\top \\
        & \qquad + \alpha \frac{\sum\limits_{\theta_i \in S^{(p)}} \nabla C(x, \zetaone_i)^\top}{|S^{(p)}|} \frac{\sum\limits_{\theta_i \in S^{(p)}} \nabla C(x, \zetatwo_i)}{|S^{(p)}|}.
    \end{aligned}
    \end{equation}
    As discussed above, the cross-product terms %
    {provide }variance reduction. However, this estimator is not symmetric because of the term involving the product of constraint Jacobians. Since nonsymmetric Hessian estimates can increase computational overhead in floating-point implementations, we symmetrize this Gram matrix term as follows:
    \begin{equation} \label{eq:mod_hess_estimator_v2}
    \begin{aligned}
        & \frac{\sum\limits_{\theta_i \in S^{(p)}} \nabla C(x, \zetaone_i)^\top}{|S^{(p)}|} \frac{\sum\limits_{\theta_i \in S^{(p)}} \nabla C(x, \zetatwo_i)}{|S^{(p)}|} \rightarrow \\
        & \qquad \frac{1}{2} \frac{\sum\limits_{\theta_i \in S^{(p)}} \nabla C(x, \zetaone_i)^\top}{|S^{(p)}|} \frac{\sum\limits_{\theta_i \in S^{(p)}} \nabla C(x, \zetatwo_i)}{|S^{(p)}|} + \frac{1}{2} \frac{\sum\limits_{\theta_i \in S^{(p)}} \nabla C(x, \zetatwo_i)^\top}{|S^{(p)}|} \frac{\sum\limits_{\theta_i \in S^{(p)}} \nabla C(x, \zetaone_i)}{|S^{(p)}|}
    \end{aligned}
    \end{equation}
    This %
    modification does not %
    {introduce} bias. {However, the resulting estimator $\nabla^2 \hat{\cL}_\alpha (x, \lambda{, S^{(p)}})$ obtained by  using \eqref{eq:mod_hess_estimator_v2} in \eqref{eq:mod_hess_estimator}} may have a higher variance compared to estimator $\nabla^2 \ddot{\cL}_\alpha (x, \lambda{, S^{(p)}})$. We use the symmetrized estimator in our practical implementations because it retains the variance-reduction benefit of the cross-product terms while improving computational efficiency through symmetry.%

    \item \textbf{Handling inequality constraints:} Typically, slack variables are handled implicitly in augmented Lagrangian methods using implicit slack formulations, e.g. see \cite{rockafellar_dual_1973}. However, we handle inequality constraints explicitly, by augmenting primal variable $x$ with slack variables $s$, and setting $\cX = \RR^d \times \RR^{\mathrm{dim}(s)}_+$ with non-negativity constraints in the projection set. 
    We record three values at every outer iteration of the primal solver: 1) the first order stationarity error $\|\mathrm{proj}_{T_{\cX}(x, s)}(-\nabla_{(x, s)} \cL_\alpha(x, s, \lambda)) \|$ corresponding to \eqref{eq:eps-kkt}(i), 2) the feasibility error $\|c(x) + s\|$ corresponding to \eqref{eq:eps-kkt}(ii), and 3) the minimum eigenvalue of the Hessian of the augmented Lagrangian function $\sigma_{\min} (\nabla^2_{(x, s)} \cL_{\alpha}(x, s, \lambda))$ which provides a lower bound on \eqref{eq:eps-kkt}(iii). We provide \Cref{rmk:proj-dual-comp-slack} in the Appendix on how the first-order stationarity error in \eqref{eq:eps-kkt}$\mathrm{(i)}$ encodes primal feasibility, dual feasibility and complementary slackness for standard inequality constraints, and why we utilize the minimum eigenvalue of the Hessian of the augmented Lagrangian function instead of utilizing \eqref{eq:eps-kkt}(iii).
\end{itemize}

We implemented the practical variant described above, together with the three primal solvers, in Python on a shared Ubuntu 22.04.5 LTS system with CUDA support.\footnote{{Our code will be made publicly available upon publication of the manuscript.}} %
All computations were performed on a machine with an Intel Xeon Gold 6426Y processor with 64 CPU cores across two NUMA nodes and an NVIDIA L40S GPU with driver version 550.120 and CUDA 12.4. Each run was allocated 4 CPU cores and 16 GB of RAM. The system has 251 GB of total memory and 15 GB of swap space, although actual memory usage varied across runs. Each solver has several hyperparameters, as described in \Cref{sec:active_solvers}. For each problem, we first selected reasonable values for these parameters based on the conditions recommended in the respective references and preliminary empirical results. We then performed a local grid search to further improve the performance of each solver. A complete fine-tuning of all solver hyperparameters is beyond the scope of this paper due to the size of the resulting parameter space and the high computational cost of individual runs.

\subsection{Nonconvex Fairness Constrained Problem} \label{subsec:num_fair}
{
We first consider a machine learning task involving a data-driven fairness constraint, where we fit a classification model with a fairness-based constraint that ensures a target group receives a sufficient rate of positive outcomes. Let $\mathcal{D} = \{(a_i, y_i)\}_{i=1}^{|\mathcal{D}|}$ where $a_i \in \RR^d$ represents the feature vector of a particular realization and $y_i \in \{-1, 1\}$ represents the class label; let $\mathcal{A} = \{a_j\}_{j=1}^{|\mathcal{A}|}$ denote a possibly unlabeled set representing the general population; and let  $\mathcal{A}_{\mathrm{tar}} \subseteq \mathcal{A}$ denote the target group. The goal is to minimize the empirical risk over $\mathcal{D}$ while ensuring the average predicted probability for the target group is at least a fraction $\tau_\mathcal{A}$ of the average predicted probability for the entire population $\mathcal{A}$. This problem can be formulated as the following constrained optimization problem:
\begin{equation}\label{eq:fair}
	\begin{aligned}
	\min_{x \in \RR^d} & \ f(x):= \frac{1}{|\mathcal{D}|} \sum_{i: (a_i, y_i) \in \mathcal{D}} \phi_\rho(\ell (x; a_i, y_i)), \\ 
	\text{s.t.} & \ c(x):= \frac{\tau_{\mathcal{A}}}{|\mathcal{A}|}\sum_{i: a_i \in \mathcal{A}} \sigma(a_i^{\top} x) - \frac{1}{|\mathcal{A}_{\mathrm{tar}}|} \sum_{i: a_i \in \mathcal{A}_{\mathrm{tar}}} \sigma(a_i^{\top} x) \le 0,
	\end{aligned}
\end{equation}
where $x \in \RR^d$ denotes the learnable parameters of the model, $\ell(x; a_i, y_i) := \log(1 + \exp(-y_i a_i^\top x))$ is the standard logistic loss function, and $\phi_\rho(u) := \rho \log \lt(1 + \tfrac{u}{\rho} \rt)$ is a concave truncation function used to truncate the loss, where $\rho$ is a scaling parameter. The function $\sigma(z) := (1 + \exp(-z))^{-1}$ represents the predicted probability of a sample belonging to the positive class, and $\tau_{\mathcal{A}}$ denotes the fairness fraction parameter; see \cite{ma_quadratically_2020} for further details. We use two datasets to evaluate the performance of our algorithm on this problem: 
}

{
\begin{itemize} 
    \item \emph{bank \cite{dua_uci_}:}  $d=81$, $(|\mathcal{D}|, |\mathcal{A}|, |\mathcal{A}_{tar}|)= (22605, 22605, 6320)$, and  $\tau_{\mathcal{A}} = 0.4$. 
    \item \emph{loan \cite{george_lendingclub}:}  $d=250$, $(|\mathcal{D}|, |\mathcal{A}|, |\mathcal{A}_{tar}|)= (63890, 64485, 31966)$, and  $\tau_{\mathcal{A}} = 0.6$.
\end{itemize}
}

To solve \eqref{eq:fair} using MESCAL, we reformulate its inequality constraint to an equality constraint $c(x) + s = 0$, where $s \geq 0$ is enforced %
by projections after %
{each} iterate update. This reformulation has stationarity conditions equivalent to those of the original problem \eqref{eq:fair} as shown in \cite{li_rateimproved_2021}. {In all experiments, we set} the tolerances %
to $\epsilon = \epsilong = \epsilonH = 0.01$. %
The algorithm %
terminates if the first-order stationarity, feasibility, and a lower bound on second-order stationarity conditions %
in \eqref{eq:eps-kkt} for the equality-constrained reformulation of \eqref{eq:fair} are all satisfied to within $\epsilon$. For all subroutines, at the $k$-th iteration of MESCAL, we set $\alpha_k = 1.1^{k+1}$. %
All methods {are initialized at} $x_{{-1}} = \tfrac{20}{\sqrt{d}} [1 1 \dots]^\top$ where $[1 1 \dots]^\top$ represents a vector of ones. 

Figure \ref{fig:fc_3_abs_cons} plots the second-order stationarity error, first-order stationarity error, and feasibility error against the total work. Here, the total work is defined as the total number of stochastic augmented Lagrangian gradient evaluations plus twice the total number of stochastic augmented Lagrangian Hessian vector product evaluations. The factor of two reflects the convention that a Hessian vector product evaluation requires approximately twice the computational effort of a gradient evaluation. As established in \Cref{lemm:gen-complexity-solver}, this measure of work accounts for the total stochastic objective gradient, constraint function, constraint gradient, objective Hessian vector product, and constraint Hessian vector product evaluations.

\begin{figure}[h]
    \centering
    \includegraphics[width=0.9\linewidth]{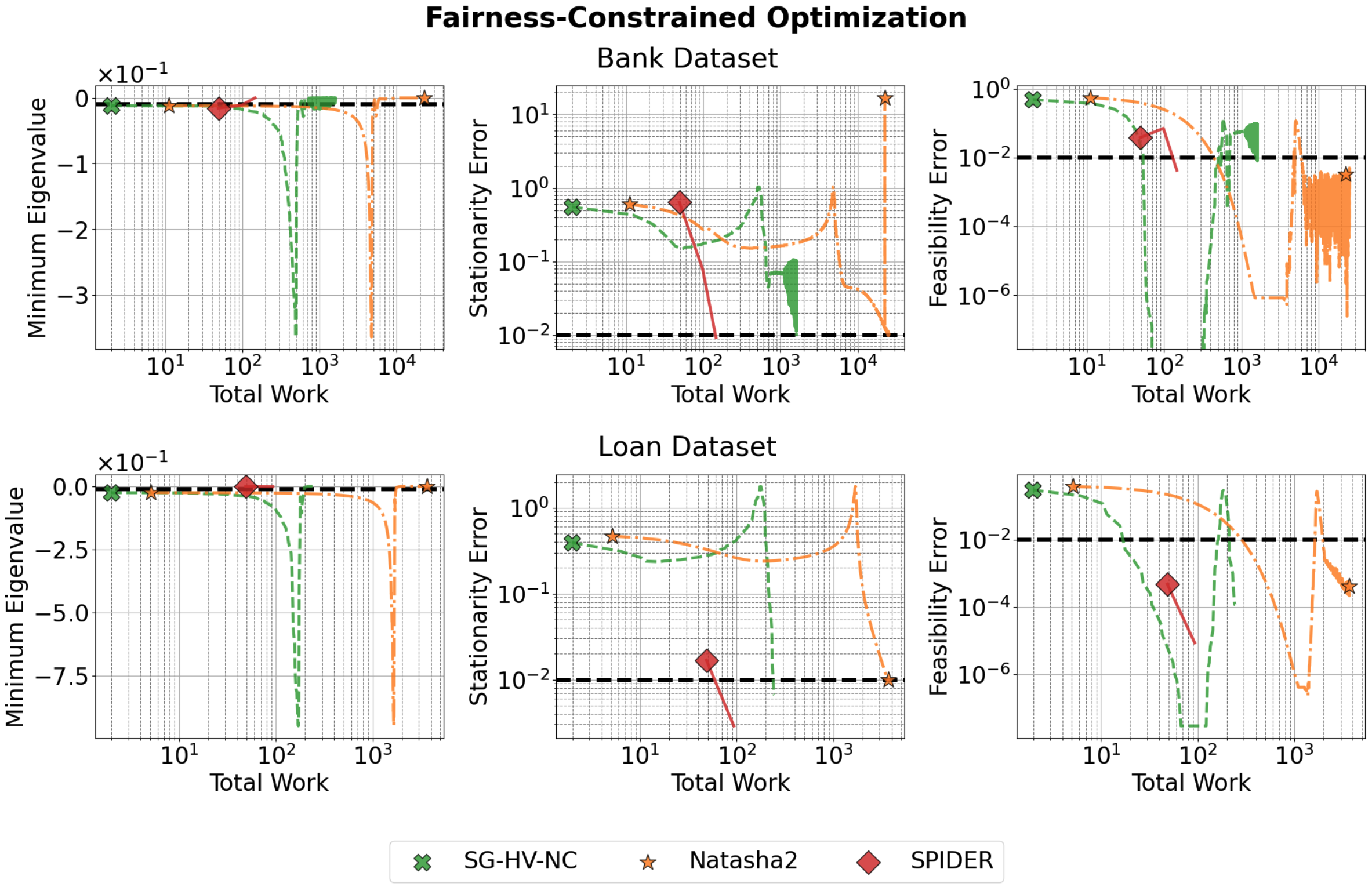}
    \caption{Results of %
    Algorithm \ref{alg:mescal-overview} on the nonconvex fairness-constrained optimization problem utilizing the \emph{bank} \cite{dua_uci_} and the \emph{loan} \cite{george_lendingclub} dataset. From left to right, the plots report the minimum eigenvalue of the Hessian of the augmented Lagrangian, the first-order stationarity error \eqref{eq:eps-kkt}$\mathrm{(i)}$, and the feasibility error \eqref{eq:eps-kkt}$\mathrm{(ii)}$. Results are shown for the three primal subproblem solvers SG-HV-NC, Natasha2, and SPIDER. The $x$-axis represents the total computational work, measured by the number of stochastic augmented Lagrangian gradient and Hessian vector product evaluations. Markers indicate the beginning of each MESCAL iteration for the corresponding solver. All solvers use the same initial point; the initial point is not shown because the $x$-axis is on a logarithmic scale.}
    
    \label{fig:fc_3_abs_cons}
\end{figure}

We observe that SPIDER requires approximately one order of magnitude less work than SG-HV-NC and two orders of magnitude less work than Natasha2 to reach the desired accuracy. Although appropriate tuning of the hyperparameters may slightly change the relative performance of these solvers, SPIDER consistently performed well in our experiments. This behavior can be attributed, at least in part, to its use of SARAH-based gradient updates and its relatively infrequent negative curvature searches. For the problems considered, the iterates $(\{x_k\})$ rapidly move out of regions of negative curvature, reducing the need for repeated negative curvature searches. Consequently, SPIDER avoids frequent Hessian vector product evaluations and achieves lower overall computational work. All methods converge to the prescribed tolerances, although the stationarity and feasibility errors for SG-HV-NC and Natasha2 exhibit some transient increases before converging to the prescribed tolerances.

\subsection{Nonconvex Neyman-Pearson Classification}
We now consider a nonconvex Neyman-Pearson classification problem \cite{rigollet_neymanpearson_2011, yan_adaptive_2022}, where the goal is to minimize the false-negative error rate while ensuring that the false-positive error rate does not exceed a prescribed maximum value. This problem can be formulated as the following constrained optimization problem:
\begin{equation} \label{eq:np-clas}
	\begin{aligned}
	\min_{x \in \RR^d} & \ f(x) := \frac{1}{N^+} \sum_{i=1}^{N^+} \phi(x^{\top} a_i^+) , \\ 
	\text{s.t. } & \ c(x) := \frac{1}{N^-} \sum_{j=1}^{N^-} \phi(-x^{\top} a_j^-) - \tau_{FP} \le 0,
	\end{aligned}
\end{equation}
where $\{a_i^+\}_{i=1}^{N^+}$ and $\{a_{j}^-\}_{j=1}^{N^-}$ denotes the positive-class samples and negative-class samples in the training data set, respectively. The parameter $\tau_{FP}$ specifies the maximum allowable false-positive error rate. In \eqref{eq:np-clas}, we set $\phi(\cdot)$ to be the sigmoid function: $\phi(u) = (1+\exp(u))^{-1}$. Similar to the approach in \Cref{subsec:num_fair}, we reformulate the inequality constraint in \eqref{eq:np-clas}
as the equality constraint $c(x) + s = 0$, with $s \geq 0$, such that \eqref{eq:np-clas} is of the form \eqref{eq:gen_formula}, where $\cX=\{x \in \RR^d, s \in \RR_+\}$, and the random variables $\zeta$ and $\xi$ in \eqref{eq:gen_formula} correspond to uniformly sampling from the $N^+$ positive-class samples and $N^-$ negative-class samples, respectively. We use two datasets to %
evaluate the performance of our algorithm on this problem:

\begin{itemize}
\item \emph{spambase \cite{dua_uci_}:} $d=57$, $(N^+,N^-)=(1813,2788)$, and $\tau_{FP}=0.2$.

\item \emph{madelon \cite{guyon_result_2004}:} $d=500$, $(N^+,N^-)=(1300,1300)$, and $\tau_{FP}=0.4$.

\end{itemize}

Following \cite{yan_adaptive_2022}, we preprocess each dataset by first standardizing each feature to zero mean and unit variance, and subsequently scaling each feature vector to unit Euclidean norm. %
The tolerance was set to $\epsilon = \epsilong = \epsilonH = 0.01$ in all tests on the \emph{spambase} and \emph{madelon} problem settings. The method is set to terminate if the first-order stationarity, feasibility, and a lower bound on second-order stationarity conditions (see \eqref{eq:eps-kkt}) of the equality-constrained reformulation of \eqref{eq:np-clas} are all satisfied to within $\epsilon$. For all subroutines, at the $k$-th outer iteration of MESCAL, we set $\alpha_k = 1.1^{k+1}$. %
All methods are initialized at $x_{{-1}} = \boldsymbol{0}$. 

Figure \ref{fig:np_3_abs_cons} plots the second-order stationarity error, first-order stationarity error, and feasibility error against the total work. As in the fairness-constrained problem, SPIDER requires lower computational work than SG-HV-NC and Natasha2, requiring approximately one order of magnitude less work for the \emph{madelon} dataset and two orders of magnitude less work for the \emph{spambase} dataset. We attribute this difference, as discussed above, to the relatively infrequent negative curvature searches performed by SPIDER.

\begin{figure}[H]
    \centering
    \includegraphics[width=0.9\linewidth]{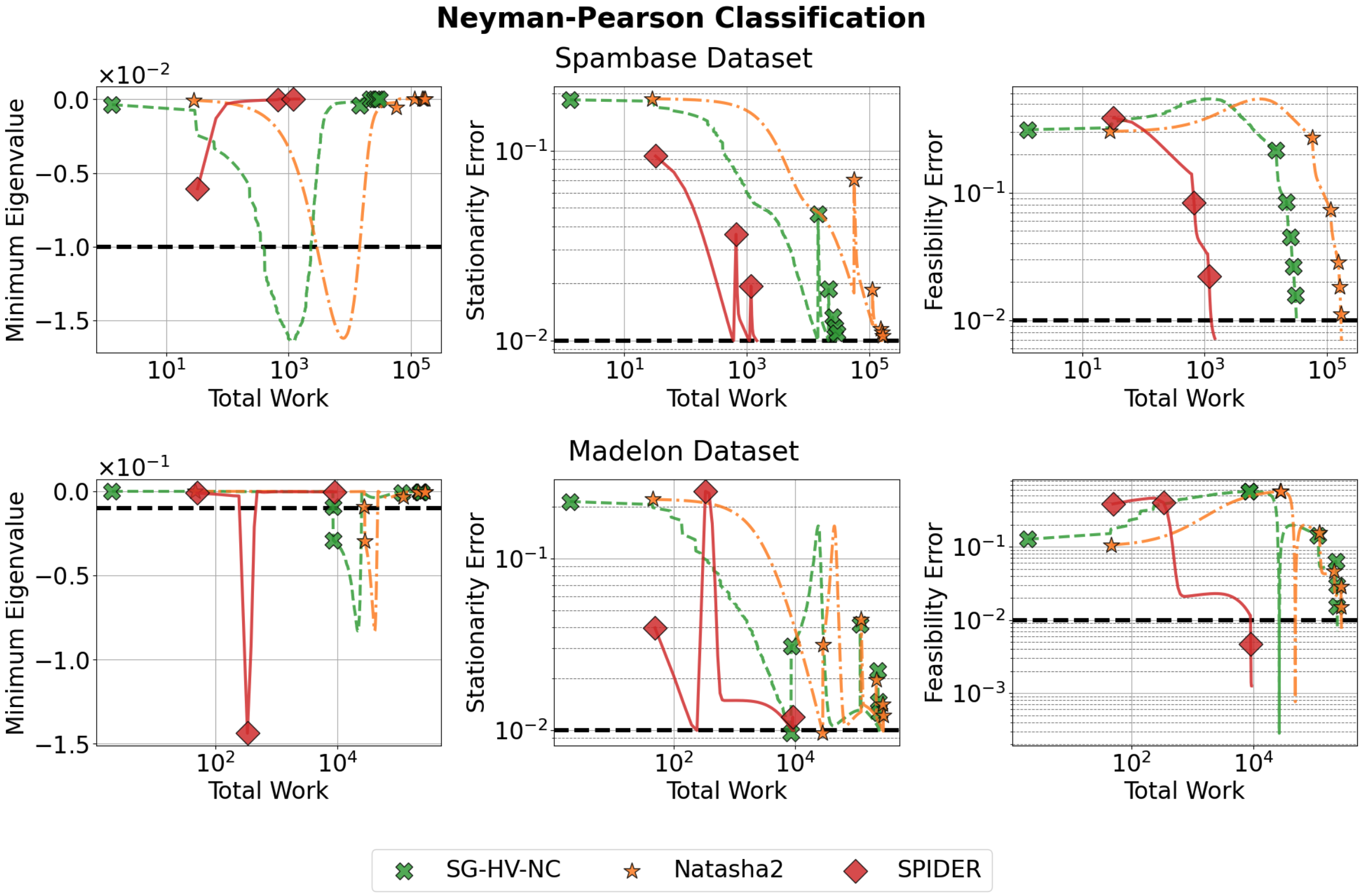}
    \caption{Results of %
    Algorithm \ref{alg:mescal-overview} on the Neyman-Pearson classification problem \eqref{eq:np-clas} utilizing the \emph{spambase} \cite{dua_uci_} and the \emph{madelon} \cite{guyon_result_2004} dataset. From left to right, the plots report the minimum eigenvalue of the Hessian of the augmented Lagrangian, the first-order stationarity error \eqref{eq:eps-kkt}$\mathrm{(i)}$, and the feasibility error \eqref{eq:eps-kkt}$\mathrm{(ii)}$. Results are shown for the three primal subproblem solvers SG-HV-NC, Natasha2, and SPIDER. The $x$-axis represents the total computational work, measured by the number of stochastic augmented Lagrangian gradient and Hessian vector product evaluations. Markers indicate the beginning of each MESCAL iteration for the corresponding solver. All solvers use the same initial point; the initial point is not shown because the $x$-axis is on a logarithmic scale.}
    
    \label{fig:np_3_abs_cons}
\end{figure}

Despite the difference in total work, the error trajectories remain comparable across the three solvers. For \emph{spambase} dataset, the feasibility error consistently decreases, while the second-order stationarity error initially increases before decreasing, and the first-order stationarity error increases after each dual variable update. SG-HV-NC and Natasha2 also require more outer iterations than SPIDER. For \emph{madelon}, all three solvers exhibit closely matching error trajectories. After the first iteration, the iterates enter region with negative curvature, reflected by a large second-order stationarity error, along with large first-order stationarity and feasibility errors. Subsequently, all three errors decrease rapidly before gradually converging to the prescribed tolerances. The nonmonotonic behavior of the errors is not unexpected due to the interplay among the first-order gradient step, second-order negative curvature step, and dual variable update step.

%% file: 6-conclusion.tex
\section{Final Remarks} \label{sec:conclusion}

We developed and analyzed MESCAL, an augmented Lagrangian framework for solving stochastic nonconvex optimization problems with expectation constraints over a closed and convex set. We proposed second-order inexactness conditions in expectation and with prescribed probability for solving the primal subproblems inexactly and established corresponding iteration complexity results. Under reasonable assumptions on the theoretical properties of a general primal solver, we established a sample complexity result of $\cO((\epsilong)^{-5} + (\epsilong)^{-4}(\epsilonH)^{-5})$ for obtaining an $(\epsilong, \epsilonH)$-accurate second-order stationary point satisfying \eqref{eq:opt_cond_exp} or \eqref{eq:opt_cond_prob}. We further incorporated three existing stochastic second-order subproblem solvers into the framework under the restriction $\cX=\RR^d$ and established corresponding deterministic sample complexity results to achieve an $(\epsilong,\epsilonH)$-accurate $p$-probabilistic second-order stationary point, with sample complexity bounds of the form $\tilde{\cO}\lt((\epsilong)^{-5}(\epsilonH)^{-5} + (\epsilong)^{-6}(\epsilonH)^{-2}\rt)$ or $\tilde{\cO}\lt((\epsilong)^{-5}(\epsilonH)^{-5} + (\epsilong)^{-7}(\epsilonH)^{-1}\rt)$ depending on the chosen primal solver. The corresponding probabilistic bounds improve the dependence on $\epsilong$ in the solver-specific terms by a factor of $\cO((\epsilong)^{-1})$. To the best of our knowledge, sample complexity guarantees for obtaining approximate second-order stationary points in nonconvex expectation-constrained optimization have not been established previously. Our numerical experiments on two nonconvex machine learning problems further illustrate the practical performance of MESCAL with the considered subproblem solvers.

Our analysis also highlights several open problems in stochastic nonconvex optimization that are directly relevant to extending the scope of the proposed framework. Based on our literature review, we are unaware of existing methods and/or complexity results for the following settings:
\begin{itemize}
\item Second-order solvers for expectation-constrained optimization problems over general closed and convex constraint sets.
\item Complexity results for obtaining second-order stationary points in expectation for unconstrained or convex-constrained stochastic optimization. In particular, Berahas et al. \cite{berahas_exploiting_2026} provide stationarity guarantees for obtaining a second-order stationary point in expectation for unconstrained optimization, but do not establish corresponding sample complexity results.
\item Algorithms for computing negative curvature directions that lie within a convex constraint set. In particular, we are unaware of a projected variant of Oja's algorithm for identifying a negative curvature direction subject to convex constraints.
\item Single-loop second-order solvers for expectation-constrained optimization problems. Developing such methods is nontrivial, as discussed {in} 
{\cite{alacaoglu_complexity_2024}}.
\end{itemize}
Developing methods for these problems would directly expand the class of subproblem solvers that can be incorporated into the proposed framework and could lead to stronger theoretical guarantees. These directions provide several avenues for future work.

\section*{Acknowledgments}
The authors thank Zichong Li for helpful initial discussions that stimulated this work.

%% file: 7-appendix.tex
\section{Additional Results} \label{sec:appendix}

\begin{lemma} \label{lemm:proj-bound}
    Suppose $\cX$ is a closed and convex set. For any vectors $a$ and $b$, it holds that
    $$\|\mathrm{proj}_\cX(a + b)\| \leq \| \mathrm{proj}_\cX(a) \| + \| b \|.$$
\end{lemma}
\begin{proof}
    By the triangle inequality, and the non-expansive property of the projection operator $\mathrm{proj}_\cX$, we have
    \begin{align*}
        \|\mathrm{proj}_\cX(a + b)\| &\leq \|\mathrm{proj}_\cX(a + b) - \mathrm{proj}_\cX(a)\| + \|\mathrm{proj}_\cX(a)\| \\
        &\leq \|(a + b) - a\| + \|\mathrm{proj}_\cX(a)\| = \|b\| + \|\mathrm{proj}_\cX(a)\|.
    \end{align*}
\end{proof}

\begin{lemma} \label{lemm:log-prob-term}
    Let $p \in (0,1)$ and $r \in \RR$ be constants. As $\epsilon \to 0^+$,
    \begin{align*}
        \lt(\log \lt( \frac{1}{1 - p^{1 / \log(1 / \epsilon)}} \rt) \rt)^r = \Theta\lt((\log\log(1/\epsilon))^r\rt) =
        \widetilde{\cO}(1).
    \end{align*}
\end{lemma}

\begin{proof}
    Let $x := \log(1/\epsilon)$. As $\epsilon \to 0^+$, we have $x \to \infty$. Since $p^{1/x} = \exp\lt(-\frac{\ln(1/p)}{x}\rt)$, a first-order Taylor expansion gives
    \begin{align*}
        1-p^{1/x} = \frac{\ln(1/p)}{x}+\cO(x^{-2})
        = \frac{\ln(1/p)}{x} \lt(1+\cO(x^{-1})\rt).
    \end{align*}
    Hence,
    \begin{align*}
        \log\lt(\frac{1}{1-p^{1/x}}\rt) 
        = \log x-\log\ln(1/p) -\log\lt(1+\cO(x^{-1})\rt) 
        = \log x+\cO(1).
    \end{align*}
    Writing $U \coloneqq \log x$, we have $U \to \infty$, and thus $\log\lt(\frac{1}{1-p^{1/x}}\rt) = U \lt( 1 + \cO(U^{-1}) \rt)$. As $r$ is constant, $\lt( \log\lt( \frac{1}{1 -p^{1/x}} \rt) \rt)^r = U^r \lt( 1 + \cO(U^{-1}) \rt)^r = U^r \lt( 1 + \cO (U^{-1}) \rt).$
    
    Substituting $U=\log\log(1/\epsilon)$ proves
    \begin{align*}
        \lt(\log \lt( \frac{1}{1-p^{1/\log(1/\epsilon)}} \rt) \rt)^r =
        \Theta\lt((\log\log(1/\epsilon))^r\rt) = \widetilde{\cO}(1).
    \end{align*}
\end{proof}

\begin{lemma} \label{lemm:grad-estimator-variance}
For $x \in \cX$, let $\{ \zetaone_i \}_{i=1}^n$ and $\{ \zetatwo_i \}_{i=1}^n$ be two independent samples, each
consisting of i.i.d. draws from $\cP_{\zeta}$.
Define, for $j \in \{1,2, \cdots, m\}$,
\[
    X^{(j)}_i := \nabla C^{(j)}(x,\zetaone_i) \in \RR^d,
    \qquad
    Y^{(j)}_i := C^{(j)} (x,\zetatwo_i) \in \RR.
\]
Assume $n \geq 2$,
$\EE[\|X^{(j)}_i\|^2] < \infty$,
$\EE[(Y^{(j)}_i)^2] < \infty$.
Then,
\[
    Var \lt(
        \frac{1}{n}\sum_{i=1}^n X^{(j)}_iY^{(j)}_i
    \rt)
    \geq
    Var \lt(
        \lt(\frac{1}{n}\sum_{i=1}^n X^{(j)}_i\rt)
        \lt(\frac{1}{n}\sum_{i=1}^n Y^{(j)}_i\rt)
    \rt) \quad \forall j \in \{1, 2, \cdots, m\}.
\]
\end{lemma}

\begin{proof}
Consider
\[
    \mu^{(j)}_X := \EE[X^{(j)}_i],
    \qquad
    \mu^{(j)}_Y := \EE[Y^{(j)}_i],
    \qquad
    v^{(j)}_X := Var(X^{(j)}_i),
    \qquad
    v^{(j)}_Y := Var(Y^{(j)}_i) \quad \forall i, j.
\]
For any independent random vector $X$ and scalar $Y$
with finite second moments, independence gives
\begin{align}
    Var(XY)
    &= \EE\bigl[\|XY\|^2\bigr]
       - \|\EE[XY]\|^2 \notag \\
    &= \EE[Y^2]\,\EE\bigl[\|X\|^2\bigr]
       - (\EE[Y])^2\|\EE[X]\|^2 \notag \\
    &= Var(X)Var(Y)
       + Var(X)(\EE[Y])^2
       + Var(Y)\|\EE[X]\|^2. \label{eq:var-product}
\end{align}
Since the vectors $X^{(j)}_iY^{(j)}_i$ are i.i.d., it follows from \eqref{eq:var-product} and linearity of variance for independent random variables,
\[
    Var\lt(
        \frac{1}{n}\sum_{i=1}^n X^{(j)}_iY^{(j)}_i
    \rt)
    =
    \frac{v^{(j)}_Xv^{(j)}_Y + v^{(j)}_X(\mu^{(j)}_Y)^2 + v^{(j)}_Y\|\mu^{(j)}_X\|^2}{n}.
\]

Now, let
\[
    \bar{X}^{(j)} := \frac{1}{n}\sum_{i=1}^n X^{(j)}_i,
    \qquad
    \bar{Y}^{(j)} := \frac{1}{n}\sum_{i=1}^n Y^{(j)}_i.
\]
The independence of the two samples implies that
$\bar{X}^{(j)}$ and $\bar{Y}^{(j)}$ are independent. Moreover,
\[
    \EE[\bar{X}^{(j)}] = \mu^{(j)}_X,
    \qquad
    \EE[\bar{Y}^{(j)}] = \mu^{(j)}_Y,
    \qquad
    Var(\bar{X}^{(j)}) = \frac{v^{(j)}_X}{n},
    \qquad
    Var(\bar{Y}^{(j)}) = \frac{v^{(j)}_Y}{n}.
\]
Applying the product-variance identity from \eqref{eq:var-product} yields
\[
    Var(\bar{X}^{(j)}\bar{Y}^{(j)})
    =
    \frac{v^{(j)}_Xv^{(j)}_Y}{n^2}
    + \frac{v^{(j)}_X(\mu^{(j)}_Y)^2}{n}
    + \frac{v^{(j)}_Y\|\mu^{(j)}_X\|^2}{n}.
\]
Consequently,
\[
    Var\lt(
        \frac{1}{n}\sum_{i=1}^n X^{(j)}_iY^{(j)}_i
    \rt)
    -
    Var(\bar{X}^{(j)}\bar{Y}^{(j)})
    =
    \frac{n-1}{n^2}v^{(j)}_Xv^{(j)}_Y
    \geq 0,
\]
where the inequality follows from $n \geq 2$.
\end{proof}

\begin{lemma} \label{lemm:epsilong-overline-equal}
Let $\epsilon_1,\epsilon_2,\epsilon_3\in(0,1)$ and let
$r,r_1,r_2,r_3>0$ be constants. Suppose
\[
\epsilon_1
=
\Theta\!\left(
\epsilon_2
\log\!\left(
\epsilon_1^{-r_1}
\epsilon_2^{-r_2}
\epsilon_3^{-r_3}
\right)
\right).
\]
Furthermore, if there exists a constant $q>0$ such that $\log(1/\epsilon_1) \geq \log^q(1/\epsilon_3)$. Then
\[
\tilde{\cO}(\epsilon_1^{-r})
=
\tilde{\cO}(\epsilon_2^{-r}).
\]
\end{lemma}

\begin{proof}
Set $a=\log(1/\epsilon_1)$, $b=\log(1/\epsilon_2)$,
$d=\log(1/\epsilon_3)$, and $r_q=\max\{1,1/q\}$.
Since $d\le a^{1/q}$, the assumed relation gives
\[
e^{b-a}=\Theta(r_1a+r_2b+r_3d)= \cO(a^{r_q}+b).
\]
Taking logarithms, for sufficiently large $a$,
\[
b\le a+{r_q}\log a+\log(1+b)+ \cO(1).
\]
Using $\log(1+b)\le b/2+ \cO(1)$ yields $b= \cO(a)$.
Consequently,

\begin{equation}
    \frac{\epsilon_1}{\epsilon_2} =\Theta(r_1a+r_2b+r_3d)\le \cO(a^{r_q}). \label{eq:epsilon-upper-bound}
\end{equation}

The assumption \(\epsilon_1/\epsilon_2=\Theta(r_1 a + r_2 b + r_3 d)\) and the fact that for some constant $t'>0$, $t' r_1 a \leq t' (r_1a+r_2b+r_3d)$ means
\begin{align}
    \Omega(a) = \epsilon_1 /\epsilon_2. \label{eq:epsilon-lower-bound}
\end{align}

Combining \eqref{eq:epsilon-upper-bound} and \eqref{eq:epsilon-lower-bound} means there exist constants $t_{\min}, t_{\max} > 0$ such that $t_{\min} a \leq \tfrac{\epsilon_1}{\epsilon_2} \leq t_{\max} a^{r_q}$. As $\epsilon_1 \to 0$, we have $a = \log(1/\epsilon_1) \to \infty$, so eventually \(t_{\min}a \ge 1\). Raising to the power \(r>0\) gives
\begin{align*}
    1 \leq \tfrac{\epsilon_1^r}{\epsilon_2^r} \leq t_{\max}^r a^{{r_q}r}.
\end{align*}

Thus, for sufficiently small $\epsilon_1$, multiplying by $\epsilon_1^{-r}$ and substituting the value of $a$ gives,
\[
\epsilon_1^{-r}\le \epsilon_2^{-r}
=\cO \left(\epsilon_1^{-r}
[\log(1/\epsilon_1)]^{{r_q}r}\right).
\]

Absorbing polylogarithmic factors proves
$\tilde{\cO}(\epsilon_1^{-r})
=\tilde{\cO}(\epsilon_2^{-r})$.
\end{proof}

\subsection{Proof of \Cref{lemm:lip-c}} \label{sec:appendix_lemm:lip-c}
\begin{proof}%
    From \Cref{assn:abs-true-bounds} and \Cref{assn:var-bounds}, we have
    \begin{align*}
        & \|\nabla F(x,\xi)\| \leq \| \nabla f(x) \| + \|\nabla F(x,\xi) - \nabla f(x) \| \leq  \kappa_{\nabla f} + \sigma_{\nabla f}, \quad \cP_\xi\text{-a.s.}
    \end{align*}
    Similarly, we can prove
    \begin{align*}
        & \|C(x, \zeta)\| \leq \|c(x)\| + \|C(x, \zeta) - c(x)\| \leq \kappa_c + \sigma_c, \quad \cP_\xi\text{-a.s.}, \\
        & \|\nabla C(x, \zeta)\| \leq \|\nabla c(x)\| + \|\nabla C(x, \zeta) - \nabla c(x)\| \leq \kappa_{\nabla c} + \sigma_{\nabla c}, \quad \cP_\xi\text{-a.s.}, \text{ and} \\
        & \|\nabla^2 C(x, \zeta)\| \leq \|\nabla^2 c(x)\| + \|\nabla^2 C(x, \zeta) - \nabla^2 c(x)\| \leq \kappa_{\nabla^2 c} + \sigma_{\nabla^2 c}, \quad \cP_\xi\text{-a.s.}
    \end{align*}
    By the mean value theorem, along the line segment between any two points $x, y \in \cX$ yields some $z$ such that $\cP_\zeta$-almost surely,
    \begin{align*}
         |C^{(j)}(x) - C^{(j)}(y)| = |\nabla C^{(j)}(z)^T (x - y)| \le \|\nabla C^{(j)}(z)\| \|x - y\| \le (\kappa_{\nabla c} + \sigma_{\nabla c}) \|x - y\|.
    \end{align*}

\end{proof}

\subsection{Proof of \Cref{lemm:iter-sample-complexity-sgd-ncs}}\label{sec:appendix_lemm:sgd}
\begin{proof}%
    Iterate $x_k$ satisfying probabilistic conditions~\eqref{cond:prob_primal_prob} with probability at least $p_k = 5/8$ is a direct consequence of \cite[Theorem 5]{arjevani_secondorder_2020}. First, we establish the conditions required for 1) the stochastic maps of  $\cL_{\alpha_k}(\cdot, \lambda_k)$ in \eqref{eq:sample_gradient_estimator} and \eqref{eq:sample_hessian_estimator} to belong to \cite[Oracle $(\mathsf{O}_F^2, P_z) \in \overline{\cO}(F, \sigma_1, \bar{\sigma}_2)$ for $F \in \cF_2 (\Delta, L_1, L_2)$]{arjevani_secondorder_2020}, and 2) for convergence of \Cref{alg:sgd-ncs}, as defined in \cite[Theorem 5]{arjevani_secondorder_2020}.

    \begin{enumerate}[label=\textbf{Condition \arabic*:}, leftmargin=*]
        \item From \eqref{eq:opt_gap_general} in  \Cref{lemm:1-lower-bound-finite}, it holds that $\cL_{\alpha_k}(x_k, \lambda_k) - \cL_{\alpha_k}(x_k^*, \lambda_k) = \Delta_k \leq \Delta^*_k < \infty$ for all $k \in \NN_0$, where $x_k^* \in \min_{x \in \cX} \cL_{\alpha_k}(x, \lambda_k)$. This corresponds to the assumption provided in \cite[Section 2 (Page 6)]{arjevani_secondorder_2020}.
        
        \item From \Cref{lemm:2-lipschitz-grad-aug-lag}, it holds that at each iteration $k \in \NN_0$, the function $\nabla \bar{\cL}(x, \lambda_k; \theta)$ is $L^{(g)}_{\cL, k}$-Lipschitz continuous on $\cX$ almost surely, i.e. $\|\nabla \bar{\cL}_{\alpha_k}(x, \lambda_k; \theta) - \nabla \bar{\cL}_{\alpha_k}(y, \lambda_k; \theta)\| \leq L^{(g)}_{\cL, k} \|x - y\|$, for all $k \in \NN_0$, for all $x, y \in \cX, \cP_\theta$-almost surely By Jensen’s inequality, \Cref{lemm:2-lipschitz-grad-aug-lag} directly implies the required condition $\|\nabla \cL_{\alpha_k}(x, \lambda_k) - \nabla \cL_{\alpha_k}(y, \lambda_k)\| \leq L^{(g)}_{\cL, k} \|x - y\|$ for all $k \in \NN_0$. This corresponds to the assumption provided in \cite[Section 2 (Page 6)]{arjevani_secondorder_2020}.

        \item From \Cref{lemm:3-lipschitz-hess-aug-lag}, it holds that at each iteration $k \in \NN_0$, the function $\nabla^2 \bar{\cL}(x, \lambda_k; \theta)$ is $L^{(H)}_{\cL, k}$-Lipschitz continuous on $\cX$ almost surely, i.e. $\|\nabla^2 \bar{\cL}_{\alpha_k}(x, \lambda_k; \theta) - \nabla^2 \bar{\cL}_{\alpha_k}(y, \lambda_k; \theta)\| \leq L^{(H)}_{\cL, k} \|x - y\|$, for all $k \in \NN_0$ for all $x, y \in \cX, \cP_\theta$-almost surely. By Jensen’s inequality, \Cref{lemm:3-lipschitz-hess-aug-lag} directly implies the required condition $\|\nabla^2 \cL_{\alpha_k}(x, \lambda_k) - \nabla^2 \cL_{\alpha_k}(y, \lambda_k)\| \leq L^{(H)}_{\cL, k} \|x - y\|$ for all $k \in \NN_0$. This corresponds to the assumption provided in \cite[Section 2 (Page 6)]{arjevani_secondorder_2020}.
        
        \item From \Cref{lemm:4-finite-bounded-variance}, it holds that at each iteration $k \in \NN_0$, the stochastic gradient estimator $\nabla \bar{\cL}_{\alpha_k}(x, \lambda_k; \theta)$ satisfy uniform error bounds $\cP_\theta$-almost surely, i.e. $\|\nabla \bar{\cL}_{\alpha_k}(x, \lambda_k; \theta) - \nabla \cL_{\alpha_k}(x, \lambda_k)\|^2 \leq (\sigma^{(g)}_{\cL, k})^2$ for all $k \in \NN_0$, for all $x \in \cX, \cP_\theta$-almost surely. By Jensen’s inequality, \Cref{lemm:4-finite-bounded-variance} directly implies the required condition $\EE_\theta [\|\nabla \bar{\cL}_{\alpha_k}(x, \lambda_k; \theta) - \nabla \cL_{\alpha_k}(x, \lambda_k)\|^2] \leq (\sigma^{(g)}_{\cL, k})^2$ for all $k \in \NN_0$. This corresponds to the assumption provided in \cite[Section 2 (Page 6)]{arjevani_secondorder_2020}.
        
        \item From \Cref{lemm:5-finite-hess-variance}, it holds that at each iteration $k \in \NN_0$, the stochastic Hessian estimator $\nabla^2 \bar{\cL}_{\alpha_k}(x, \lambda_k; \theta)$ satisfy uniform error bounds $\cP_\theta$-almost surely, i.e. $\|\nabla^2 \bar{\cL}_{\alpha_k}(x, \lambda_k; \theta) - \nabla^2 \cL_{\alpha_k}(x, \lambda_k)\|^2 \leq (\sigma^{(H)}_{\cL, k})^2$ for all $k \in \NN_0$, for all $x \in \cX, \cP_\theta\text{-a.s.}$. This corresponds to the assumption provided in \cite[Section 4.1 (Page 11)]{arjevani_secondorder_2020}.

        \item Finally, $(\epsilong_k, \epsilonH_k)$ satisfies the constraints on \cite[$(\epsilon, \gamma)$]{arjevani_secondorder_2020} as required in \cite[Section 2 (Page 6)]{arjevani_secondorder_2020}.
    \end{enumerate}

    By \textbf{Conditions 1-6}, \cite[Theorem 5]{arjevani_secondorder_2020} holds. Therefore, $x_k$ satisfies probabilistic conditions~\eqref{cond:prob_primal_prob}. Furthermore, this establishes the work complexity by \cite[Theorem 5]{arjevani_secondorder_2020} as per \Cref{def:primal_work_at_k} to be
    \begin{align*}
        \cW^{(p)}_{SG, k} = \tilde{\cO} \lt( 
        \tfrac{\Delta_k \sigma^{(g)}_{\cL, k} \sigma^{(H)}_{\cL, k}}{(\epsilong_k)^3} + 
        \tfrac{\Delta_k L^{(H)}_{\cL, k} \sigma^{(g)}_{\cL, k} \sigma^{(H)}_{\cL, k}}{(\epsilonH_k)^2(\epsilong_k)^2} + 
        \tfrac{\Delta_k (L^{(H)}_{\cL, k})^2 (\sigma^{(H)}_{\cL, k} + L^{(g)}_{\cL, k})^2}{(\epsilonH_k)^5} + 
        \tfrac{\Delta_k L^{(g)}_{\cL, k}}{(\epsilong_k)^2}
        \rt).
    \end{align*}
    By the argument in \Cref{lemm:gen-complexity-solver}, this is the work complexity of the primal solver at each iteration $k \in \NN_0$.
\end{proof}

\subsection{Proof of \Cref{lemm:iter-sample-complexity-natasha2}}\label{sec:appendix_lemm:nat2}

\begin{proof}
    Iterate $x_k$ satisfying probabilistic conditions~\eqref{cond:prob_primal_prob} with probability at least $p_k = 2/3$ is a direct consequence of \cite[Theorem 2]{allen-zhu_natasha_2018}. First, we establish the conditions required for 1) the stochastic maps of $\cL_{\alpha_k}(\cdot, \lambda_k)$ in \eqref{eq:sample_gradient_estimator} and \eqref{eq:sample_hessian_estimator} to belong to \cite[function class $f_i$]{allen-zhu_natasha_2018}, and 2) for convergence of \Cref{alg:natasha2}, as defined in \cite[Theorem 2]{allen-zhu_natasha_2018}.

    \begin{enumerate}[label=\textbf{Condition \arabic*:}, leftmargin=*]
        \item From \eqref{eq:opt_gap_general} in \Cref{lemm:1-lower-bound-finite}, it holds that $\cL_{\alpha_k}(x_k, \lambda_k) - \cL_{\alpha_k}(x_k^*, \lambda_k) = \Delta_k \leq \Delta^*_k < \infty$ for all $k \in \NN_0$, where $x_k^* \in \min_{x \in \cX} \cL_{\alpha_k}(x, \lambda_k)$. This corresponds to the assumption stated in \cite[Theorem 2]{allen-zhu_natasha_2018}.
        
        \item From \Cref{lemm:2-lipschitz-grad-aug-lag}, it holds that at each iteration $k \in \NN_0$, the function $\nabla \bar{\cL}(x, \lambda_k; \theta)$ is $L^{(g)}_{\cL, k}$-Lipschitz continuous on $\cX$ almost surely, i.e. $\|\nabla \bar{\cL}_{\alpha_k}(x, \lambda_k; \theta) - \nabla \bar{\cL}_{\alpha_k}(y, \lambda_k; \theta)\| \leq L^{(g)}_{\cL, k} \|x - y\|$, for all $k \in \NN_0$, for all $x, y \in \cX, \cP_\theta$-almost surely. This corresponds to \cite[Assumption (A2)]{allen-zhu_natasha_2018}.

        \item From \Cref{lemm:3-lipschitz-hess-aug-lag}, it holds that at each iteration $k \in \NN_0$, the function $\nabla^2 \bar{\cL}(x, \lambda_k; \theta)$ is $L^{(H)}_{\cL, k}$-Lipschitz continuous on $\cX$ almost surely, i.e. $\|\nabla^2 \bar{\cL}_{\alpha_k}(x, \lambda_k; \theta) - \nabla^2 \bar{\cL}_{\alpha_k}(y, \lambda_k; \theta)\| \leq L^{(H)}_{\cL, k} \|x - y\|$, for all $k \in \NN_0$, for all $x, y \in \cX, \cP_\theta$-almost surely. \Cref{lemm:3-lipschitz-hess-aug-lag} directly implies the required condition $\|\nabla^2 \cL_{\alpha_k}(x, \lambda_k) - \nabla^2 \cL_{\alpha_k}(y, \lambda_k)\| \leq L^{(H)}_{\cL, k} \|x - y\|$ for all $k \in \NN_0$. This corresponds to \cite[Assumption (A4)]{allen-zhu_natasha_2018}.
        
        \item From \Cref{lemm:4-finite-bounded-variance}, it holds that at each iteration $k \in \NN_0$, the stochastic gradient estimator $\nabla \bar{\cL}_{\alpha_k}(x, \lambda_k; \theta)$ satisfy uniform error bounds $\cP_\theta$-almost surely, i.e. $\|\nabla \bar{\cL}_{\alpha_k}(x, \lambda_k; \theta) - \nabla \cL_{\alpha_k}(x, \lambda_k)\|^2 \leq (\sigma^{(g)}_{\cL, k})^2$, for all $k \in \NN_0$, for all $x \in \cX, \cP_\theta\text{-a.s.}$. \Cref{lemm:4-finite-bounded-variance} directly implies the required condition $\EE_\theta [\|\nabla \bar{\cL}_{\alpha_k}(x, \lambda_k; \theta) - \nabla \cL_{\alpha_k}(x, \lambda_k)\|^2] \leq (\sigma^{(g)}_{\cL, k})^2$, for all $k \in \NN_0$. This corresponds to \cite[Assumption (A1)]{allen-zhu_natasha_2018}.

        \item Finally, $(\epsilong_k, \epsilonH_k)$ satisfies the constraints on \cite[$(\varepsilon, 3\delta)$]{allen-zhu_natasha_2018} in \cite[Theorem 2]{allen-zhu_natasha_2018}.
    \end{enumerate}

    By \textbf{Conditions 1-5}, \cite[Theorem 2]{allen-zhu_natasha_2018} holds. Therefore, $x_k$ satisfies probabilistic conditions~\eqref{cond:prob_primal_prob}. Furthermore, this establishes the work complexity by \cite[Theorem 2]{allen-zhu_natasha_2018} as per \Cref{def:primal_work_at_k} to be
    \begin{align*}
        \cW^{(p)}_{Nat2, k} = \tilde{\cO} \lt( 
        \tfrac{(\sigma^{(g)}_{\cL, k})^4}{(\epsilong_k)^2}
        + \tfrac{(L^{(H)}_{\cL, k})^2 (L^{(g)}_{\cL, k})^2 \Delta_k}{(\epsilonH_k)^5}
        + \tfrac{L^{(H)}_{\cL, k} \Delta_k (L^{(g)}_{\cL, k})^2}{\epsilong_k (\epsilonH_k)^3}
        + \tfrac{L^{(H)}_{\cL, k} \Delta_k (\sigma^{(g)}_{\cL, k})^2}{(\epsilong_k)^3 \epsilonH_k}
        + \tfrac{(L^{(g)}_{\cL, k})^3 \Delta_k}{\epsilong_k (\epsilonH_k)^2 \sigma^{(g)}_{\cL, k}}
        \rt).
    \end{align*}
    By the argument in \Cref{lemm:gen-complexity-solver}, this is the work complexity of the primal solver at each iteration $k \in \NN_0$.
\end{proof}

\subsection{Proof of \Cref{lemm:iter-sample-complexity-spider}} \label{sec:appendix_lemm:spdr}

\begin{proof}
    Iterate $x_k$ satisfying probabilistic conditions~\eqref{cond:prob_primal_prob} with probability at least $p_k = 1/2$ is a direct consequence of \cite[Theorem 6]{fang_spider_2018}. First, we establish the conditions required for 1) the stochastic maps of $\cL_{\alpha_k}(\cdot, \lambda_k)$ in \eqref{eq:sample_gradient_estimator} and \eqref{eq:sample_hessian_estimator} to belong to \cite[function class $f_i$]{fang_spider_2018}, and 2) for convergence of \Cref{alg:spider_sfo_plus}, as defined in \cite[Theorem 6]{fang_spider_2018}.

    \begin{enumerate}[label=\textbf{Condition \arabic*:}, leftmargin=*]
        \item From \eqref{eq:opt_gap_general} in \Cref{lemm:1-lower-bound-finite}, it holds that $\cL_{\alpha_k}(x_k, \lambda_k) - \cL_{\alpha_k}(x_k^*, \lambda_k) = \Delta_k \leq \Delta^*_k < \infty$ for all $k \in \NN_0$, where $x_k^* \in \min_{x \in \cX} \cL_{\alpha_k}(x, \lambda_k)$. This corresponds to \cite[Assumption 1(i)]{fang_spider_2018}.
        
        \item From \Cref{lemm:2-lipschitz-grad-aug-lag}, it holds that at each iteration $k \in \NN_0$, the function $\nabla \bar{\cL}(x, \lambda_k; \theta)$ is $L^{(g)}_{\cL, k}$-Lipschitz continuous on $\cX$ almost surely, i.e. $\|\nabla \bar{\cL}_{\alpha_k}(x, \lambda_k; \theta) - \nabla \bar{\cL}_{\alpha_k}(y, \lambda_k; \theta)\| \leq L^{(g)}_{\cL, k} \|x - y\|$, for all $k \in \NN_0$, for all $x, y \in \cX, \cP_\theta$-almost surely. This corresponds to \cite[Assumption 2(ii')]{fang_spider_2018}.

        \item From \Cref{lemm:3-lipschitz-hess-aug-lag}, it holds that at each iteration $k \in \NN_0$, the function $\nabla^2 \bar{\cL}(x, \lambda_k; \theta)$ is $L^{(H)}_{\cL, k}$-Lipschitz continuous on $\cX$ almost surely, i.e. $\|\nabla^2 \bar{\cL}_{\alpha_k}(x, \lambda_k; \theta) - \nabla^2 \bar{\cL}_{\alpha_k}(y, \lambda_k; \theta)\| \leq L^{(H)}_{\cL, k} \|x - y\|$, for all $k \in \NN_0$, for all $x, y \in \cX, \cP_\theta$-almost surely. This corresponds to \cite[Assumption 3]{fang_spider_2018}.
        
        \item From \Cref{lemm:4-finite-bounded-variance}, it holds that at each iteration $k \in \NN_0$, the stochastic gradient estimator $\nabla \bar{\cL}_{\alpha_k}(x, \lambda_k; \theta)$ satisfy uniform error bounds $\cP_\theta$-almost surely, i.e. $\|\nabla \bar{\cL}_{\alpha_k}(x, \lambda_k; \theta) - \nabla \cL_{\alpha_k}(x, \lambda_k)\|^2 \leq (\sigma^{(g)}_{\cL, k})^2$, for all $k \in \NN_0$, for all $x \in \cX, \cP_\theta$-almost surely. This corresponds to \cite[Assumption 2(iii')]{fang_spider_2018}.

        \item Finally, primal tolerance parameter $\overline{\epsilon}^{(g)}_k$ satisfies the condition on \cite[$10 \epsilon$]{fang_spider_2018} and subproblem tolerance $\epsilonH_k$ satisfies the constraints on \cite[$3\delta$]{fang_spider_2018} in \cite[Theorem 6]{fang_spider_2018}.
    \end{enumerate}

    By \textbf{Conditions 1-5}, \cite[Theorem 6]{fang_spider_2018} holds. Therefore, $x_k$ satisfies probabilistic conditions~\eqref{cond:prob_primal_prob} Furthermore, this establishes the work complexity by \cite[Theorem 6]{fang_spider_2018} as per \Cref{def:primal_work_at_k} to be
    \begin{align*}
        \cW^{(p)}_{SPDR, k} &= \tilde{\cO} \lt( 
        \tfrac{\Delta_k L^{(g)}_{\cL, k} \sigma^{(g)}_{\cL, k}}{(\overline{\epsilon}^{(g)}_k)^3} 
        + \tfrac{\Delta_k L^{(g)}_{\cL, k} L^{(H)}_{\cL, k} \sigma^{(g)}_{\cL, k}}{(\overline{\epsilon}^{(g)}_k)^2 (\epsilonH_k)^2} 
        + \tfrac{\Delta_k (L^{(g)}_{\cL, k})^2 (L^{(H)}_{\cL, k})^2}{(\epsilonH_k)^5} 
        + \tfrac{\Delta_k (L^{(g)}_{\cL, k})^2 L^{(H)}_{\cL, k}}{\overline{\epsilon}^{(g)}_k (\epsilonH_k)^3} 
        + \tfrac{(\sigma^{(g)}_{\cL, k})^2}{(\overline{\epsilon}^{(g)}_k)^2} 
        + \tfrac{L^{(g)}_{\cL, k} \sigma^{(g)}_{\cL, k} \epsilonH_k}{(\overline{\epsilon}^{(g)}_k)^2 L^{(H)}_{\cL, k}} \rt) \\
    \end{align*}
    
    By \Cref{lemm:gen-complexity-solver}, this is the work complexity of the primal solver at each iteration $k \in \NN_0$.
\end{proof}

\subsection{Additional Remarks}

\begin{remark}  \label{rmk:proj-tan-cone}
    We summarize several key properties of the projection operator onto tangent cones that may be useful in convergence and complexity analysis. Throughout this discussion, let $\cX \subset \RR^d$ be a nonempty, closed, and convex set. The projection operator $\mathrm{proj}_{T_\cX(x)}(\cdot)$ satisfies the following properties:
    \begin{itemize}
        \item \textbf{Closedness and Convexity:} For any $x \in \cX$, the tangent cone $T_\cX(x)$ is a closed convex cone. Consequently, $\mathrm{proj}_{T_\cX(x)}$ inherits all standard properties of projection onto closed convex sets, including non-expansiveness.

        \item \textbf{Moreau Decomposition and Duality:} The tangent cone $T_\cX(x)$ and the normal cone $N_\cX(x)$ are polar to each other, i.e., $N_\cX(x) = (T_\cX(x))^\circ$. By Moreau's Decomposition Theorem \cite{moreau_decomposition_1962, rockafellar_variational_1998}, any vector $y \in \RR^d$ satisfies $y = \mathrm{proj}_{T_\cX(x)}(y) + \mathrm{proj}_{N_\cX(x)}(y)$ with $\langle \mathrm{proj}_{T_\cX(x)}(y), \mathrm{proj}_{N_\cX(x)}(y) \rangle = 0$. This yields the dual distance identity: $\|\mathrm{proj}_{T_\cX(x)}(y)\| = \mathrm{dist}(y, N_\cX(x))$, for all $y \in \RR^d$.

        \item \textbf{Norm Bound:} Since $0 \in T_\cX(x)$ for all $x \in \cX$, applying \Cref{lemm:proj-bound} directly yields a uniform norm bound: $\|\mathrm{proj}_{T_\cX(x)}(v)\| \le \|v\|$, for all $v \in \RR^d$.

        \item \textbf{Positive Homogeneity:} As a consequence of $T_\cX(x)$ being a cone, the projection operator is positively homogeneous of degree one: $\mathrm{proj}_{T_\cX(x)}(\lambda v) = \lambda \mathrm{proj}_{T_\cX(x)}(v), \quad \forall \lambda > 0, \, v \in \RR^d$.
    \end{itemize}
\end{remark}

\begin{remark} \label{rmk:proj-dual-comp-slack}
    We make the following remarks regarding the analysis of inequality constraints:
    \begin{itemize}
        \item As stated earlier, to convert an inequality {constraint} into an equality constraint, a non-negative slack variable is added to the constraint with the non-negativity constraints embedded into the projection set $\cX$. We note that the first-order stationarity error term \eqref{eq:eps-kkt}$\mathrm{(i)}$, encodes dual feasibility and complementary slackness, %
        in addition to the first-order stationarity error.
    
        Consider $x' := (x, s) \in \cX := \RR^d \times \RR_+^m$. The constraint function is updated to $c(x) + s = 0$. Now, the standard first-order stationarity conditions for the problem
        \begin{align*}
            \min_{(x, s)\in \cX}   &\quad f(x) \\
            \mathrm{s.t.}   &\quad c(x) + s = 0, \\
                            &\quad s \geq 0,
        \end{align*}
        are
        \begin{align*}
            \nabla f(x) + \nabla c(x)^\top \lambda &= 0, \\
            c(x) + s &= 0, \\
            \lambda &\geq 0, \\
            \lambda \cdot s &= 0.
        \end{align*}
        Under the projection settings, the optimality conditions are 
        \begin{align*}
            \|\mathrm{proj}_{T_{\cX}(x')}(- \nabla f(x) - \nabla c(x)^\top\lambda \| &= 0, \\
            c(x) + s &= 0.
        \end{align*}
        The first condition corresponds to
        \begin{align*}
            \|\mathrm{proj}_{T_{\cX}((x, s))}(- \nabla f(x) - \nabla c(x)^\top\lambda \| := \sqrt{\| \nabla f(x) + \nabla c(x)^\top \lambda \|^2 + \sum_{j=1}^m \|\mathrm{proj}_{T_{\RR_+}(s^{(j)})}(- \lambda^{(j)}) \|^2}.
        \end{align*}
        The first term within the square root corresponds to the first-order stationarity term. The second term within the square root can be simplified as
        \begin{align*}
            \mathrm{proj}_{T_{\RR_+}(s^{(j)})}(- \lambda^{(j)}) = 
            \lt\{ 
            \begin{matrix} 
            -\lambda^{(j)}, & s^{(j)} > 0 \\ 
            \max (- \lambda^{(j)}, 0), & s^{(j)} = 0 
            \end{matrix} \rt\}.
        \end{align*}
        Setting the norm of this term to zero exactly corresponds to the dual feasibility and complementary slackness condition.

        \item We found no algorithms for checking the minimum eigenvalue corresponding to a critical cone of vectors as required. Thus, instead of ensuring the minimum eigenvalue over the critical cone is greater than $\epsilonH$, we ensure $\sigma_{\min}(\nabla^2_{(x,s)(x, s)} \cL_\alpha((x, s), \lambda)) \geq -\epsilonH$. This allows us to use Oja's and Neon2 algorithm to search for negative curvature directions. It is not necessary to ensure this, as the negative curvature direction could correspond to a direction that is not within the critical cone. The iterate could have already converged to an $(\epsilong, \epsilonH)$-accurate point earlier, but this is a sufficient condition for second-order stationarity. Our empirical observations, within the considered problems, found $\sigma_{\min}(\nabla^2_{(x, s)(x, s)} \cL_\alpha) \approx \sigma_{\min}(\nabla^2_{xx} \cL_\alpha)$. We will consider developing algorithms to search for  eigenvectors corresponding to a closed and convex set $\cX \subset \RR^n$ in the future.
    \end{itemize}
\end{remark}